\UseRawInputEncoding  
\documentclass[11pt]{article}
\usepackage[total={7in, 8in}]{geometry}
\usepackage{graphicx}
\usepackage[margin=0.5in]{caption}
\usepackage{comment}
\usepackage{amsmath, amsthm, latexsym, amssymb, color,cite,enumerate, physics, framed}
\usepackage{caption,subcaption,verbatim, empheq, cancel, esint, bbm,multicol}
\usepackage[shortlabels]{enumitem}
\usepackage{mathtools}
\usepackage{appendix}
\newtheorem{theorem}{Theorem}[section]
\newtheorem{lemma}[theorem]{Lemma}
\newtheorem{definition}[theorem]{Definition}
\newtheorem{corollary}[theorem]{Corollary}
\newtheorem{proposition}[theorem]{Proposition}

\newtheorem{remark}{Remark}
\theoremstyle{remark}
\newtheorem{examplex}{Example}
\newenvironment{example}
  {\pushQED{\qed}\examplex}
  {\popQED\endexamplex}

\DeclareMathOperator{\sign}{sign}
\newcommand*{\myproofname}{Proof}
\newenvironment{subproof}[1][\myproofname]{\begin{proof}[#1]}{\end{proof}}
\renewcommand\Re{\operatorname{Re}}

\newcommand\Ort{\operatorname{O}}
\newcommand\End{\operatorname{End}} 
\newcommand\Gl{\operatorname{Gl}} 
\newcommand\OdR{\mbox{O}(\mathbb{R}^d)} 
\newcommand\Sym{\operatorname{Sym}}
\newcommand\Exp{\operatorname{Exp}}
\newcommand\diag{\operatorname{diag}}
\newcommand\supp{\operatorname{Supp}}
\newcommand\Spec{\operatorname{Spec}}
\renewcommand\det{\operatorname{det}}

\usepackage{hyperref}
\usepackage{tensor}
\usepackage{xcolor}
\hypersetup{
	colorlinks,
	linkcolor={black!50!black},
	citecolor={blue!50!black},
	urlcolor={blue!80!black}
}
\newcommand{\slantedslash}{\mathbin{\rotatebox[origin=c]{23}{$-$}}}

\newcommand{\docircint}[2]{%
  \ifx#1\displaystyle
    \displayrint
  \else
    \normalrint{#1}%
  \fi
}
\newcommand{\displayrint}{\displaystyle \slantedslash \mkern-18mu}
\newcommand{\normalrint}[1]{%
  \smallerc{#1}\ifx#1\textstyle\mkern-9mu\else\mkern-8.2mu\fi
}
\newcommand{\smallerc}[1]{%
  \vcenter{\hbox{$\ifx#1\textstyle\scriptstyle\else\scriptscriptstyle\fi \slantedslash $}}%
}

\newcommand\preceqdot{\mathrel{\ooalign{$\preceq$\cr
  \hidewidth\raise0.225ex\hbox{$\cdot\mkern0.5mu$}\cr}}}
\author{Evan Randles}
\title{Local Limit Theorems for Complex Functions on $\mathbb{Z}^d$}
\date{\today}
\begin{document}
\maketitle
\abstract{The local (central) limit theorem precisely describes the behavior of iterated convolution powers of a probability distribution on the $d$-dimensional integer lattice, $\mathbb{Z}^d$. Under certain mild assumptions on the distribution, the theorem says that the convolution powers are well-approximated by a single scaled Gaussian density which we call an \textit{attractor}. When such distributions are allowed to take on complex values, their convolution powers exhibit new and disparate behaviors not seen in the probabilistic setting. Following works of I. J. Schoenberg, T. N. E. Greville, P. Diaconis, and L. Saloff-Coste, the author and L. Saloff-Coste provided a complete description of local limit theorems for the class of finitely supported complex-valued functions on $\mathbb{Z}$ and the list of possible attractors includes the biharmonic heat kernel, the Airy function, and the heat kernel evaluated at purely imaginary time. For convolution powers of complex-valued functions on $\mathbb{Z}^d$, much less is known. In a previous work by the author and L. Saloff-Coste, local limit theorems were established for complex-valued functions whose Fourier transform is maximized in absolute value at so-called points of positive homogeneous type and, in that case, the resultant attractors are generalized heat kernels corresponding to a class of higher order partial differential operators. By considering the possibility that the Fourier transform can be maximized in absolute value at points of imaginary homogeneous type, a notion motivated by V. Thom\'{e}e, this article extends previous work of the author and L. Saloff-Coste to broaden the class of complex-valued functions for which it is possible to obtain local limit theorems. These local limit theorems contain attractors given by certain oscillatory integrals and their convergence is established using a generalized polar-coordinate integration formula, due to H. Bui and the author, and the Van der Corput lemma. The article also extends recent results on sup-norm type estimates of H. Bui and the author.}\\

\noindent{\small\bf Keywords:} Convolution Powers, Local Limit Theorems\\

\noindent{\small\bf Mathematics Subject Classification:} Primary 42A85; Secondary 60F99  \& 42B20

\section{Introduction}

Denote by $\ell^1(\mathbb{Z}^d)$ the set of functions $\phi:\mathbb{Z}^d\to\mathbb{C}$ for which
\begin{equation*}
\|\phi\|_1:=\sum_{x\in\mathbb{Z}^d}\abs{\phi(x)}<\infty.
\end{equation*}
For a fixed $\phi\in\ell^1(\mathbb{Z}^d)$, we define the convolution powers $\phi^{(n)}\in \ell^1(\mathbb{Z}^d)$ of $\phi$ iteratively by putting $\phi^{(1)}=\phi$ and, for each integer $n\geq 2$, 
\begin{equation*}
\phi^{(n)}(x)=\sum_{y\in\mathbb{Z}^d}\phi^{(n-1)}(x-y)\phi(y)
\end{equation*}
for $x\in\mathbb{Z}^d$. Motivated by its central importance in random walk theory \cite{spitzer_principles_1964,lawler_random_2010} and its applications to data smoothing algorithms \cite{Sch53,Gre66} and numerical solutions in partial differential equations \cite{coulombel2022generalized,Tho65,Tho69}, we are interested in the asymptotic behavior of $\phi^{(n)}(x)$ as $n\to\infty$. Following the articles \cite{DSC14}, \cite{RSC15}, and \cite{RSC17}, our goal in this article is to broaden the class of functions $\phi\in\ell^1(\mathbb{Z}^d)$ for which it is possible to obtain ``simple" pointwise descriptions of $\phi^{(n)}(x)$, for sufficiently large $n$, in the form of local limit theorems. To better understand this goal, let us first discuss some background whose origins are rooted in probability.\\

\noindent In the case that $\phi$ is a probability distribution on $\mathbb{Z}^d$, i.e., $\phi\geq 0$ and $\|\phi\|_1=\sum_{x\in\mathbb{Z}^d}\phi(x)=1$, there is a natural Markov process on $\mathbb{Z}^d$ whose $n$th-step transition kernels are given by $k_n(x,y)=\phi^{(n)}(y-x)$; we call this process the \textbf{random walk on $\mathbb{Z}^d$ driven by $\phi$}. In particular, $\phi^{(n)}(x)=k_n(0,x)$ represents the probability that the ``random walker" starting at the origin will be at position $x$ after $n$ steps. The study of random walks on $\mathbb{Z}^d$ has a long and storied history and we encourage the reader to take a look at the wonderful book of F. Spitzer \cite{spitzer_principles_1964} for an account (see also \cite{lawler_random_2010}). In the case that the random walk is aperiodic, irreducible, and of finite range, the classical local (central) limit theorem states that
\begin{equation}\label{eq:LCLTProb1}
\phi^{(n)}(x)=n^{-d/2}G_\phi\left(n^{-1/2}(x-n\alpha_\phi)\right)+o(n^{-d/2})
\end{equation}
uniformly for $x\in\mathbb{Z}^d$ where $\alpha_\phi\in\mathbb{R}^d$ denotes the mean of $\phi$ and $G_\phi$ is the generalized Gaussian given by
\begin{equation*}
G_\phi(x)=\frac{1}{(2\pi)^{d/2}\sqrt{\det(C_\phi)}}\exp\left(-\frac{x\cdot C_\phi^{-1}x}{2}\right)=\frac{1}{(2\pi)^d}\int_{\mathbb{R}^d}e^{-P_\phi(\xi)}e^{-ix\cdot\xi}\,d\xi
\end{equation*}
where $\cdot$ denotes the dot product, $C_\phi$ is the symmetric and positive definite covariance matrix of $\phi$, and $P_\phi(\xi)=\xi\cdot(C_\phi\xi)$ is its associated positive definite homogeneous second-order polynomial \cite{spitzer_principles_1964,lawler_random_2010,RSC17}. Under the additional assumption that $\phi$ is symmetric, $\alpha_\phi=0$ and so \eqref{eq:LCLTProb1} yields the two-sided estimate
\begin{equation*}
C_1n^{-d/2}\leq \phi^{(n)}(0)=k_{n}(x,x)\leq C_2 n^{-d/2}
\end{equation*}
for all $n\in\mathbb{N}_+:=\{1,2,3,\dots\}$ and $x\in\mathbb{Z}^d$; here, $C_1$ and $C_2$ are positive constants. This so-called on-diagonal estimate describes the return probabilities of the random walk and was used by G. P\'{o}lya to establish the dichotomy of recurrence/transience of simple random walk \cite{Polya21}. The hypotheses that a probability distribution $\phi$ is symmetric, aperiodic, irreducible, and of finite range can be weakened significantly and still a local limit theorem will hold. For example, if one assumes only that a probability distribution $\phi$ has finite second moments and is genuinely $d$-dimensional\footnote{In the language of F. Spitzer, $\phi$ is said to be \textbf{genuinely $d$-dimensional} if it is not supported in any affine hyperplane of $\mathbb{R}^d$ \cite{spitzer_principles_1964}.}, then
\begin{equation}\label{eq:LCLTProb2}
\phi^{(n)}(x)=n^{-d/2}\Theta(n,x)G_\phi\left(n^{-1/2}(x-n\alpha_\phi)\right)+o(n^{-d/2})
\end{equation}
uniformly for $x\in\mathbb{Z}^d$ where $\Theta(n,x)$ is a ``support function" which characterizes the periodicity of the random walk. For a proof of the local limit theorem \eqref{eq:LCLTProb2}, we refer the reader to Subsection 7.6 of \cite{RSC17}. Also, there is a rich theory for improving the error $o(n^{-d/2})$ in \eqref{eq:LCLTProb2} which is nicely presented in \cite{lawler_random_2010}.\\

\noindent In taking our discussion beyond the realm of probability, it is helpful to introduce some basic objects that play a role in the theory. For a given $\phi\in\ell^1(\mathbb{Z}^d)$, we define its Fourier transform  $\widehat{\phi}$ by the trigonometric series
\begin{equation*}
\widehat{\phi}(\xi)=\sum_{x\in\mathbb{Z}^d}\phi(x)e^{ix\cdot\xi}
\end{equation*}
which is everywhere uniformly convergent and has $|\widehat{\phi}(\xi)|\leq \|\phi\|_1$ for all $\xi\in\mathbb{R}^d$. As in \cite{BR21} and \cite{RSC17}, we shall focus on the subspace $\mathcal{S}_d$ of $\ell^1(\mathbb{Z}^d)$ consisting of those $\phi\in\ell^1(\mathbb{Z}^d)$ for which
\begin{equation*}
\|x^\beta\phi\|_1=\sum_{x\in\mathbb{Z}^d}\abs{x^\beta \phi(x)}<\infty
\end{equation*}
for each $\beta=(\beta_1,\beta_2,\dots,\beta_d)\in\mathbb{N}^d$ where $x^\beta:=(x_1)^{\beta_1}(x_2)^{\beta_2}\cdots(x_d)^{\beta_d}$ for $x=(x_1,x_2,\dots,x_d)\in\mathbb{Z}^d$. We observe that $\mathcal{S}_d$ contains all finitely supported functions in $\phi\in\ell^1(\mathbb{Z}^d)$. It is easy to see that, for each $\phi\in\mathcal{S}_d$, $\widehat{\phi}\in C^\infty(\mathbb{R}^d)$ and, in fact, $\widehat{\phi}$ is analytic whenever $\phi$ is finitely supported. As discussed in \cite{BR21,RSC17,RSC15,DSC14,coulombel2022generalized,Tho65}, the asymptotic behavior of $\phi^{(n)}$ is characterized by the local behavior of $\widehat{\phi}$ near points in $\mathbb{T}^d:=(-\pi,\pi]^d$ at which $|\widehat{\phi}|$ is maximized. The fact is seen evident through the Fourier transform identity
\begin{equation}\label{eq:FTConvolution}
\phi^{(n)}(x)=\frac{1}{(2\pi)^d}\int_{\mathbb{T}^d}\widehat{\phi}(\xi)^ne^{-ix\cdot\xi}\,d\xi
\end{equation}
which holds for all $n\in\mathbb{N}_+$ and $x\in\mathbb{Z}^d$. For simplicity, we shall assume that $\phi\in\mathcal{S}_d$ is normalized so that $\sup_{\xi}|\widehat{\phi}(\xi)|=1$ and with this we define
\begin{equation}\label{eq:DefOfOmega}
\Omega(\phi)=\left\{\xi\in\mathbb{T}^d:\abs{\widehat{\phi}(\xi)}=1\right\}.
\end{equation}
\begin{remark}
In the case that $\phi\geq 0$ is a probability distribution, the normalization $\sup|\widehat{\phi}|=1$ is automatic. In this case, recognizing $\mathbb{T}^d$ as the $d$-dimensional torus group, $\Omega(\phi)$ is a subgroup of $\mathbb{T}^d$ and the support function $\Theta(n,x)$ appearing in \eqref{eq:LCLTProb2} can be expressed in terms of the elements $\xi\in \Omega(\phi)$  \cite{RSC17}. The additional hypotheses that the random walk driven by $\phi$ is aperiodic and irreducible are equivalent to the hypothesis that $\Omega(\phi)=\{0\}$ and, in this case, $\Theta(n,x)\equiv 1$ making \eqref{eq:LCLTProb1} a special case of \eqref{eq:LCLTProb2}.
\end{remark}
\noindent For each $\xi_0\in\Omega(\phi)$, consider $\Gamma_{\xi_0}:\mathcal{U}\to\mathbb{C}$ defined by
\begin{equation}\label{eq:GammeDef}
\Gamma_{\xi_0}(\xi)=\log\left(\frac{\widehat{\phi}(\xi+\xi_0)}{\widehat{\phi}(\xi_0)}\right)
\end{equation}
where $\log$ is the principal branch of the logarithm and $\mathcal{U}\subseteq\mathbb{R}^d$ is an open convex neighborhood of $0$ which is small enough to ensure that $\log$ is continuous on $\{\widehat{\phi}(\xi+\xi_0)/\widehat{\phi}(\xi_0)\}$ as $\xi$ varies in $\mathcal{U}$. Our supposition that $\phi\in\mathcal{S}_d$ guarantees that $\Gamma_{\xi_0}\in C^\infty(\mathcal{U})$ and so it makes sense to consider its Taylor expansion at $0$. It is the nature of these Taylor expansions (and hence the nature of $\Gamma_{\xi_0}$) that determines the asymptotic nature of $\phi^{(n)}$.\\

\noindent The vast majority of existing theory on the asymptotic behavior of convolution powers pertains to the $1$-dimensional setting, i.e., $d=1$. In fact, determining the asymptotic behavior of $\phi^{(n)}$ where $\phi$ is finitely supported on $\mathbb{Z}$ and takes on only real values is known as de Forest's problem and dates back to its initial investigation by Erastus L. de Forest in the nineteenth century driven by de Forest's interest in data smoothing \cite{Stigler78,DSC14}. This study was continued by I. J. Schoenberg and T. N. E. Greville \cite{Gre66,Sch53}, both of whom proved local limit theorems under hypotheses (stronger than) described below. Tied to advancements in scientific computing, the general problem of understanding the asymptotic behavior of the convolution powers of a complex function on $\mathbb{Z}$ was reinvigorated in the second half of the twentieth century by its appearance in numerical solution algorithms for partial differential equations; we encourage the reader to look at the excellent survey of V. Thom\'{e}e \cite{Tho69} for an account of this story (see also Subsection 5.2 of \cite{DSC14}). From the perspective of numerical PDEs, one is often concerned with the so-called max-norm stability\footnote{We say that $L_\phi$ \textbf{stable in the maximum norm} if, for some $C>0$, $\|L_\phi^n f\|_{L^\infty(\mathbb{R})}\leq C\|f\|_{L^\infty(\mathbb{R})}$ for all $n\in\mathbb{N}_+$ and $f\in L^\infty(\mathbb{R})$. Upon noting that $L_\phi^n=L_{\phi^{(n)}}$, it is easily seen that max-norm stability is equivalent to the property that $\sup_n\|\phi^{(n)}\|_1<\infty$. Recognizing $\ell^1(\mathbb{Z}^d)$ as a Banach algebra equipped with the convolution product, this property is referred to as power boundedness \cite{KLU14,BSch70}.} of operators $L_\phi:L^\infty(\mathbb{R})\to L^\infty(\mathbb{R})$ defined by $L_\phi f(x)=\sum_y \phi(y)f(x+hy)$ for $f\in L^\infty(\mathbb{R})$ where $\phi\in \ell^1(\mathbb{Z})$ and $h$ is a fixed positive parameter. In 1965, V. Thom\'{e}e characterized max-norm stability for operators $L_\phi$ associated to finitely supported functions $\phi\in \ell^1(\mathbb{Z})$ \cite[Theorem 1]{Tho65}. Though this characterization is not directly relevant to our present goals, in the course of his proof, Thom\'{e}e introduced the following definition which was found essential to the theory in \cite{RSC15}.
\begin{definition}\label{def:OneDTypes}
Let $\phi:\mathbb{Z}\to\mathbb{C}$ be such that $\sup_{\xi}|\widehat{\phi}|=1$ and let $\xi_0\in\Omega(\phi)$. 
\begin{enumerate}
\item We say that $\xi_0$ is of Type 1 (or Type $\gamma$) of order $m$ for $\widehat{\phi}$ if there is an even natural number $m=m_{\xi_0}\geq 2$, a real number $\alpha_{\xi_0}$, and a complex number $\beta_{\xi_0}$ with $\Re(\beta_{\xi_0})>0$ for which
\begin{equation*}
\Gamma_{\xi_0}(\xi)=i\alpha_{\xi_0}\xi-\beta_{\xi_0}\xi^{m_{\xi_0}}+o\left(\xi^{m_{\xi_0}}\right)
\end{equation*}
as $\xi\to 0$.
\item We say that $\xi_0$ is of Type 2 (or Type $\beta$)  of order $m$ for $\widehat{\phi}$ if there is a natural number $m=m_{\xi_0}\in\{2,3,\dots,\}$, a real number $\alpha_{\xi_0}$, a real-valued polynomial $q_{\xi_0}(\xi)$ with $\beta_{\xi_0}=iq_{\xi_0}(0)\neq 0$, an even number $k_{\xi_0}>m_{\xi_0}$, and a positive number $\gamma_{\xi_0}$ for which
\begin{equation*}
\Gamma_{\xi_0}(\xi)=i\alpha_{\xi_0}\xi-iq_{\xi_0}(\xi)\xi^{m_{\xi_0}}-\gamma_{\xi_{0}}\xi^{k_{\xi_0}}+o\left(\xi^{k_{\xi_0}}\right)
\end{equation*}
as $\xi\to 0$.
\end{enumerate}
\end{definition}
\noindent The distinction made by the definition above essentially concerns the nature of the lowest order (non-linear) monomial appearing in the Taylor expansion for $\Gamma_{\xi_0}$. If the monomial has a coefficient $\beta_{\xi_0}$ with strictly positive real part, $\xi_0$ is of Type 1 for $\widehat{\phi}$ and, if the coefficient $\beta_{\xi_0}=ip_{\xi_0}(0)$ is non-zero and purely imaginary, $\xi_0$ is of Type 2 for $\widehat{\phi}$. As observed by Thom\'{e}e (see Section 3 of \cite{Tho65}) , for a finitely supported $\phi:\mathbb{Z}\to\mathbb{C}$ which is supported on more than one point in $\mathbb{Z}$ and has $\sup_{\xi}|\widehat{\phi}|=1$, $\Omega(\phi)$ is necessarily a finite set and every element $\xi\in \Omega(\phi)$ is a point of Type 1 or Type 2 for $\widehat{\phi}$. As a consequence and in view of \eqref{eq:FTConvolution}, understanding the asymptotic behavior of $\phi^{(n)}(x)$ reduces to the problem of analyzing the nature of the contributions from points of Type 1 and Type 2 to $\widehat{\phi}(\xi)^n$ which we explain as follows.\\

\noindent For illustrative purposes, let's assume that $\Omega(\phi)$ consists only of a single element, i.e., $\Omega(\phi)=\{\xi_0\}$. If $\xi_0$ is a point of Type 1 for $\widehat{\phi}$ with associated even integer $m=m_{\xi_0}$, real number $\alpha=\alpha_{\xi_0}$ and complex number $\beta=\beta_{\xi_0}$ with $\Re(\beta)>0$, Theorem 2.3 of \cite{DSC14} guarantees that
\begin{equation}\label{eq:LLTType1OneAttractor}
\phi^{(n)}(x)=\widehat{\phi}(\xi_0)^n e^{-ix\xi_0}H_{m,\beta}^n(x-n\alpha)+o(n^{-1/m})
\end{equation}
uniformly for $x\in\mathbb{Z}$ where $H_{m,\beta}^{(\cdot)}(\cdot)$ is defined by
\begin{equation}\label{eq:OneDPositiveHomAttractor}
H_{m,\beta}^t(x)=\frac{1}{2\pi}\int_{\mathbb{R}}e^{-t\beta\xi^m}e^{-ix\xi}\,d\xi
\end{equation}
for $t>0$ and $x\in\mathbb{R}$. Because $\Re(\beta)>0$, the integral defining $H_{m,\beta}$ converges absolutely for all $x\in\mathbb{R}$. In fact, $H_{m,\beta}^t(\cdot)$ is a Schwartz function (for each $t>0$) and is a fundamental solution to the heat-type equation
\begin{equation*}
\frac{\partial u}{\partial t}+i^m\beta\frac{\partial^m u}{\partial x^m}=0.
\end{equation*}
The function $H_{m,\beta}$ also enjoys the property
\begin{equation}\label{eq:OneDScale}
H_{m,\beta}^t(x)=t^{-1/m}H_{m,\beta}^1(t^{-1/m}x)
\end{equation} for all $t>0$ and $x\in\mathbb{R}$ and so \eqref{eq:LLTType1OneAttractor} can be equivalently stated as
\begin{equation*}
\phi^{(n)}(x)=n^{-1/m}\widehat{\phi}(\xi_0)^ne^{-ix\xi_0}H_{m,\beta}^1(n^{-1/m}(x-n\alpha))+o(n^{-1/m})
\end{equation*}
uniformly for $x\in\mathbb{Z}$. We remark that the error in \eqref{eq:LLTType1OneAttractor} was recently improved by L. Coeuret (see Theorem 1 and Corollary 1 of \cite{Co22}). Coeuret's results show, in particular, that the uniform error term $o(n^{-1/m})$ can be replaced by $O(n^{-2/m})$; we discuss this further in Section \ref{sec:Discussion}.

\begin{example}\label{ex:OneDEx1}
Let $\phi:\mathbb{Z}\to\mathbb{R}$ be defined by
\begin{equation*}
\phi(x)=\begin{cases}
1/2 & x=0\\
1/3 & x=\pm 1\\
-1/12 & x=\pm 2\\
0 & \mbox{otherwise}
\end{cases}
\end{equation*}
for $x\in\mathbb{Z}$. We observe easily that
\begin{equation*}
\widehat{\phi}(\xi)=\frac{1}{2}+\frac{2}{3}\cos(\xi)-\frac{1}{6}\cos(2\xi)
\end{equation*}
for $\xi\in\mathbb{R}$. With this, it is easy to verify that $\sup_{\xi}|\widehat{\phi}|=\widehat{\phi}(0)=1$, $\Omega(\phi)=\{0\}$, and
\begin{equation*}
\Gamma_0(\xi)=\log\left(\frac{\widehat{\phi}(\xi)}{\widehat{\phi}(0)}\right)=-\frac{1}{12}\xi^4+O(\xi^6)
\end{equation*}
as $\xi\to 0$. Thus, $\xi_0=0$ is a point of Type 1 for $\widehat{\phi}$ of order $m=4$, $\alpha=0$ and $\beta=1/12$. Hence, our local limit theorem says that
\begin{equation*}
\phi^{(n)}(x)=H_{4,1/12}^n(x)+o(n^{-1/4})=n^{-1/4}H_{4,1/12}^1(n^{-1/4}x)+o(n^{-1/4})
\end{equation*}
uniformly for $x\in\mathbb{Z}$; here
\begin{equation*}
H_{4,1/12}^1(x)=\frac{1}{2\pi}\int_{\mathbb{R}}e^{-\xi^4/12}e^{-ix\xi}\,d\xi.
\end{equation*}
Illustrating this local limit theorem, Figure \ref{fig:OneDEx1} depicts $\phi^{(n)}$ alongside $H^n_{4,1/12}$ for $n=300,600$.
\begin{figure}[h!]
\centering\includegraphics[width=5in]{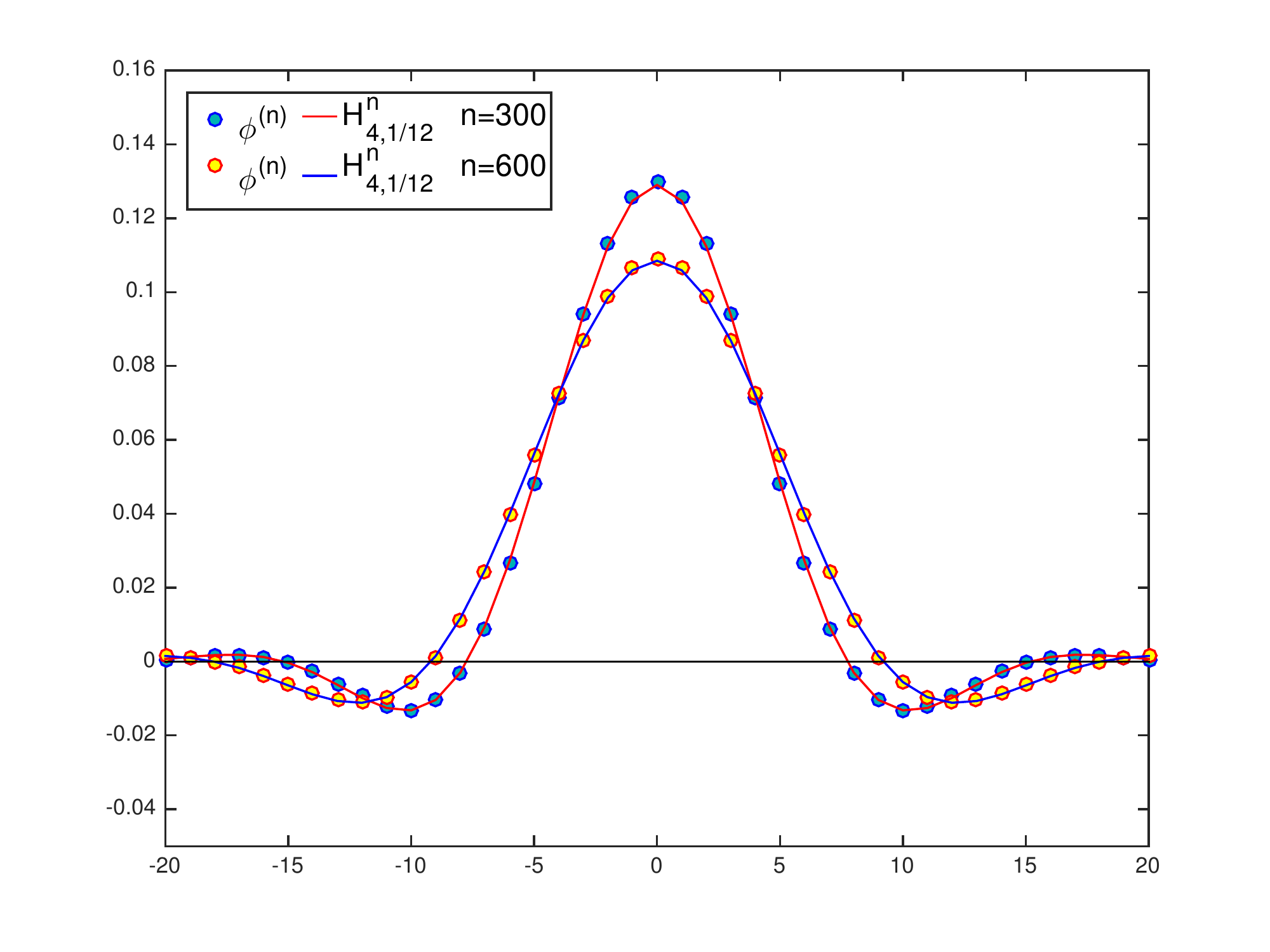}
 \caption{$\phi^{(n)}$ and $H_{4,1/12}^n$ for $n=300,600$}
\label{fig:OneDEx1}
\end{figure}
\end{example}

\noindent In the case that $\Omega(\phi)=\{\xi_0\}$ and $\xi_0$ is a point of Type 2 for $\widehat{\phi}$ with integer $m=m_{\xi_0}$, real number $\alpha=\alpha_{\xi_0}$ and purely imaginary $\beta=\beta_{\xi_0}=iq_{\xi_0}(0)$, Theorem 1.3 of \cite{RSC15} guarantees that, for any compact set $K$,
\begin{equation}\label{eq:LLTType2OneAttractor}
\phi^{(n)}(x)=\widehat{\phi}(\xi_0)^n e^{-ix\xi_0}H_{m,\beta}^n(x-n\alpha)+o(n^{-1/m})
\end{equation}
uniformly for $x\in \left(n\alpha+n^{1/m}K\right)\cap \mathbb{Z}$ where $H_{m,\beta}^{(\cdot)}(\cdot)$ is given by
\begin{equation}\label{eq:OneDImagHomAttractor}
H_{m,\beta}^t(x)=\frac{1}{2\pi}\int_{\mathbb{R}}e^{-t\beta\xi^m}e^{-ix\xi}\,d\xi=\frac{1}{2\pi}\int_{\mathbb{R}}e^{-itq\xi^m}e^{-ix\xi}\,d\xi
\end{equation}
for $t>0$ and $x\in\mathbb{R}$ where $\beta=iq$ for the real non-zero $q=q_{\xi_0}(0)$. In contrast to \eqref{eq:OneDPositiveHomAttractor}, the integrand $e^{-itq\xi^m}e^{-ix\xi}$ is not absolutely integrable and so we interpret the integral in \eqref{eq:OneDImagHomAttractor} as a sum of the improper Riemann integrals
\begin{equation*}
\int_{-\infty}^0 e^{-itq\xi^m}e^{-ix\xi}\,d\xi\hspace{1cm}\mbox{and}\hspace{1cm}\int_0^\infty e^{-itq\xi^m}e^{-ix\xi}\,d\xi
\end{equation*}
which converge on account of the highly oscillatory nature of the integrand as $\xi\to\pm\infty$. Despite the differing interpretations of the integrals in \eqref{eq:OneDPositiveHomAttractor} and \eqref{eq:OneDImagHomAttractor}, the function $H_{m,\beta}=H_{m, iq}$ is smooth and also satisfies the property \eqref{eq:OneDScale}. We remark that, when $m=2$, $H^t_{2,\beta}=H^t_{2,iq}$ is the heat kernel evaluated at imaginary time $\tau=itq$, i.e.,
\begin{equation*}
H^t_{2,iq}(x)=\frac{1}{\sqrt{4\pi i qt}}\exp\left(-\frac{x^2}{4iqt}\right)
\end{equation*}
for $t>0$ and $x\in\mathbb{R}^d$ and, when $m=3$, $H^t_{3,\beta}=H^t_{3,iq}$ is a scaled Airy function. 

\begin{example}\label{ex:OneDEx2}
Let $\phi:\mathbb{Z}\to\mathbb{C}$ be given by
\begin{equation*}
\phi(x)=\begin{cases}
\frac{1}{4}(3-i) & x=0\\
\frac{1}{24}(4+3i) & x=\pm 1\\
\frac{1}{24} & x=\pm 2\\
0 & \mbox{otherwise}
\end{cases}
\end{equation*}
for $x\in\mathbb{Z}$. In this case, it is easily verified that $\sup_\xi|\widehat{\phi}|=1$, $\Omega(\phi)=\{0\}$, and 
\begin{equation*}
\Gamma_0(\xi)=-\frac{i}{8}\xi^2-\left(\frac{13}{384}-\frac{i}{96}\right)\xi^4+O(\xi^5)
\end{equation*}
as $\xi\to 0$. From this, we see easily that $\xi_0=0$ is a point of Type 2 for $\widehat{\phi}$ of order $2$, $\alpha=0$, and $\beta=i/8$. In this case, our local limit theorem says that, for any compact set $K\subseteq\mathbb{R}$,
\begin{eqnarray*}
\phi^{(n)}(x)&=&H_{2,i/8}^n(x)+o(n^{-1/2})\\
&=&n^{-1/2}H_{2,i/8}(n^{-1/2}x)+o(n^{-1/2})\\
&=&n^{-1/2}\frac{1}{\sqrt{4\pi i/8}}\exp(-x^2/(4n i/8))+o(n^{-1/2})\\
\end{eqnarray*}
for $x\in \left(n^{1/2}K\right)\cap\mathbb{Z}$. To illustrate this local limit theorem, Figure \ref{fig:OneDEx2} shows real parts of $\phi^{(n)}$ and the attractor $H_{2,i/8}^n$ for $n=1,000$.

\begin{figure}[h!]
\centering\includegraphics[width=5in]{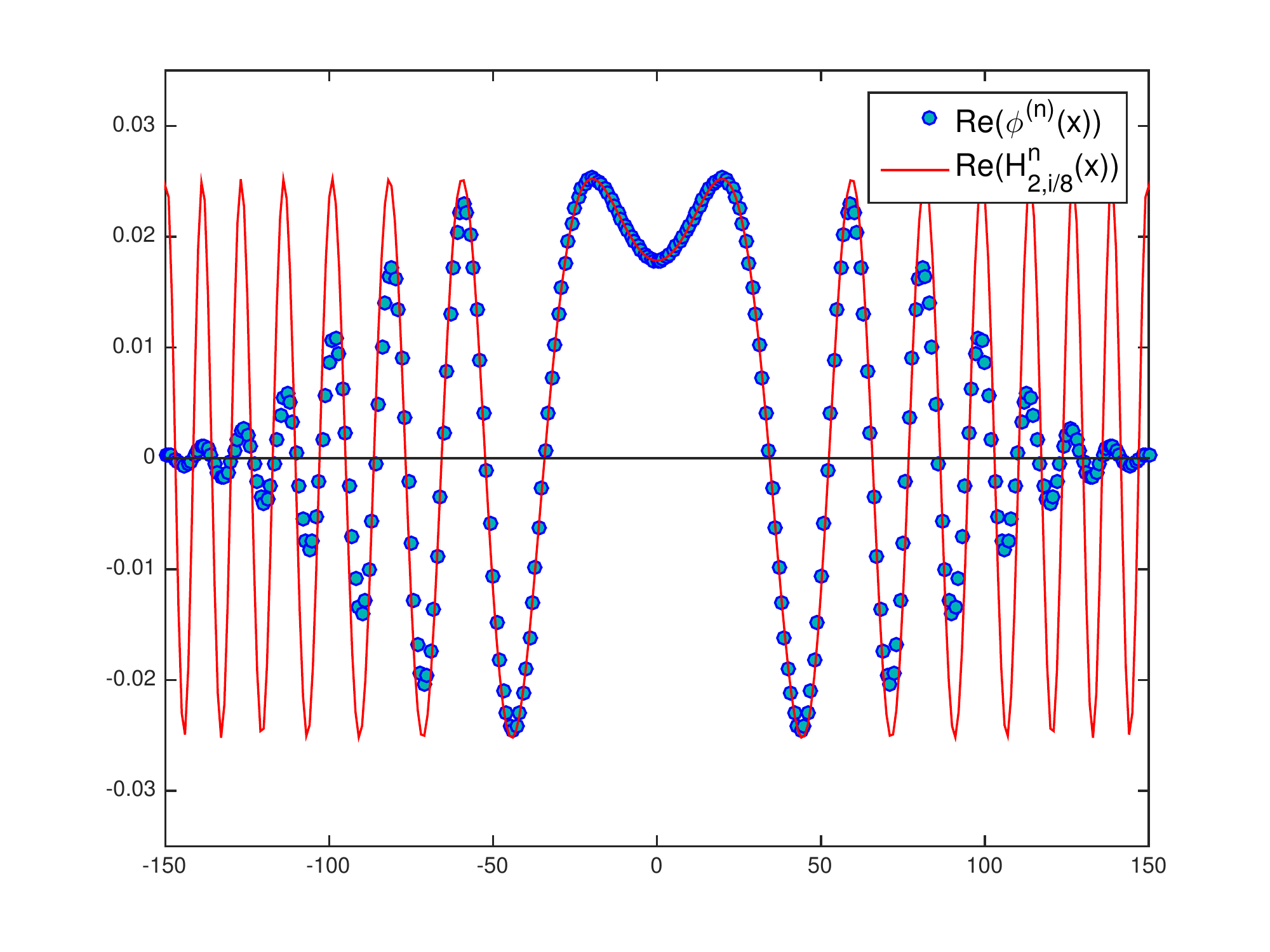}
 \caption{$\Re(\phi^{(n)})$ and $\Re(H_{2,i/8}^n)$ for $n=1,000$}
\label{fig:OneDEx2}
\end{figure}
\end{example}

\begin{remark}\label{rmk:Support} In comparing the local limit theorems \eqref{eq:LLTType1OneAttractor} and \eqref{eq:LLTType2OneAttractor}, observe that \eqref{eq:LLTType2OneAttractor} is guaranteed to hold uniformly for $x\in \left(n\alpha+ n^{1/m}K\right)\cap\mathbb{Z}$ for a compact set $K\subseteq \mathbb{R}$ whereas \eqref{eq:LLTType1OneAttractor} is valid uniformly for $x\in\mathbb{Z}$. As illustrated in Figure \ref{fig:OneDEx2}, this limitation of \eqref{eq:LLTType2OneAttractor} is necessary because $\phi^{(n)}$ is finitely supported whereas the attractor $H_{2,iq}^n(y)$ has constant modulus $1/\sqrt{4 \pi i  q n}$. An analogous limitation will also play a role in our $d$-dimensional theory. As we discuss in Section \ref{sec:Discussion}, there are a number of open questions related to these observations, one of which concerns the behavior of $\phi^{(n)}$ outside $n\alpha+n^{1/m}K$ but within its support. We remark that, when $m>2$, the local limit theorem \eqref{eq:LLTType2OneAttractor} does hold uniformly for $x\in\mathbb{Z}$ (see Theorem 1.5 of \cite{RSC15}).
\end{remark}

\noindent Beyond the case that $\Omega(\phi)=\{\xi_0\}$, i.e., in cases where $|\widehat{\phi}(\xi)|$ is maximized at more than one point, local limit theorems for $\phi^{(n)}$ involve sums of the so-called attractors $H_{m,\beta}$ of the forms \eqref{eq:OneDPositiveHomAttractor} and \eqref{eq:OneDImagHomAttractor} according to whether the points $\xi\in\Omega(\phi)$ are of Type 1 or Type 2 for $\widehat{\phi}$. For example, if $\Omega(\phi)=\{\xi_1,\xi_2,\dots,\xi_l\}$ and, for each $k=1,2,\dots, l$, $\xi_k$ is a point of Type 1 for $\widehat{\phi}$ with $m_k=m_{\xi_k}$, $\alpha_k=\alpha_{\xi_k}$, and $\beta_k=\beta_{\xi_k}$, Theorem 2.3 of \cite{DSC14} (see also Theorem 1.2 of \cite{RSC15}, Theorem 1 of \cite{Co22}, and Theorem 1.5 of \cite{RSC17}) guarantees that
\begin{equation}\label{eq:LLTType1OneAttractorMultiple}
\phi^{(n)}(x)=\sum_{j=1}^A \widehat{\phi}(\xi_{k_j})^ne^{-ix\xi_{k_j}}H_{m,\beta_{k_j}}^n(x-n\alpha_{k_j})+o(n^{-1/m})
\end{equation}
uniformly for $x\in\mathbb{Z}$ where $m=\max_{k=1,2,\dots,l}m_{k}$ and the points $\xi_{k_1},\xi_{k_2},\dots,\xi_{k_A}\in \Omega(\phi)$ are those for which $m_{k_j}=m$ for all $j=1,2,\dots,A$. In fact, according to Theorem 1.2 of \cite{RSC17}, \eqref{eq:LLTType1OneAttractorMultiple} also holds in the case that the points $\xi_1,\xi_2,\dots,\xi_l\in\Omega(\phi)$ are of Type 1 or Type 2 (which is always true when $\phi$ is finitely supported) as long as $m>2$. Thanks to Thom\'{e}e's key observations, proving a local limit theorem robust enough to handle the entire class of finitely supported functions $\phi:\mathbb{Z}\to\mathbb{C}$ required the consideration that $\Omega(\phi)$ can consist of more than one point \emph{and} that its element(s) can be of both Type 1 and Type 2 for $\widehat{\phi}$. The results of I. J. Schoenberg \cite{Sch53}, T. N. E. Greville \cite{Gre66}, and P. Diaconis and L. Saloff-Coste \cite{DSC14} treated this theory for (proper) subsets of these possibilities (see also the work of K. Hochberg \cite{Hoch80} who proved a generalized central limit theorem for the convolution powers of a signed Borel measure satisfying analogous hypotheses). Though we will not state it here due to its complexity, Theorem 1.3 of \cite{RSC15} (a local limit theorem) provides a complete description of the asymptotic behaviors of $\phi^{(n)}$ when $\phi$ is any (normalized) finitely supported complex-valued function on $\mathbb{Z}$. When applied to real-valued and finitely supported functions on $\mathbb{Z}$, it represents a full solution to de Forest's problem.\\

\noindent We now move beyond one dimension and consider a general $\phi\in S_d\subseteq\ell^1(\mathbb{Z}^d)$ for which $\sup_\xi|\widehat{\phi}(\xi)|=1$. In light of the natural complexity of $\mathbb{R}^d$, a categorization of the possible behaviors of $\Gamma_{\xi_0}$ along the lines of Definition \ref{def:OneDTypes}, even for finitely supported functions, appears to be a very difficult task (if it is possible at all). Despite this drawback, we still seek a reasonable generalization of Definition \ref{def:OneDTypes} which will capture much of the behaviors commonly seen for $\phi\in\mathcal{S}_d$. To produce such a generalization, we must first consider what will replace the monomials $\xi^m$ appearing as the lowest order (non-linear) terms in the expansion for $\Gamma_{\xi_0}$ in Definition \ref{def:OneDTypes}.\\

\noindent First, let $E$ be a linear endomorphism of $\mathbb{R}^d$ (written $E\in\End(\mathbb{R}^d)$) and define
\begin{equation*}
T_t=t^E=\exp(\log(t)E)=\sum_{k=0}^\infty \frac{(\log(t))^k}{k!}E^k
\end{equation*}
for each $t>0$. The map $t\mapsto T_t=t^E$ is a Lie group homomorphism from the positive real numbers (under multiplication) to the general linear group, $\Gl(\mathbb{R}^d)$, and its image is a one-parameter subgroup of $\Gl(\mathbb{R}^d)$ which we denote by $\{T_t\}$ or $\{t^E\}$. This one parameter group $\{t^E\}$ is said to be \textbf{generated by $E$} and $E$ is said to be \textbf{the generator of $\{t^E\}$}. In fact, all such one-parameter subgroups $\{T_t\}$ of $\Gl(\mathbb{R}^d)$ are of this form. In Subsection \ref{subsec:OneParameterGroups} of the Appendix, we have collected some useful facts about these one-parameter groups. For a function $P:\mathbb{R}^d\to\mathbb{C}$, we say that \textbf{$P$ is homogeneous with respect to $\{t^E\}$} if
\begin{equation}\label{eq:Homogeneous}
P(t^E\xi)=tP(\xi)
\end{equation}
for all $t>0$ and $\xi\in\mathbb{R}^d$. By an abuse of language, we also say that \textbf{$P$ is homogeneous with respect to $E$}. The set of all $E\in\End(\mathbb{R}^d)$ for which \eqref{eq:Homogeneous} is satisfied is said to be the \textbf{exponent set of $P$} and will be denoted by $\Exp(P)$. When $P$ is real valued, we say that $P$ is \textbf{positive definite} provided that $P\geq 0$ and $P(\xi)=0$ only when $\xi=0$. Also, the set $S_P=\{\eta\in\mathbb{R}^d:P(\eta)=1\}$ is called the \textbf{unital level set of $P$}. 
\begin{definition}
Let $P:\mathbb{R}^d\to\mathbb{R}$ be continuous and positive-definite. If $\Exp(P)$ is non-empty and $S_P$ is compact, we say that $P$ is a positive homogeneous function.
\end{definition}

\begin{example}\label{ex:PowersofNorm}
Given any $\nu>0$, consider the function $\mathbb{R}^d\ni\xi\mapsto |\xi|^\nu$ where $|\cdot|$ is the Euclidean norm on $\mathbb{R}^d$. It is not hard to see that $|\xi|^\nu$ is continuous, positive definite and $S_{|\cdot|^\nu}$ is precisely the unit sphere $\mathbb{S}\subseteq\mathbb{R}^d$, which is compact. Consider $E_{1/\nu}:=(1/\nu)I$ where $I\in \Gl(\mathbb{R}^d)$ is the identity transformation and observe that
\begin{equation*}
|t^{E_{1/\nu}}\xi|^\nu=|t^{1/\nu}\xi|^\nu=t|\xi|^\nu
\end{equation*}
for all $t>0$ and $\xi\in\mathbb{R}^d$. Thus $E_{1/\nu}\in\Exp(|\cdot|^\nu)$ and so $\xi\mapsto |\xi|^\nu$ is a positive homogeneous function. We remark that
\begin{equation*}
\Exp(|\cdot|^\nu)=E_{1/\nu}+\mathfrak{o}(d)
\end{equation*}
where $\mathfrak{o}(d)$ is the Lie algebra of the orthogonal group $\Ort(\mathbb{R}^d)$ and is characterized by the set of skew symmetric matrices. Correspondingly, $\End(|\cdot|^\nu)$ is not a singleton when $d>1$. In the context of one dimension, our argument above ensures that, for each even natural number $m\in\mathbb{N}_+$, the monomial $\xi^{m}$ is positive homogeneous. Also, for each $n\in\mathbb{N}_+$, $\mathbb{R}\ni \xi\mapsto |\xi^n|$ is positive homogeneous. 
\end{example}

\noindent We refer the reader to the recent article \cite{BR21} which develops a theory of positive homogeneous functions and presents many examples and pictures. Before we introduce a central class of positive homogeneous functions that will be of interest for us, it is helpful to state the following characterization of positive homogeneous functions which appears as Proposition 1.1. of \cite{BR21} and is proven in Section 2 therein.

\begin{proposition}\label{prop:CharofPosHom}
Let $P:\mathbb{R}^d\to\mathbb{R}$ be continuous, positive definite, and have non-empty exponent set $\Exp(P)$. Then the following are equivalent.
\begin{enumerate}
\item $P$ is positive homogeneous.
\item There exists $M>0$ for which $P(\xi)>1$ for all $|\xi|\geq M$.
\item For each $E\in\Exp(P)$, $\{t^E\}$ is \textbf{contracting} in the sense that
\begin{equation*}
\lim_{t\to 0}\|t^E\|=0
\end{equation*}
where $\|\cdot\|$ is the operator norm on $\End(\mathbb{R}^d)$.
\item There exists $E\in\Exp(P)$ for which $\{t^E\}$ is contracting.
\item\label{item:CharofPosHomLimToInf} We have $\lim_{\xi\to\infty}P(\xi)=\infty$.
\end{enumerate}
\end{proposition}

\noindent Armed with this proposition, we now introduce the class of semi-elliptic polynomials.

\begin{example}\label{ex:SemiElliptic}
Let $\mathbf{m}=(m_1,m_2,\dots,m_d)\in\mathbb{N}_+^d$ be a $d$-tuple of positive integers and, for a multi-index $\alpha=(\alpha_1,\alpha_2,\dots,\alpha_d)\in\mathbb{N}^d$, define
\begin{equation*}
|\alpha:\mathbf{m}|=\sum_{k=1}^d\frac{\alpha_k}{m_k}.
\end{equation*}
In this notation, we consider a multivariate polynomial $P:\mathbb{R}^d\to\mathbb{C}$ of the form
\begin{equation}\label{eq:SemiEllipticPolynomial}
P(\xi)=\sum_{|\alpha:\mathbf{m}|=1}a_\alpha \xi^\alpha
\end{equation}
where $\{a_\alpha\}\subseteq\mathbb{C}$ and $\xi^\alpha=\xi_1^{\alpha_1}\xi_2^{\alpha_2}\cdots\xi_d^{\alpha_d}$ for each multi-index $\alpha=(\alpha_1,\alpha_2,\dots,\alpha_d)\in\mathbb{N}^d$ and $\xi=(\xi_1,\xi_2,\dots,\xi_d)\in\mathbb{R}^d$. In the language of L. H\"{o}mander \cite{hormander_analysis_1983}, the polynomial \eqref{eq:SemiEllipticPolynomial} is said to be semi-elliptic if it vanishes only at the origin. Observe that, for $E=E_{1/\mathbf{m}}\in\End(\mathbb{R}^d)$ with standard matrix representation $\diag(m_1^{-1},m_2^{-1},\dots,m_d^{-1})$,
\begin{equation*}
P(t^{E}\xi)=\sum_{|\alpha:\mathbf{m}|=1}a_\alpha (t^{1/m_1}\xi_1)^{\alpha_1}(t^{1/m_2}\xi_2)^{\alpha_2}\cdots(t^{1/m_d}\xi_d)^{\alpha_d}=\sum_{|\alpha:\mathbf{m}|=1}a_{\alpha}t^{|\alpha:\mathbf{m}|}\xi^\alpha=tP(\xi)
\end{equation*}
for all $t>0$ and $\xi=(\xi_1,\xi_2,\dots,\xi_d)\in\mathbb{R}^d$. Thus, $E\in \Exp(P)$. It is clear that $\{t^{E}\}$ is contracting and so, by virtue of Proposition \ref{prop:CharofPosHom}, we obtain the following result:
\begin{quote}
If a semi-elliptic polynomial of the form \eqref{eq:SemiEllipticPolynomial} is positive definite, it is positive homogeneous.
\end{quote}
For two concrete examples, consider the polynomials $P_1$ and $P_2$ on $\mathbb{R}^2$ given by
\begin{equation*}
P_1(\xi_1,\xi_2)=\xi_1^2+\xi_1\xi_2^2+\xi_2^4\hspace{1cm}\mbox{and}\hspace{1cm}P_2(\xi_1,\xi_2)=\xi_1^6+\xi_2^4.
\end{equation*}
It is easy to see that both polynomials are of the form \eqref{eq:SemiEllipticPolynomial} with $\mathbf{m}=(2,4)$ for $P_1$ and $\mathbf{m}=(6,4)$ for $P_2$. A routine verification shows that both are, in fact, positive definite and therefore positive homogeneous. As illustrated in these examples, a real-valued polynomial of the form \eqref{eq:SemiEllipticPolynomial} can only be positive definite (and hence positive homogeneous) provided that $\mathbf{m}$ is a $d$-tuple of positive even integers, i.e., $\mathbf{m}=2\mathbf{n}$, for $\mathbf{n}\in\mathbb{N}_+^d$.
\end{example}

\begin{remark}\label{rmk:ZdPosHomComp}
In \cite{RSC17}, a positive homogeneous polynomial is a complex-valued polynomial $P$ for which $R(\xi)=\Re(P(\xi))$ is positive definite and $\Exp(P)$ contains an endomorphism whose spectrum is real. By virtue of Proposition 2.2 of \cite{RSC17}, every such positive homogeneous polynomial $P$ is semi-elliptic in some coordinate system and from this it follows that $R$ is a positive homogeneous function in the sense of the present article. 
\end{remark}

\noindent Let $P$ be a positive homogeneous function and consider the set $\Sym(P)$ consisting of those $O\in\End(\mathbb{R}^d)$ for which
\begin{equation*}
P(O\xi)=P(\xi)
\end{equation*}
for all $\xi\in\mathbb{R}^d$. In the case that $P$ is the Euclidean norm $\mathbb{R}^d\ni \xi\mapsto |\xi|$, it is easy to see that $\Sym(|\cdot|)$ coincides with the orthogonal group, $\OdR$. As the following proposition shows, this is archetypal of the situation in general. We note that this proposition appears in Section 2.1 of \cite{BR21}; we give a proof here for completeness.

\begin{proposition}\label{prop:SymPCompact}
For a positive homogeneous function $P:\mathbb{R}^d\to\mathbb{R}$, $\Sym(P)$ is a compact subgroup of the general linear group, $\Gl(\mathbb{R}^d)$.
\end{proposition}

\begin{proof}
In view of the fact that $P$ is positive definite, it is straightforward to see that $\Sym(P)$ is a subgroup of $\Gl(\mathbb{R}^d)$. We shall prove that $\Sym(P)$ is compact by showing it is both closed and bounded. To see that $\Sym(P)$ is closed, let $\{O_n\}\subseteq\Sym(P)$ be a sequence which converges to $O\in\Gl(\mathbb{R}^d)$. By virtue of the continuity of $P$, for each $\xi\in\mathbb{R}^d$, we have
\begin{equation*}
P(O\xi)=\lim_{n\to\infty}P(O_n\xi)=\lim_{n\to\infty}P(\xi)=P(\xi)
\end{equation*}
which proves that $O\in\Sym(P)$ and therefore $\Sym(P)$ is closed. To see that $\Sym(P)$ is bounded, we assume, to reach a contradiction, that there exists a sequence $\{O_n\}\subseteq \Sym(P)$ and a sequence $\{\xi_n\}$ of elements on the unit sphere $\mathbb{S}$ of $\mathbb{R}^d$ for which $\lim_{n\to\infty}|O_n\xi_n|=\infty$. By virtue of Item \ref{item:CharofPosHomLimToInf} of Proposition \ref{prop:CharofPosHom}, we have
\begin{equation*}
\lim_{n\to\infty}P(\xi_n)=\lim_{n\to\infty}P(O_n\xi_n)=\infty
\end{equation*}
but this is impossible for we know that $P$ is continuous on $\mathbb{R}^d$ and therefore bounded on $\mathbb{S}$.
\end{proof}

\noindent In view of the preceding proposition, we shall refer to $\Sym(P)$ as the \textbf{symmetric group of $P$.} The fact that $\Sym(P)$ is compact allows us to establish the following important invariant of $\Exp(P)$.
\begin{proposition}
Let $P$ be a positive homogeneous function and let $E_1,E_2\in\Exp(P)$. Then
\begin{equation*}
\tr E_1=\tr E_2>0.
\end{equation*}
\end{proposition}
\begin{proof}
Suppose that, for $E\in\End(\mathbb{R}^d)$, $\{t^E\}$ is contracting. Then, by virtue of the continuity of the determinant map and the fact that $\det(t^E)=t^{\tr E}$, we find that
\begin{equation*}
\lim_{t\to 0}t^{\tr E}=\lim_{t\to 0}\det(t^E)=0
\end{equation*}
and so $\tr E>0$. Thus, in view of Proposition \ref{prop:CharofPosHom}, every element of $\Exp(P)$ has positive trace. It remains to show that $\tr E_1=\tr E_2$ for all $E_1,E_2\in\Exp(P)$. To this end, let $E_1$ and $E_2$ be two such elements and, for $t>0$, set $O_t=t^{E_1}t^{-E_2}=t^{E_1}(1/t)^{E_2}$. We observe that
\begin{equation*}
P(O_t \xi)=P(t^{E_1}t^{-E_2}\xi)=tP((1/t)^{E_2}\xi)=t(1/t)P(\xi)=P(\xi)
\end{equation*}
for all $\xi\in\mathbb{R}^d$. Consequently, $O_t\in\Sym(P)$ and, because $\Sym(P)$ is a compact group by virtue of Proposition \ref{prop:SymPCompact}, $\det(O_t)=\pm 1$ for each $t>0$.   Consequently,
\begin{equation*}
1=\abs{\det(O_t)}=\abs{\det(t^{E_1})\det(t^{-E_2})}=\abs{t^{\tr E_1}t^{-\tr E_2}}=t^{\tr E_1-\tr E_2}
\end{equation*}
for all $t>0$ and therefore $\tr E_1=\tr E_2$.
\end{proof}
\noindent By virtue of the above proposition, given a positive homogeneous function $P$, we define \textbf{the homogeneous order of $P$} to be the unique positive number $\mu_P$ for which
\begin{equation*}
\mu_P=\tr E
\end{equation*}
for all $E\in\Exp(P)$. Henceforth, when we speak of a positive homogeneous function, we take along with it its exponent set $\Exp(P)$, its symmetric group $\Sym(P)$, its unital level set $S_P$, and its homogeneous order $\mu_P$. \\

\begin{example}
In Example \ref{ex:PowersofNorm}, we considered the positive homogeneous function $\xi\mapsto |\xi|^\nu$ where $\nu>0$. In this case, $S_{|\cdot|^\nu}$ coincides with the unit sphere $\mathbb{S}\subseteq\mathbb{R}^d$, as we mentioned previously, $\Sym(|\cdot|^\nu)$ is precisely the orthogonal group $\OdR$ and, because $E=I/\nu\in \Exp(|\cdot|^\nu)$, we have $\mu_{|\cdot|^\nu}=d/\nu$. In particular, when $d=1$ and $\nu=m$ for some even natural number $m$, we find that $\xi\mapsto \xi^m$ is positive homogeneous with $\mu_{\xi^m}=1/m$. This one-dimensional case is noteworthy because, in addition to the fact that $\xi^m$ appears as the dominant term in the Type 1 in Definition \ref{def:OneDTypes}, the error $o(n^{-1/m})$ in \eqref{eq:LLTType1OneAttractor} (and the so-called on-diagonal scaling $t^{-1/m}$ of $H_{m,\beta}^t$ in \eqref{eq:OneDScale}) can be expressed in terms of the homogeneous order $\mu_{\xi^m}=1/m$. A similar observation can be made for points of Type 2 where $m$ is not necessarily even. As we will see, these observations are emblematic of the situation in $\mathbb{Z}^d$ for $d>1$ (see also Theorem 1.4. of \cite{RSC17}). 

In reference to Example \ref{ex:SemiElliptic}, for the positive homogeneous semi-elliptic polynomials $P$ of the form \eqref{eq:SemiEllipticPolynomial} (with $\mathbf{m}=(m_1,m_2,\dots,m_d)\in\mathbb{N}_+^d$), the unital level set $S_P$ and symmetric group $\Sym(P)$ will vary from example to example. It is straightforward to see that the homogeneous order of such a semi-elliptic polynomial $P$ is given by
\begin{equation*}
\mu_P=|\mathbf{1}:\mathbf{m}|=\frac{1}{m_1}+\frac{1}{m_2}+\cdots+\frac{1}{m_d}.
\end{equation*}
\end{example}

\noindent We are almost ready to present our generalization of Definition \ref{def:OneDTypes} in the context of $\mathbb{Z}^d$. It remains to develop an appropriate notion of $o(\xi^m)$ where $\xi^m$ is replaced by a positive homogeneous function. To this end, we present the following definition taken from \cite{BR21}.

\begin{definition}
Let $\widetilde{P}$ be a complex-valued function which is defined and continuous on an open neighborhood $\mathcal{U}$ of $0$ in $\mathbb{R}^d$.  Also, let $E\in\End(\mathbb{R}^d)$ be such that $\{t^E\}$ is contracting. 
\begin{enumerate}
\item We say that $\widetilde{P}$ is \textbf{subhomogeneous with respect to $E$} if, for each $\epsilon>0$ and compact set $K\subseteq\mathbb{R}^d$, there exists $\tau>0$ for which 
\begin{equation*}
\abs{\widetilde{P}(t^E\xi)}\leq \epsilon t
\end{equation*}
for all $0<t<\tau$ and $\xi\in K$.
\item Given $l\in\mathbb{N}_+$, we say that $\widetilde{P}$ is \textbf{strongly subhomogeneous with respect to $E$ of order $l$} if,  $\widetilde{P}\in C^l(\mathcal{U})$ and, for each $\epsilon>0$ and compact set $K\subseteq\mathbb{R}^d$, there exists $\tau>0$ such that
\begin{equation*}
\abs{t^k\partial_t^k\widetilde{P}(t^E\xi)}\leq \epsilon t
\end{equation*}
for all $k\in\{0,1,2,\dots,l\}$, $0<t<\tau$, and $\xi\in K$.
\end{enumerate}
When $E$ is understood (and fixed), we will say that $\widetilde{P}$ is subhomogeneous if it is subhomogeneous with respect to $E$. Also, we will say that $\widetilde{P}$ is $l$-strongly subhomogeneous if it is strongly subhomogeneous with respect to $E$ of order $l$.
\end{definition}

\noindent From the definition, it is apparent that all strongly subhomogenous functions are subhomogenous and it is reasonable to interpret subhomogeneous as strongly subhomogeneous of order $0$. So that the inequalities in the definition make sense, we remark that the supposition that $\{t^E\}$ is contracting guarantees that, for each compact set $K$, there is $\tau_0>0$ for which $t^E\xi\in\mathcal{U}$ for all $0<t<\tau_0$ and $\xi\in K$ (see Proposition A.6 of \cite{BR21}). In Section \ref{sec:Examples}, we introduce a large class smooth functions which, for a given $E$, are strongly subhomogeneous of all orders. The following proposition connects the concept of subhomogeneity with that of a function being $o(P(\xi))$ as $\xi\to 0$ for some positive homogeneous function $P$. Its proof appears in Subsection \ref{subsec:HomAndSubHom} of the Appendix.

\begin{proposition}\label{prop:Subhomequivtolittleoh}
Let $P$ be a positive homogeneous function and $\widetilde{P}$ be a complex-valued function which is continuous on a neighborhood $0$ of $\mathbb{R}^d$. The following are equivalent:
\begin{enumerate}
\item\label{item:Subhomequivtolittleoh1} $\widetilde{P}(\xi)=o(P(\xi))$ as $\xi\to 0$.
\item\label{item:Subhomequivtolittleoh2} For every $E\in \Exp(P)$, $\widetilde{P}$ is subhomogeneous with respect to $E$.
\item\label{item:Subhomequivtolittleoh3} There exists $E\in \Exp(P)$ for which $\widetilde{P}$ is subhomogeneous with respect to $E$. 
\end{enumerate}
\end{proposition}

\noindent We are now ready to introduce the $d$-dimensional generalization of Definition \ref{def:OneDTypes} considered in this article. We recall that, because $\widehat{\phi}$ is smooth for $\phi\in S_d$, $\Gamma_{\xi_0}\in C^{\infty}(\mathcal{U})$ and so we can use Taylor's theorem to approximate $\Gamma_{\xi_0}$ near $0$. Presicsely, we write
\begin{equation}\label{eq:GammaExpansion}
    \Gamma_{\xi_0}(\xi)=i\alpha_{\xi_0}\cdot\xi -i\left(Q_{\xi_0}(\xi)+\widetilde{Q}_{\xi_0}(\xi)\right)-\left(R_{\xi_0}(\xi)+\widetilde{R}_{\xi_0}(\xi)\right)
\end{equation}
where $\alpha_{\xi_0}\in\mathbb{R}^d$, $Q_{\xi_0}$ and $R_{\xi_0}$ are real-valued polynomials which vanish at $0$ and contain no linear terms, and $\widetilde{Q}_{\xi_0}$ and $\widetilde{R}_{\xi_0}$ are real-valued smooth functions on $\mathcal{U}$ which vanish at $0$. The fact that this expansion contains no real linear part is seen necessary because $|\widehat{\phi}|$ has a local maximum at $\xi_0$.

\begin{definition}\label{def:Types}
Let $\phi\in\mathcal{S}_d$ with $\sup_{\xi}|\widehat{\phi}(\xi)|=1$ and, given $\xi_0\in\Omega(\phi)$, consider the expansion \eqref{eq:GammaExpansion} above.
\begin{enumerate}
    \item We say that $\xi_0$ is of \textbf{positive homogeneous type} for $\widehat{\phi}$ if $R_{\xi_0}$ is positive homogeneous and, there exists $E\in \Exp(R_{\xi_0})$ for which $Q_{\xi_0}$ is homogeneous with respect to $E$ and both $\widetilde{R}_{\xi_0}$ and $\widetilde{Q}_{\xi_0}$ are subhomogeneous with respect to $E$. In this case, we will write  $\mu_{\xi_0}=\mu_{R_{\xi_0}}$.
\item We say that $\xi_0$ is of \textbf{imaginary homogeneous type} for $\widehat{\phi}$ if $|Q_{\xi_0}|$ and $R_{\xi_0}$ are both positive homogeneous and, there exists $E\in\Exp(|Q_{\xi_0}|)$ and $k>1$ for which $R_{\xi_0}$ is homogeneous with respect to $E/k$, $\widetilde{Q}_{\xi_0}$ is strongly subhomogeneous with respect to $E$ of order $2$, and $\widetilde{R}_{\xi_0}$ is strongly subhomogeneous with respect to $E/k$ of order $1$. In this case, we write $\mu_{\xi_0}=\mu_{|Q_{\xi_0}|}$.
\end{enumerate}
In either case, $\mu_{\xi_0}$ is said to be the homogeneous order associated to $\xi_0$ and $\alpha_{\xi_0}\in\mathbb{R}^d$ is said to be the drift associated to $\xi_0$.
\end{definition}

\noindent On account of the simplicity of positive homogeneous functions in one dimension, it is straightforward to verify that the notions of positive homogeneous type and imaginary homogeneous type coincide with  Thom\'{e}e's notions of Type 1 and Type 2 presented in Definition \ref{def:OneDTypes}, respectively, when $d=1$. In Definition 1.3 of \cite{RSC17}, a point $\xi_0\in\Omega(\phi)$ is said to be of positive homogeneous type for $\widehat{\phi}$ provided that
\begin{equation}\label{eq:GammaInTermsofP}
    \Gamma_{\xi_0}(\xi)=i\alpha_{\xi_0}\cdot\xi-P_{\xi_0}(\xi)-\widetilde{P}_{\xi_0}(\xi)
\end{equation}
for $\xi\in\mathcal{U}$ where $P_{\xi_0}$ is a positive homogeneous polynomial (in the sense of Remark \ref{rmk:ZdPosHomComp}) and $\widetilde{P}_{\xi_0}(\xi)=o(R_{\xi_0}(\xi))$ as $\xi\to 0$ where $R_{\xi_0}=\Re P_{\xi_0}$. To see this in the context of Definition \ref{def:Types}, let us write $P_{\xi_0}(\xi)=R_{\xi_0}(\xi)+iQ_{\xi_0}(\xi)$ and $\widetilde{P}_{\xi_0}(\xi)=\widetilde{R}_{\xi_0}(\xi)+i\widetilde{Q}_{\xi_0}(\xi)$ and, in this case, \eqref{eq:GammaExpansion} coincides with \eqref{eq:GammaInTermsofP}. If $\xi_0$ is of positive homogeneous type for $\widehat{\phi}$ in the sense of the definition above, it follows that $P_{\xi_0}$ is a complex-valued polynomial which is homogeneous with respect to $E$ (and so $\Exp(P_{\xi_0})$ contains $E\in\End(\mathbb{R}^d)$ for which $\{r^E\}$ is contracting) and $R_{\xi_0}=\Re P_{\xi_0}$ is positive definite. In view of Remark \ref{rmk:ZdPosHomComp}, this is consistent with (and perhaps generalizes) the assumption that $P_{\xi_0}$ is a positive homogeneous polynomial in \cite{RSC17}. Further, the assumption that $\widetilde{Q}_{\xi_0}$ and $\widetilde{R}_{\xi_0}$ are subhomogeneous with respect to $E$ guarantees that $\widetilde{P}_{\xi_0}(\xi)=o(R_{\xi_0}(\xi))$ as $\xi\to 0$ by virtue of Proposition \ref{prop:Subhomequivtolittleoh}. With these two observations, we see that our definition of a point $\xi_0$ being of positive homogeneous type, which is stated in terms of subhomogeneity, is consistent with that of \cite{RSC17}.\\

\begin{remark}\label{rem:AbsPosHomExponentSets}
In the case that $\xi_0$ is of imaginary homogeneous type, the assumption that $\abs{Q_{\xi_0}}$ is positive homogeneous guarantees that, for each $E\in\Exp(\abs{Q_{\xi_0}})$, $Q_{\xi_0}$ is homogeneous with respect to $E$. In fact, by virtue of Lemma \ref{lem:AbsPosHomExponentSets}, $\Exp(\abs{Q_{\xi_0}})=\Exp(Q_{\xi_0})$. We shall use this fact many times throughout this article. In particular, when $\xi_0\in\Omega(\phi)$ is of imaginary homogeneous type for $\widehat{\phi}$, we may always choose $E\in \Exp(Q_{\xi_0})$ which has $\tr E=\mu_{\abs{Q_{\xi_0}}}=\mu_{\xi_0}$.
\end{remark}

\noindent Our hypotheses for local limit theorems in this article will be stated under the assumption that, given $\phi\in\mathcal{S}_d$ with $\sup_{\xi}|\widehat{\phi}|=1$, every point $\xi_0\in\Omega(\phi)$ is of positive homogeneous or imaginary homogeneous type for $\widehat{\phi}$. In either case, because $R_{\xi_0}$ is positive definite and $\widetilde{R}_{\xi_0}(\xi)=o(R_{\xi_0}(\xi))$ as $\xi\to 0$ (by virtue of Proposition \ref{prop:Subhomequivtolittleoh}), $\xi_0$ is necessarily an isolated point of $\mathbb{T}^d$ and so it follows that $\Omega(\phi)$ is a finite set. Under these hypotheses, we define \textbf{the homogeneous order of $\phi$} to be the positive number
\begin{equation}\label{eq:HomOrderPhi}
\mu_{\phi}=\min_{\xi\in \Omega(\phi)}\mu_{\xi}>0.
\end{equation}
Of course, in the case that $\Omega(\phi)=\{\xi_0\}$, $\mu_{\phi}=\mu_{\xi_0}$.\\

\noindent As we did in one dimension, to state our local limit theorems, it is necessary to introduce the attractors appearing therein. In the case that $\xi_0\in\Omega(\phi)$ is of positive homogeneous type for $\widehat{\phi}$ with associated $R_{\xi_0}$ and $Q_{\xi_0}$, we consider the polynomial
\begin{equation*}
P_{\xi_0}(\xi):=R_{\xi_0}(\xi)+iQ_{\xi_0}(\xi)
\end{equation*}
and define $H_{P_{\xi_0}}^{(\cdot)}(\cdot):(0,\infty)\times\mathbb{R}^d\to\mathbb{C}$ by
\begin{equation}\label{eq:PosHomAttractor}
H_{P_{\xi_0}}^t(x)=\frac{1}{(2\pi)^d}\int_{\mathbb{R}^d}e^{-tP_{\xi_0}(\xi)}e^{-ix\cdot\xi}\,d\xi
\end{equation}
for $t>0$ and $x\in\mathbb{R}^d$. Like its one-dimensional analogue \eqref{eq:OneDPositiveHomAttractor}, the integral above converges absolutely for all $x\in\mathbb{R}^d$. In fact, as we will show in Section \ref{sec:Attractor}, $x\mapsto H_{P_{\xi_0}}^t(x)$ is a Schwartz function for each $t>0$ and, for each $E\in\Exp(P_{\xi_0})=\Exp(R_{\xi_0})\cap\Exp(Q_{\xi_0})$,
\begin{equation}\label{eq:PosHomAttractorScale}
H_{P_{\xi_0}}^t(x)=t^{-\mu_{\xi_0}}H_{P_{\xi_0}}^1(t^{-E^*}x)
\end{equation}
for $t>0$ and $x\in\mathbb{R}^d$ where $E^*$ is the adjoint of $E$; this scaling property generalizes \eqref{eq:OneDScale}. The article \cite{RSC17} studies extensively the properties of the attractors $H_{P_{\xi_0}}$ and their role as fundamental solutions to higher-order heat-type equations, i.e., generalized heat kernels. In particular, \cite{RSC17} establishes Gaussian-type estimates for $H_{P_{\xi_0}}$ and its derivatives; these estimates are given in terms of the Legendre-Fenchel transform of $R_{\xi_0}$. For more on this perspective, we encourage the reader to see the articles \cite{RSC17a} and \cite{RSC20}, both of which focus on related variable-coefficient heat-type equations and their associated heat kernel estimates. Let us assume now that $\xi_0\in\Omega(\phi)$ is of imaginary homogeneous type for $\widehat{\phi}$ with associated $Q_{\xi_0}$. In this case, taking our cues from \eqref{eq:OneDImagHomAttractor} and \eqref{eq:PosHomAttractor}, we expect that, at least formally, the associated attractor will be given by
\begin{equation}\label{eq:ImagHomAttractorFormal}
H_{iQ_{\xi_0}}^t(x)=\frac{1}{(2\pi)^d}\int_{\mathbb{R}^d}e^{-tiQ_{\xi_0}(\xi)}e^{-ix\cdot\xi}\,d\xi
\end{equation}
for $t>0$ and $x\in\mathbb{R}^d$. Unlike its positive homogeneous counterpart, the convergence of the above integral is a delicate matter. Since $iQ_{\xi_0}(\xi)$ is purely imaginary, this integral is oscillatory in nature and does not converge in the sense of Lebesgue (for any values of $t>0$ and $x\in\mathbb{R}^d$). Also, because there is no canonical notion of improper Riemann integral in $\mathbb{R}^d$ for $d>1$, extending \eqref{eq:OneDImagHomAttractor} and proving an associated generalization of \eqref{eq:LLTType2OneAttractor} are not straightforward tasks. In some sense, the major hurdle faced in this article is establishing the convergence of oscillatory integrals of the form \eqref{eq:ImagHomAttractorFormal} in an appropriate sense. We will interpret the integral in \eqref{eq:ImagHomAttractorFormal} as a \textit{renormalized integral} in the sense of C. B\"{a}r \cite{Baer12} (see Section \ref{sec:Attractor}) and, using the generalized polar-coordinate integration formula of \cite{BR21}, we prove that the integral converges for all $t>0$ and $x\in\mathbb{R}^d$ provided that $\mu_{\xi_0}<1$; this is Theorem \ref{thm:Attractor}. The theorem also shows that $(t,x)\mapsto H_{iQ_{\xi_0}}^t(x)$ is continuous and, for each $E\in \Exp(Q_{\xi_0})=\Exp(\abs{Q}_{\xi_0})$, 
\begin{equation*}
H^t_{iQ_{\xi_0}}(x)=t^{-\mu_{\xi_0}}H_{iQ_{\xi_0}}^1(t^{-E^*}x)
\end{equation*}
for $t>0$ and $x\in\mathbb{R}^d$ where $E^*$ is the adjoint of $E$. Though we suspect the attractors $H_{iQ_{\xi_0}}^t(x)$ are smooth, this remains an open question. With the attractors $H_{P_{\xi_0}}$ and $H_{iQ_{\xi_0}}$ in hand, we are ready to state our first local limit theorem; this result is new and extends the local limit theorems \eqref{eq:LLTType1OneAttractor} and \eqref{eq:LLTType2OneAttractor}.

\begin{theorem}\label{thm:LLTIntro}
Let $\phi\in\mathcal{S}_d$ be such that $\sup|\widehat{\phi}|=1$. Suppose that $\Omega(\phi)=\{\xi_0\}$ and that $\xi_0$ is of positive homogeneous type or imaginary homogeneous type for $\widehat{\phi}$ with drift $\alpha_{\xi_0}$ and homogeneous order $\mu_{\xi_0}$. In either case, note that $\mu_{\phi}=\mu_{\xi_0}$. When $\xi_0$ is of imaginary homogeneous type for $\widehat{\phi}$, we assume additionally that $\mu_{\phi}=\mu_{\xi_0}<1$.
\begin{enumerate}
\item In the case that $\xi_0$ is of positive homogeneous type for $\widehat{\phi}$, let $H_{P_{\xi_0}}$ be as given in the previous paragraph. Then,
\begin{equation*}
\phi^{(n)}(x)=\widehat{\phi}(\xi_0)^n e^{-ix\cdot\xi_0}H_{P_{\xi_0}}^n(x-n\alpha_{\xi_0})+o(n^{-\mu_\phi})
\end{equation*}
uniformly for $x\in\mathbb{Z}^d$.
\item In the case that $\xi_0$ is of imaginary homogeneous type for $\widehat{\phi}$, let $H_{iQ_{\xi_0}}$ be as discussed above (and defined precisely in Section \ref{sec:Attractor}). Then, for each compact set $K\subseteq\mathbb{R}^d$ and $E\in\Exp(Q_{\xi_0})=\Exp(\abs{Q_{\xi_0}})$, 
\begin{equation*}
\phi^{(n)}(x)=\widehat{\phi}(\xi_0)^ne^{-ix\cdot\xi_0}H_{iQ_{\xi_0}}^n(x-n\alpha_{\xi_0})+o(n^{-\mu_\phi})
\end{equation*}
uniformly for $x\in \left(n\alpha_{\xi_0}+n^{E^*}(K)\right)\bigcap\mathbb{Z}^d$ where $E^*$ is the adjoint of $E$.
\end{enumerate}
\end{theorem}

\begin{example}\label{ex:TwoDEx1}
Consider the function $\phi : \mathbb{Z}^2 \to \mathbb{C}$ defined by 
\begin{equation*}
    \phi(x,y) = 
    \frac{1}{768}\times
    \begin{cases}
    602 - 112i &(x,y) = (0,0)\\
    56 + 32i   &(x,y) = (0,\pm 1)\mbox{ or }(-1,0)\\
    72 + 32i   &(x,y) = (1,0)\\
    -28 - 8i   &(x,y) = (0,\pm 2)\\
    -16        &(x,y) = (\pm 2,0)\\
    56         &(x,y) = (0,\pm 3)\\
    -1         &(x,y) = (0,\pm 4)\\
    4          &(x,y) = (-1,\pm 1)\\
    -4         &(x,y) = (1,\pm 1)\\
    0          &\text{otherwise}.
    \end{cases}
\end{equation*}
This example was considered in \cite{BR21} (specifically, Example 6 of Subsection 3.1) wherein it was shown that $\sup_{\xi}|\widehat{\phi}(\xi)|=1$, $\Omega(\phi)=\{\xi_0\}$, $\widehat{\phi}(\xi_0)=1$, and $\xi_0=(0,0)$ is of imaginary homogeneous type of $\widehat{\phi}$ with drift $\alpha=(0,0)$, homogeneous order $\mu_\phi=\mu_{\xi_0}=3/4<1$, and associated polynomial
\begin{equation*}
    Q_{\xi_0}(\xi)=Q(\eta,\gamma)=\frac{1}{96}\left(4\eta^2-\eta\gamma^2+\gamma^4\right)
\end{equation*}
defined for $\xi=(\eta,\gamma)\in\mathbb{R}^2$. For completeness, these details are verified in Section \ref{sec:Examples}. Thus, Theorem \ref{thm:LLTIntro} guarantees that, for any compact set $K\subseteq\mathbb{R}^d$, 
\begin{equation}\label{eq:TwoDEx1_1}
    \phi^{(n)}(x)=H_{iQ}^n(x)+o(n^{-3/4})
\end{equation}
uniformly for $x\in n^{E}(K)=\{(n^{1/2}y_1,n^{1/4}y_2)\in\mathbb{R}^d:(y_1,y_2)\in K\}$ where $E=E^*\in \Exp(Q)$ has the standard matrix representation $\diag(1/2,1/4)$. This local limit theorem is illustrated in Figure \ref{fig:TwoDEx1} which depicts $\Re(\phi^{(n)}(x))$ and approximations of $\Re(H_{iQ}^n(x))$ for $n=2000$ and $n=4000$. We remark that the small irregularities present in Figures \ref{fig:TwoDEx1_H_2000} and \ref{fig:TwoDEx1_H_4000} are due to error in the numerical integration of the oscillatory integrals defining the attractor.

\begin{figure}[!htb]
\begin{center}
\resizebox{\textwidth}{!}{	
	    \begin{subfigure}{0.5\textwidth}
		\includegraphics[width=\textwidth]{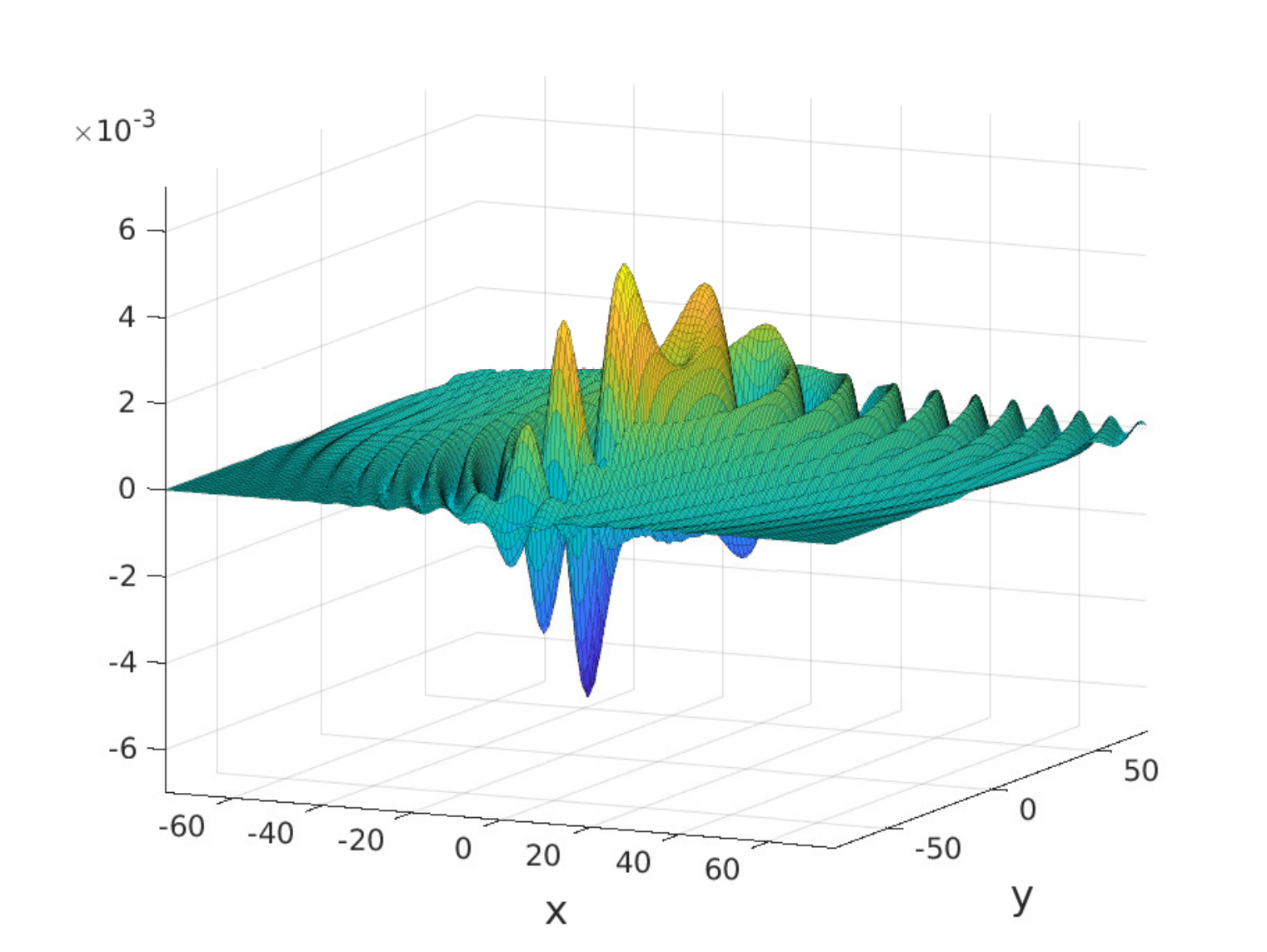}
		\caption{$\Re(\phi^{(n)})$ for $n=2000$}
		\label{fig:TwoDEx1_Conv_2000}
	    \end{subfigure}
	    \begin{subfigure}{0.5\textwidth}
		\includegraphics[width=\textwidth]{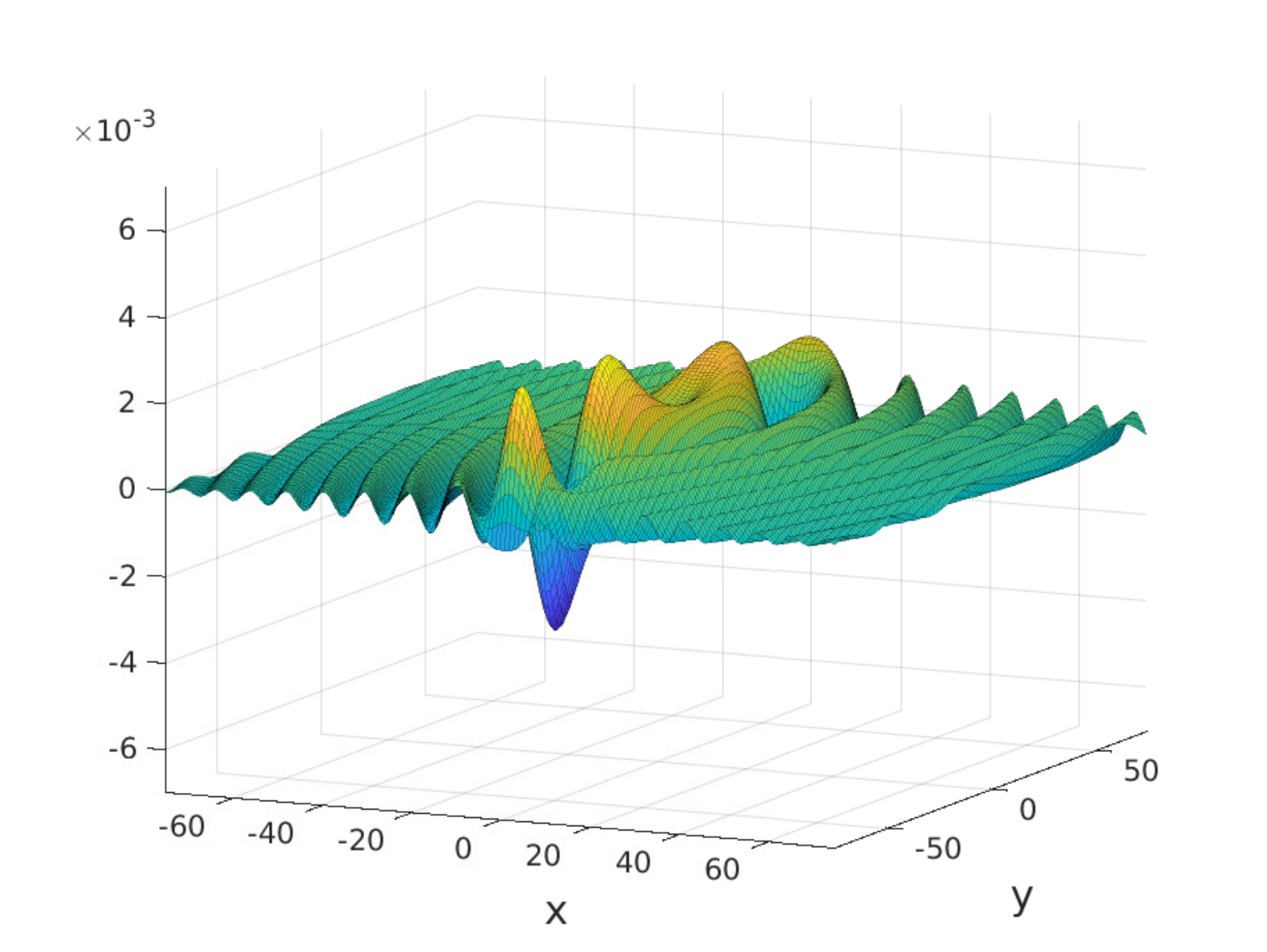}
		\caption{$\Re(\phi^{(n)})$ for $n=4000$}
		\label{fig:TwoDEx1_Conv_4000}
	    \end{subfigure}}
\resizebox{\textwidth}{!}{	
	    \begin{subfigure}{0.5\textwidth}
		\includegraphics[width=\textwidth]{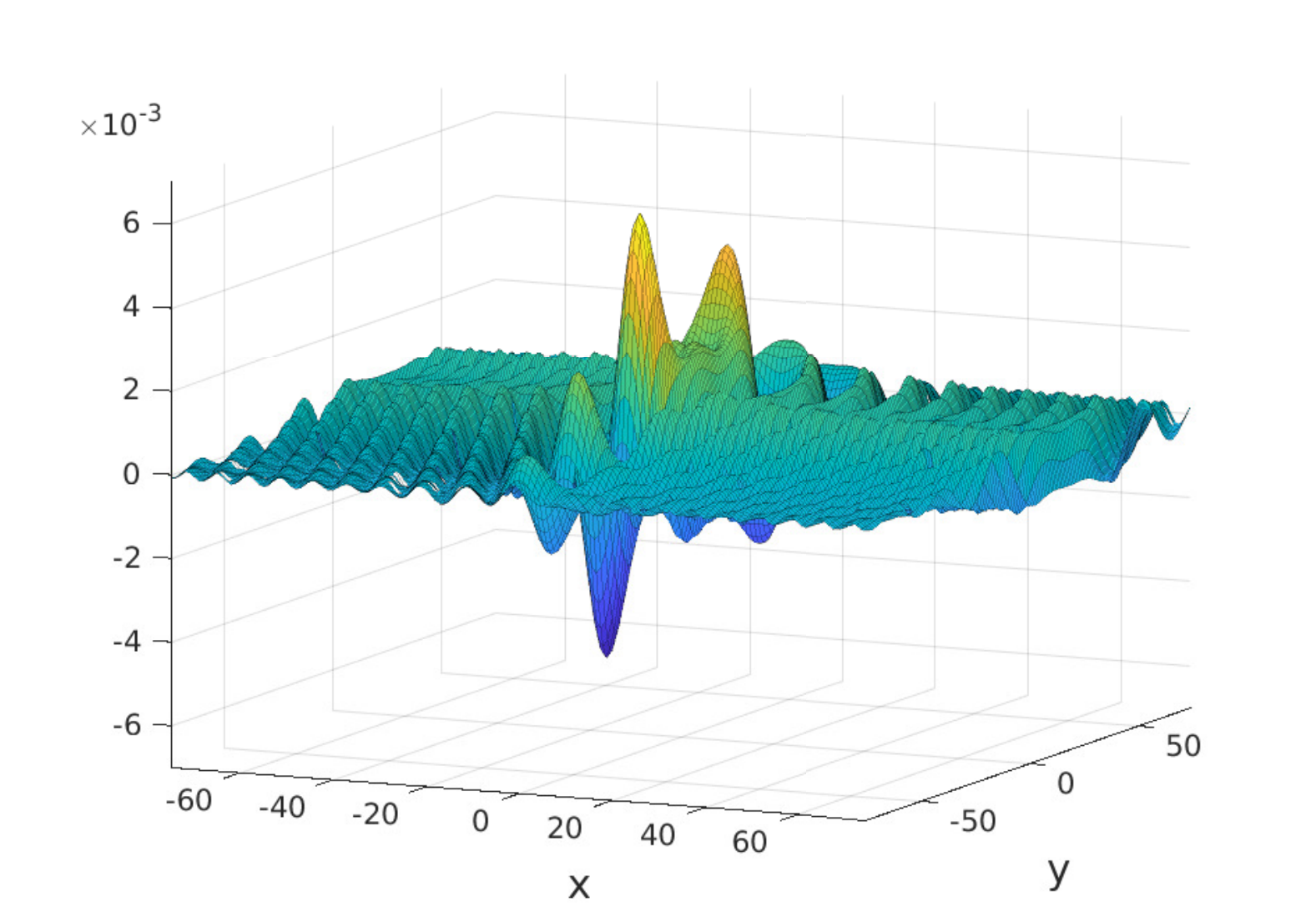}
		\caption{$\Re(H_{iQ}^n)$ for $n=2000$}
		\label{fig:TwoDEx1_H_2000}
	    \end{subfigure}
	    \begin{subfigure}{0.5\textwidth}
		\includegraphics[width=\textwidth]{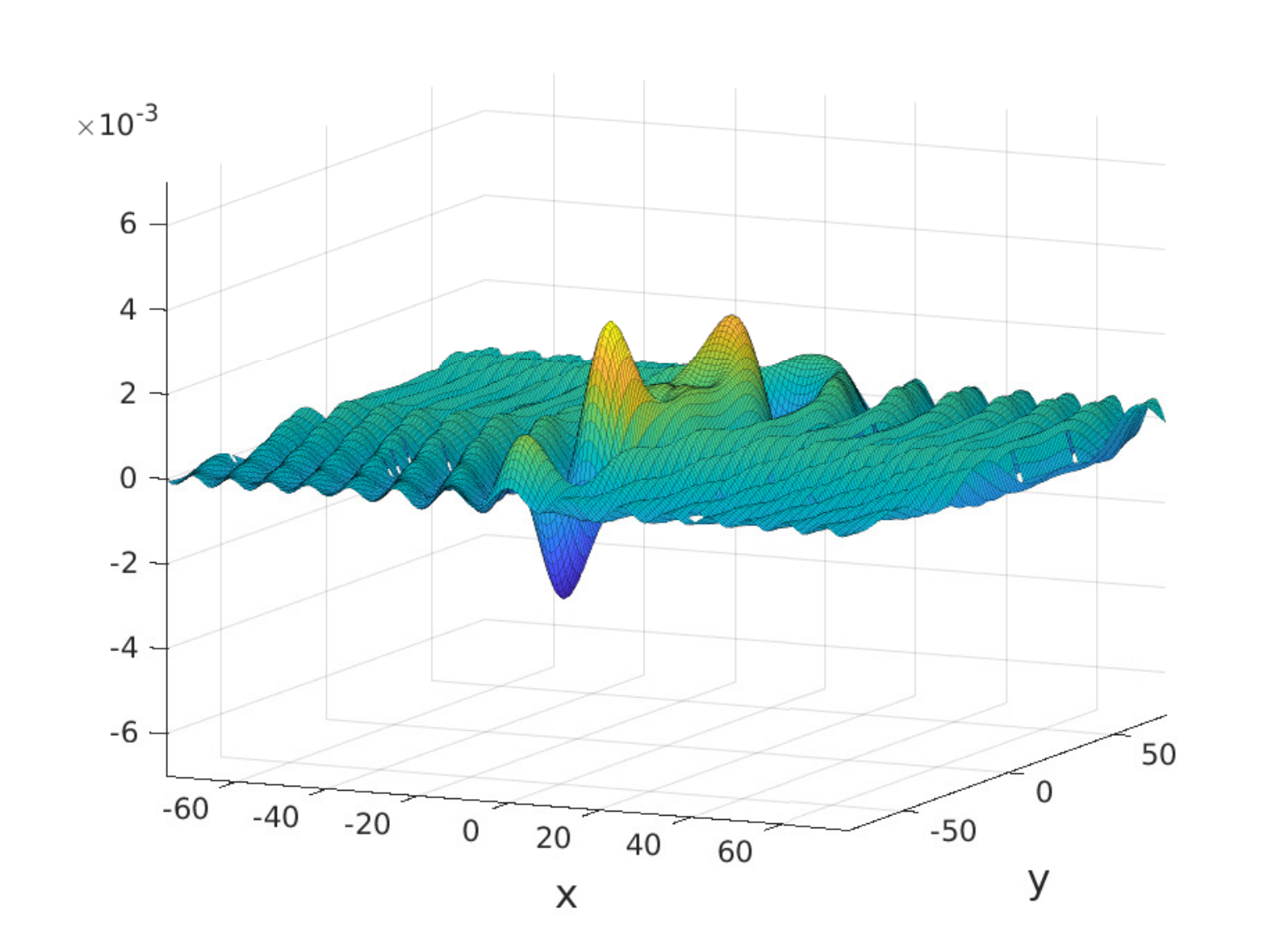}
		\caption{$\Re(H_{iQ}^n)$ for $n=4000$}
		\label{fig:TwoDEx1_H_4000}
	    \end{subfigure}
	    }
\caption{The graphs of $\Re(\phi^{(n)})$ and $\Re(H_{iQ}^n)$ for $n=2000,4000$. }
\label{fig:TwoDEx1}
\end{center}
\end{figure}

As we will show in Section \ref{sec:Attractor}, 
\begin{equation*}
    H_{iQ}^n(x)=n^{-3/4}H_{iQ}^1(n^{-E}x)=n^{-3/4}H^1_{iQ}(n^{-1/2}x_1,n^{-1/4}x_2)
\end{equation*}
for $x=(x_1,x_2)\in\mathbb{R}^2$. Consequently, the local limit theorem \eqref{eq:TwoDEx1_1} can be equivalently stated as
\begin{equation*}
    \phi^{(n)}(x)=n^{-3/4}H_{iQ}^1(n^{-1/2}x_1,n^{-1/4}x_2)+o(n^{-3/4})
\end{equation*}
uniformly for $x=(x_1,x_2)\in n^{E}(K)$. We remark that this is consistent with the following sup-norm type estimate established in \cite{BR21} (see (26) of Example 6 therein): There exists $C'>0$ for which
\begin{equation}\label{eq:TwoDEx1_2}
    \abs{\phi^{(n)}(x)}\leq\frac{C'}{n^{3/4}}
\end{equation}
for $n\in\mathbb{N}_+$ and $x\in K$. In Section \ref{sec:SupNormEst}, we will established the following matching lower estimate: There is some $C>0$ for which
\begin{equation}\label{eq:TwoDEx1_3}
\frac{C}{n^{3/4}}\leq \|\phi^{(n)}\|_\infty=:\sup_{x\in\mathbb{Z}^d}|\phi^{(n)}(x)|
\end{equation}
for all $n\in\mathbb{N}_+$.
\end{example}

\noindent Let us briefly discuss the condition $\mu_\phi<1$ appearing in Theorem \ref{thm:LLTIntro}. First, in one dimension, it is automatically satisfied because $\mu_\phi=1/m$ for some integer $m\geq 2$. Moving into higher dimensions, things become problematic. As the reader will see in Section \ref{sec:Attractor},  the renormalized integral defining $H_{iQ_{\xi_0}}$ need not converge when $\mu_\phi\geq 1$ (see the discussion following Theorem \ref{thm:Attractor}). Still, in some special cases, attractors corresponding to points of imaginary homogeneous type can be defined by other means and local limit theorems can be established when $\mu_\phi\geq 1$. This is discussed in Section \ref{sec:Discussion}. Curiously, an analogous condition appears in the study of higher-order partial differential equations. In the context of heat kernel estimates for higher-order semi-elliptic operators, this $\mu<1$ condition appears in Lemma 5.4 and Proposition 6.6 of \cite{RSC20}. As the article \cite{RSC20} explains, the threshold $\mu=1$ corresponds to $d=2m$ for $2m$th-order elliptic operators on $\mathbb{R}^d$. We refer the reader to the survey \cite{DaviesSurvey} which highlights the role that this dimension-order threshold plays in the theory of such operators. \\

\noindent For simplicity of presentation, we have stated Theorem \ref{thm:LLTIntro} with the assumption that $\Omega(\phi)$ is a singleton. In Section \ref{sec:LLT}, we state and prove a general theorem which allows for $\Omega(\phi)$ to consist of more than one point; this is Theorem \ref{thm:LLTMain}. Our theorem recaptures Theorem 1.6 of \cite{RSC17}, a result that assumes that every point of $\Omega(\phi)$ is of positive homogeneous type for $\widehat{\phi}$. What is new in Theorem \ref{thm:LLTMain} (and Theorem \ref{thm:LLTIntro}) is the consideration of points of imaginary homogeneous type for $\widehat{\phi}$ and this consideration, in its essence, is the major task treated in this article. While our results are not able to treat the full generality in which all points of $\Omega(\phi)$ are either of positive homogeneous or imaginary homogeneous type for $\widehat{\phi}$ (with no other restrictions), our results partially extend the local limit theorems of \cite{RSC15} and do so in the spirit of that article. \\

\noindent This article proceeds as follows. In Section \ref{sec:Attractor}, we first prove some basic facts concerning attractors of the form \eqref{eq:PosHomAttractor} corresponding to points of positive homogeneous type. We then discuss the renormalized integral of C. B\"{a}r \cite{Baer12} and establish some basic properties relevant to our context. With the renormalized integral in hand, we then introduce the attractors corresponding to points of imaginary homogeneous type and establish some of their basic properties; this is Theorem \ref{thm:Attractor}. In Section \ref{sec:LLT}, we first treat two basic lemmas pertaining to points of positive homogeneous type, Lemmas \ref{lem:LocalLimitPosHom} and \ref{lem:SupNormEstPositiveHom}. We then turn our focus to points of imaginary homogeneous type and, following three technical lemmas concerning oscillatory integrals, we establish a ``local limit lemma" for points of imaginary homogeneous type which parallels Lemma \ref{lem:LocalLimitPosHom} (and Lemma 4.3 of \cite{RSC17}); this is Lemma \ref{lem:LocalLimitImagHom}. We then state and prove our main result, Theorem \ref{thm:LLTMain}. Using Theorem \ref{thm:LLTMain} and some of the machinery developed within its proof, we prove Theorem \ref{thm:LLTIntro}. Section \ref{sec:SupNormEst} is a short section on sup-norm type estimates. Making use of the relevant analysis developed in Section 4 of \cite{RSC17} and Theorem \ref{thm:LLTMain}, we obtain Theorem \ref{thm:SupNormEst}, a result on sup-norm type estimates for convolution powers which provides a matching lower bound for the estimate of Theorem 3.2 of \cite{BR21} as well as (slightly) extending that result. In Section \ref{sec:Examples}, we present examples demonstrating the conclusion of Theorem \ref{thm:LLTMain} and provide the details for Example \ref{ex:TwoDEx1}.  Section \ref{sec:Discussion} is dedicated to discussing a number of open questions and describing further areas of research. Finally, the appendix contains some basic results on one-parameter contracting groups, positive homogeneous functions and subhomogeneous functions. \\

\noindent \textbf{Notation:} In this article, we denote by $\mathcal{M}_d$ the $\sigma$-algebra of Lebesgue measurable set of $\mathbb{R}^d$ and, for $A\in\mathcal{M}_d$, $\chi_A$ is the characteristic function of $A$. We denote by $m=m_d$ the Lebesgue measure on $\mathbb{R}^d$ and write $dm=m(d\xi)=d\xi$. For $X\in\mathcal{M}_d$ and $1\leq p\leq \infty$, $L^p(X)=L^p(X,m)$ will denote the usual Lebesgue space equipped with its norm $\|\cdot\|_{L^p(X)}$. When $k\in\mathbb{N}$ and $X\subseteq\mathbb{R}^d$ is non-empty and open (or, generally, when the following concept makes sense), we will denote by $C^k(X)$ the set of $k$-times continuously differentiable complex-valued functions on $X$. Of course, $C^0(X)$ and $C^\infty(X)$ are, respectively, the set of continuous and smooth functions on $X$. For $0\leq k\leq\infty$, the notation $C^k(X;\mathbb{R})$ will denote the subclass of function $C^k(X)$ which are real valued. The set of locally integrable functions on $\mathbb{R}^d$ will be denoted by $L^1_{\mbox{\tiny loc}}(\mathbb{R}^d)$; this is set of complex-valued measurable functions $f$ for which, given any compact set $K\subseteq\mathbb{R}^d$, $\xi\mapsto f(\xi)\chi_K(\xi)$ belongs to $L^p(\mathbb{R}^d)$. Finally, $\mathcal{S}(\mathbb{R}^d)$ will denote the Schwartz class of complex-valued functions on $\mathbb{R}^d$.

\section{The Attractors}\label{sec:Attractor}

\noindent In this section, we study the attractors \eqref{eq:PosHomAttractor} and \eqref{eq:ImagHomAttractorFormal} discussed in the introduction. Our first task is to establish some basic facts about the attractor \eqref{eq:PosHomAttractor} corresponding to a point $\xi_0$ of positive homogeneous type. To this end, we present the following proposition which ensures, in particular, that \eqref{eq:PosHomAttractor} is smooth and exhibits the scaling property \eqref{eq:PosHomAttractorScale}. The proposition is stated in the context of positive homogeneous functions (that is, it does not assume that $P=R+iQ$ is a polynomial), however, thanks to the theory developed in \cite{BR21}, working in this level of generality is natural. We remark that a similar statement can be found as Proposition 2.6 of \cite{RSC17}.

\begin{proposition}\label{prop:PosHomKernel}
Let $R:\mathbb{R}^d\to\mathbb{R}$ be a positive homogeneous function and let $Q:\mathbb{R}^d\to\mathbb{R}$ be continuous and such that $\Exp(R)\cap\Exp(Q)$ is non-empty. Set $P(\xi)=R(\xi)+iQ(\xi)$ for $\xi\in\mathbb{R}^d$, observe that $\Exp(P)=\Exp(R)\cap\Exp(Q)$, and set $\mu_P=\mu_R$. Then $e^{-tP(\xi)}\in L^1(\mathbb{R}^d)$ for each $t>0$ and we may therefore define
\begin{equation}\label{eq:PosHomKernel}
H_P^t(x)=\frac{1}{(2\pi)^d}\int_{\mathbb{R}^d}e^{-tP(\xi)}e^{-ix\cdot\xi}\,d\xi
\end{equation}
for $t>0$ and $x\in\mathbb{R}^d$. There holds the following:
\begin{enumerate}
\item\label{item:PosHomKernelSmooth}  As a function on $(0,\infty)\times \mathbb{R}^d$, $H_P=H_P^{(\cdot)}(\cdot)$ is smooth. If, additionally, $R$ and $Q$ are polynomials, then, for each $t>0$, $H_P^t\in \mathcal{S}(\mathbb{R}^d)$.
\item\label{item:PosHomKernelScaling} For any $E\in\Exp(P)$, 
\begin{equation*}
H_P^t(x)=t^{-\mu_P}H_P^1(t^{-E^*}x)
\end{equation*}
for $t>0$ and $x\in\mathbb{R}^d$. 
\item\label{item:PosHomKernelOnDiagional} We have
\begin{equation*}
\abs{H_P^t(x)}\leq t^{-\mu_P}H_R(0)=\frac{t^{-\mu_P}}{(2\pi)^d} m(B_R) \Gamma(\mu_P+1)
\end{equation*}
for all $x\in\mathbb{R}^d$ and $t>0$; here, $B_R=\{\xi\in\mathbb{R}^d:R(\xi)<1\}$ (and has finite measure by virtue of Proposition \ref{prop:CharofPosHom}), $\Gamma(\cdot)$ is Euler's Gamma function, and $H_R$ is defined in precisely the same way at $H_P$; they coincide, of course, when $Q\equiv 0$.
\end{enumerate}
\end{proposition}

\begin{proof}
We first claim that, for each $k\in\mathbb{N}$ and multi-index $\alpha=(\alpha_1,\alpha_2,\dots,\alpha_d)\in\mathbb{N}^d$, 
\begin{equation}\label{eq:PosHomKernelProof1}
\sup_{\xi\in\mathbb{R}^d}\abs{\xi^\alpha P(\xi)^k e^{-tP(\xi)}}<\infty.
\end{equation}
for each $t>0$; here, $\xi^\alpha=\xi_1^{\alpha_1}\xi_2^{\alpha_2}\cdots\xi_d^{\alpha_d}$ for $\xi=(\xi_1,\xi_2,\dots,\xi_d)\in\mathbb{R}^d$. From this claim, the absolute integrability of $e^{-tP(\xi)}$ follows immediately. Also, standard arguments (e.g., those which show that the Fourier transform exchanges decay at infinity for smoothness) show that, by virtue of \eqref{eq:PosHomKernelProof1}, one may differentiate (in both $t$ and $x$) through the integral sign in \eqref{eq:PosHomKernel} as many times as one likes. In this way, \eqref{eq:PosHomKernelProof1} ensures that $H_P\in C^\infty((0,\infty)\times\mathbb{R}^d)$. To prove \eqref{eq:PosHomKernelProof1},  in view of the fact that $\xi\mapsto\xi^{\alpha}P(\xi)^k e^{-tP(\xi)}$ is continuous and $\Re(P)=R$, it suffices to prove that
\begin{equation*}
\sup_{\xi\in\mathbb{R}^d/B_R}\abs{\xi^{\alpha}P(\xi)^k e^{-tR(\xi)}}<\infty
\end{equation*}
where, as given in the statement of the proposition, $B_R=\{\xi\in\mathbb{R}^d:R(\xi)<1\}$. For $E\in\Exp(P)\subseteq\Exp(R)$, the fact that $\{r^E\}$ is contracting guarantees that $\mathbb{R}^d\setminus B_R=\{r^E\eta:r\geq 1, \eta\in S_R\}$ where $S_R$ is the unital level set of $R$. From this it follows that
\begin{equation*}
\sup_{\xi\in\mathbb{R}^d\setminus B_R}\abs{\xi^\alpha P(\xi)^k e^{-tR(\xi)}}=\sup_{r\geq 1,\eta\in S_R}\abs{(r^E\eta)^\alpha P(r^E\eta)^k e^{-tR(r^E\eta)}}=\sup_{r\geq 1,\eta\in S_R}\abs{(r^E\eta)^\alpha r^k P(\eta)^k e^{-tr}}.
\end{equation*}
Now, for $r\geq 1$, $\|r^E\|\leq r^{\|E\|}$ and therefore $\abs{(r^E\eta)^\alpha}\leq \abs{r^E\eta}^{|\alpha|}\leq \abs{\eta}^{|\alpha|}r^{|\alpha|\|E\|}$ where $|\alpha|=\alpha_1+\alpha_2+\cdots+\alpha_d$. Given that $P$ is continuous, $S_R$ is compact, and $t>0$, we conclude that
\begin{equation*}
\sup_{\xi\in\mathbb{R}^d\setminus B_R}\abs{\xi^\alpha P(\xi)^k e^{-tP(\xi)}}\leq \sup_{r\geq 1,\eta\in S_R}\abs{\eta}^{|\alpha|}\abs{P(\eta)}^k r^{\|E\||\alpha|+k}e^{-tr}<\infty,
\end{equation*}
as was asserted. 

In the case that $R$ and $Q$ are polynomials, for each pair of multi-indices $\alpha,\beta\in\mathbb{N}^d$, 
\begin{equation*}
\xi^\alpha D^\beta\left(e^{-tP(\xi)}\right)=U_{\alpha,\beta}(t,\xi)e^{-tP(\xi)}
\end{equation*}
where $U_{\alpha,\beta}(t,\xi)$ is a polynomial in $t$ and $\xi$; here $D^{\beta}=(i\partial_{\xi_1})^{\beta_1}(i\partial_{\xi_2})^{\beta_2}\cdots(i\partial_{\xi_d})^{\beta_d}$. By a similar argument to that given above, we find that
\begin{equation*}
\sup_{\xi\in \mathbb{R}^d\setminus B_R}\abs{\xi^\alpha D^{\beta}\left(e^{-tP(\xi)}\right)}=\sup_{r\geq 1,\eta\in S_R}\abs{U(t,r^E\eta)e^{-tr}}<\infty
\end{equation*}
for each pair of multi-indices $\alpha$ and $\beta$ and therefore $\xi\mapsto e^{-tP(\xi)}\in \mathcal{S}(\mathbb{R}^d)$. Because the Fourier transform is an automorphism of $\mathcal{S}(\mathbb{R}^d)$, it follows that $H_P^t\in\mathcal{S}(\mathbb{R}^d)$ for each $t>0$ and this concludes the proof of Item \ref{item:PosHomKernelSmooth}.

Throughout the remainder of the proof, we shall write $\mu=\mu_P$ and note that $\mu=\mu_R=\tr E$ for any $E\in \Exp(R)$. To see Item \ref{item:PosHomKernelScaling}, observe that, for any $E\in\Exp(P)=\Exp(R)\cap\Exp(Q)$, $t>0$, and $x\in\mathbb{R}^d$,
\begin{equation*}
H_P^t(x)=\frac{1}{(2\pi)^d}\int_{\mathbb{R}^d}e^{-tP(\xi)}e^{-ix\cdot\xi}\,d\xi=\frac{1}{(2\pi)^d}\int_{\mathbb{R}^d}e^{-P(t^E\xi)}e^{-ix\cdot\xi}\,d\xi.
\end{equation*}
Upon making the change of variables $\xi=t^{-E}\zeta$, for any $t>0$ and $x\in\mathbb{R}^d$, we have
\begin{eqnarray*}
H_P^t(x)&=&\frac{1}{(2\pi)^d}\int_{\mathbb{R}^d}e^{-P(\zeta)}e^{-ix\cdot (t^{-E}\zeta)}\det(t^{-E})\,d\zeta\\
&=&\frac{1}{(2\pi)^d}\int_{\mathbb{R}^d}e^{-P(\zeta)}e^{-i(t^{-E^*}x)\cdot\zeta}t^{-\mu}\,d\zeta\\
&=&t^{-\mu}H_P^1(t^{-E^*}x)
\end{eqnarray*}
where we have used the fact that $\det(t^{-E})=t^{-\tr E}=t^{-\mu}$ and $(t^{-E})^*=t^{-E^*}$. This proves Item \ref{item:PosHomKernelScaling}. 

To prove Item \ref{item:PosHomKernelOnDiagional}, we first appeal to the previous item to see that, for any $t>0$ and $x\in\mathbb{R}^d$,
\begin{equation*}
\abs{H_P^t(x)}=t^{-\mu}\abs{H_P^1(t^{-E^*}x)}\leq \frac{t^{-\mu}}{(2\pi)^d}\int_{\mathbb{R}^d}\abs{e^{-P(\xi)}e^{-i(t^{-E^*}x)\cdot\xi}}\,d\xi=\frac{t^{-\mu}}{(2\pi)^d}\int_{\mathbb{R}^d}e^{-R(\xi)}\,d\xi;
\end{equation*}
Thus,
\begin{equation*}
\abs{H_P^t(x)}\leq \frac{t^{-\mu}}{(2\pi)^d}\int_{\mathbb{R}^d}e^{-R(\xi)}\,d\xi=t^{-\mu}H_R^1(0)
\end{equation*}
for all $t>0$ and $x\in\mathbb{R}^d$. It remains to compute $H_R^1(0)$. To this end, we appeal to the polar-coordinate integration formula (Theorem 1.3) of \cite{BR21}. For the positive homogeneous function $R$, the theorem gives us a Radon measure $\sigma_R$ on $S_R$ for which, given any $f\in L^1(\mathbb{R}^d)$ and $E\in\Exp(R)$,
\begin{equation*}
\int_{\mathbb{R}^d}f(\xi)\,d\xi=\int_{S_R}\int_0^\infty f(r^E\eta)r^{\mu-1}\,dr\sigma_R(d\eta).
\end{equation*}
Applying this to $f(\xi)=e^{-R(\xi)}$, we have
\begin{eqnarray*}
H_R^1(0)&=&\frac{1}{(2\pi)^d}\int_{S_R}\int_0^\infty e^{-R(r^E\eta)}r^{\mu-1}\,dr\sigma_R(d\eta)\\
&=&\frac{1}{(2\pi)^d}\int_{S_R}\int_0^\infty e^{-r}r^{\mu-1}\,dr\,\sigma_R(d\eta)\\
&=&\frac{\sigma_R(S_R)}{(2\pi)^d}\int_0^\infty e^{-r}r^{\mu-1}\,dr\\
&=&\frac{\mu \cdot m(B_R)}{(2\pi)^d}\Gamma(\mu)\\
&=&\frac{\Gamma(\mu+1)}{(2\pi)^d}m(B_R)
\end{eqnarray*}
where we have noted that $\sigma_R(S_R)=\mu\cdot m(B_R)$ by virtue of Item 3 of Theorem 1.3 of \cite{BR21}, recognized the integral representation of the Gamma function, and made use of the fundamental identity, $\mu\Gamma(\mu)=\Gamma(\mu+1)$.
\end{proof}

\noindent We now turn our attention to the attractors corresponding to points $\xi_0$ of imaginary homogeneous type. As we discussed in the introduction, we must first discuss a theory for the convergence of oscillatory integrals given formally by \eqref{eq:ImagHomAttractorFormal}. The integral we introduce below can be seen as a special case the renormalized integral of \cite{Baer12}. While still having its limitations, this renormalized integral is broad enough to handle a large variety of possible $Q$s and sufficiently robust to make tractable our analysis. \\

\noindent Given $X\in\mathcal{M}_d$, a family of measurable sets $\mathcal{A}=\{\mathcal{O}_{\tau}\subseteq X:\tau>0\}\subseteq\mathcal{M}_d$ is said to be an \textbf{approximating family of $X$} if $\mathcal{A}$ is nested increasing and covers $X$, i.e., $\mathcal{O}_{\tau_1}\subseteq\mathcal{O}_{\tau_2}$ whenever $\tau_1\leq \tau_2$, and $X=\cup_{\tau>0}\mathcal{O}_{\tau}$. 

\begin{definition}\label{def:RenormIntegral}
Let $X\in\mathcal{M}_d$ and suppose that $\mathcal{A}=\{\mathcal{O}_{\tau}:\tau>0\}$ is an approximating family of $X$.
\begin{enumerate}
\item  Given a measurable function $f:X\to\mathbb{C}$, we say that the renormalized integral of $f$ over $X$ associated to $\mathcal{A}$ converges if the limit
\begin{equation}\label{eq:RenormIntegral}
\lim_{\tau\to\infty}\int_{\mathcal{O}_{\tau}}f\,dm
\end{equation}
exists. In this case, the renormalized integral of $f$ over $X$ associated to $\mathcal{A}$ is defined to be the value of the limit \eqref{eq:RenormIntegral} and denoted by
\begin{equation*}
\fint_{X} f\,d_\mathcal{A}\hspace{1cm}\mbox{or}\hspace{1cm}\fint_{X} f(\xi)\,d_\mathcal{A}\xi.
\end{equation*}
When the context is clear, we will omit the vocabulary ``over $X$ associated to $\mathcal{A}$". 
\item Let $Y$ be a non-empty set and, for a function $g=g_{(\cdot)}(\cdot):Y\times X\to\mathbb{C}$, and suppose that $\xi\mapsto g_y(\xi)$ is measurable for each $y\in Y$. If, the renormalized integral $\fint_{X}g_y\,d_{\mathcal{A}}=\fint_{X}g_y(\xi)\,d_{\mathcal{A}}\xi$ converges for each $y\in Y$ and,  for all $\epsilon>0$, there exists $\tau_0>0$ for which
\begin{equation*}
\abs{\fint g_y\,d_\mathcal{A}-\int_{\mathcal{O}_{\tau}} g_y\,dm}<\epsilon
\end{equation*}
for all $y\in Y$ and $\tau\geq \tau_0$, we say that the renormalized integral of $g_y$ converges uniformly on $Y$. We also say that the renormalized integral of $g_y$ converges uniformly for $y\in Y$.
\end{enumerate}
\end{definition}

\noindent As mentioned above, the above definition is a special case\footnote{Looking at Definition \ref{def:RenormIntegral} from the perspective of \cite{Baer12}, the relevant family of measure spaces is $\Omega=\{\mathcal{O}_{\tau},d\xi\vert_{\mathcal{O}_{\tau}}\}_{\tau>0}$ which is indexed by the directed set of positive real numbers with its usual ordering.} of the renormalized integral of \cite{Baer12}, a notion which itself recaptures several instances of Cauchy's principal value.  Some simple observations following Definition \ref{def:RenormIntegral} are presented in the following examples.
\begin{example} In one dimension, the improper integral
\begin{equation*}
\int_0^\infty f(\xi)\,d\xi=\lim_{s\to\infty}\int_0^sf(\xi)\,d\xi,
\end{equation*}
coincides with the renormalized integral of $f$ over $\mathbb{R}$ and associated to the family $\mathcal{A}=\{[0,s),s>0\}$. As noted in the introduction (c.f., the discussion surrounding \eqref{eq:OneDImagHomAttractor}), this integral is used to define the attractors in \cite{RSC15}. 
\end{example}
\begin{example}
Given $X\in\mathcal{M}_d$ and $f\in L^1(X)$, the Lebesgue integral of $f$ over $X$ coincides with the renormalized integral of $f$ over $X$ associated to any (and so every) approximating family $\mathcal{A}$ of $X$. In other words, given any $f\in L^1(X)$ and approximating family $\mathcal{A}$ of $X$, the renormalized integral of $f$ over $X$ associated to $\mathcal{A}$ converges and
\begin{equation*}
\fint_{X}f(\xi)\,d_{\mathcal{A}}\xi=\int_{X}f(\xi)\,d\xi.
\end{equation*}
In fact, the absolute Lebesgue integrability of a measurable function $f$ can be characterized by the convergence of the renormalized integral of $\abs{f(\xi)}$ over $X$ associated to any (and every) approximating family $\mathcal{A}$ of $X$. These assertions can be seen straightforwardly as consequences of the monotone and dominated convergence theorems.

As an application, let $P=R+iQ$ be as in the statement of Proposition \ref{prop:PosHomKernel} and consider the approximating family $\mathcal{A}=\{\mathcal{O}_\tau:\tau>0\}$ of $\mathbb{R}^d$ where, for each $\tau>0$, $\mathcal{O}_\tau=\{\xi\in\mathbb{R}^d:R(\xi)<\tau\}$. Then the attractor $H_P^t$ of Proposition \ref{prop:PosHomKernel} is equivalently given by
\begin{equation*}
H_P^t(x)=\frac{1}{(2\pi)^d}\fint_{\mathbb{R}^d}e^{-tP(\xi)}e^{-ix\cdot\xi}\,d_{\mathcal{A}}\xi=\lim_{\tau\to\infty}\frac{1}{(2\pi)^d}\int_{\mathcal{O}_\tau}e^{-tP(\xi)}e^{-ix\cdot\xi}\,d\xi
\end{equation*}
for $t>0$ and $x\in\mathbb{R}^d$. It is easy to see that, for each $t_0>0$, this renormalized integral converges uniformly for $(t,x)\in [t_0,\infty)\times\mathbb{R}^d$.
\end{example}

\noindent As \cite{Baer12} points out, the renormalized integral is linear, monotonic, and satisfies the triangle inequality. Unsurprisingly however, the renormalized integral does not satisfy basic limit theorems of measure theory, including Fatou's lemma and the monotone and dominated convergence theorems. The following proposition is a characterization of uniform convergence of renormalized integrals in terms of a Cauchy condition. Its proof is straightforward and omitted.

\begin{proposition}\label{prop:Cauchy}
Let $\mathcal{A}$ be an approximating family of $X\in\mathcal{M}_d$ and let $Y$ be a non-empty set.  Suppose that $g=g_{(\cdot)}(\cdot):Y\times X\to\mathbb{C}$ is such that $\xi\mapsto g_y(\xi)$ is Lebesgue measurable for each $y\in Y$. Then renormalized integral of $g_y$ converges uniformly on $Y$ if and only if, the following Cauchy condition is satisfied: For each $\epsilon>0$, there exists $\tau_0>0$ such that
\begin{equation*}
\left|\int_{\mathcal{O}_{\tau_2}\setminus\mathcal{O}_{\tau_1}}g_y\,dm\right|<\epsilon
\end{equation*}
for all $y\in Y$ and $\tau_0\leq\tau_1\leq\tau_2$.
\end{proposition}

\noindent Armed with the renormalized integral as a basic tool, we turn our focus to the attractors given formally by \eqref{eq:ImagHomAttractorFormal}. In what follows, we let $Q:\mathbb{R}^d\to\mathbb{R}$ be a continuous function for which $\abs{Q}$ is positive homogeneous with homogeneous order $\mu=\mu_Q=\mu_{|Q|}$ and exponent set $\Exp(Q)=\Exp(\abs{Q})$ (c.f., Remark \ref{rem:AbsPosHomExponentSets}). For each $\tau>0$, define $\mathcal{O}_{\tau}=\{\xi\in\mathbb{R}^d:\abs{Q(\xi)}<\tau\}$ and observe that $\mathcal{A}=\{\mathcal{O}_{\tau}:\tau >0\}$ is an approximating family of $\mathbb{R}^d$. The following theorem, stated in terms of the renormalized integral, introduces the functions which will appear as attractors in local limit theorems corresponding to points of imaginary homogeneous type.

\begin{theorem}\label{thm:Attractor}
Let $Q:\mathbb{R}^d\to\mathbb{R}$ be continuous and such that $\abs{Q}$ is positive homogeneous with homogeneous order $\mu=\mu_{\abs{Q}}>0$. Also, let $\mathcal{A}$ be as above. If $\mu<1$, then the following statements hold.
\begin{enumerate}
\item\label{item:ImagHomAttractorConverge} For any $t>0$ and $x\in\mathbb{R}^d$, the renormalized integral of $\xi\mapsto e^{-itQ(\xi)-ix\cdot\xi}$ over $\mathbb{R}^d$ associated to the family $\mathcal{A}$ converges and we set
\begin{equation}\label{eq:ImagHomAttractorRenorm}
H_{iQ}^t(x)=\frac{1}{(2\pi)^d}\fint_{\mathbb{R}^d} e^{-itQ(\xi)-ix\cdot \xi}\,d_{\mathcal{A}}\xi.
\end{equation}
Further, $H_{iQ}^t(x)$ is continuous on its domain, $(0,\infty)\times\mathbb{R}^d$.
\item\label{item:ImagHomAttractorConvergeUniformly} For each $t_0>0$ and compact set $K\subseteq\mathbb{R}^d$, the renormalized integral in \eqref{eq:ImagHomAttractorRenorm} converges uniformly for $t\geq t_0$ and $x\in K$. 
\item\label{item:ImagHomAttractorScale} For any $E\in\Exp(Q)$, 
\begin{equation}\label{eq:ImagHomAttractorScale}
H_{iQ}^t(x)=\frac{1}{t^\mu}H_{iQ}^1\left(t^{-E^*}x\right)
\end{equation}
for every $t>0$ and $x\in\mathbb{R}^d$.
\item\label{item:ImagHomAttractorAtOrigin} For all $t>0$,
\begin{equation*}
\Re\left(H_{iQ}^t(0)\right)=\frac{t^{-\mu}}{(2\pi)^d}m(B_{\abs{Q}})\Gamma(1+\mu)\cos\left(\frac{\mu\pi}{2}\right)>0
\end{equation*}
where $B_{\abs{Q}}=\{\xi\in\mathbb{R}^d:\abs{Q(\xi)}<1\}$ (and has finite measure by virtue of Proposition \ref{prop:CharofPosHom}) and $\Gamma(\cdot)$ is Euler's Gamma function.
\end{enumerate}
\end{theorem}

\noindent To see that the hypothesis $\mu<1$ cannot be weakened, in general, consider $Q:\mathbb{R}^2\to\mathbb{R}$ defined by $Q(\xi)=q\abs{\xi}^2=q(\xi_1^2+\xi_2^2)$ for $\xi=(\xi_1,\xi_2)\in\mathbb{R}^2$ where $q>0$. We see immediately that $\abs{Q(\xi)}=q\abs{\xi}^2$ is positive homogeneous with $\mu=1/2+1/2=1$. In this case, for each $\tau>0$, observe that $\mathcal{O}_\tau=\{\xi\in\mathbb{R}^2:\abs{\xi}<\sqrt{\tau/q}\}$ and so the polar-coordinate integration formula gives
\begin{equation*}
\int_{\mathcal{O}_\tau}e^{-itQ(\xi)}\,d\xi=\int_{0}^{2\pi}\int_0^{\sqrt{\tau/q}}e^{-itqr^2}r\,dr\,d\theta=\frac{\pi i}{q t}\left(e^{it\tau}-1\right)
\end{equation*}
for $t>0$. Consequently, the renormalized integral
\begin{equation*}
\fint_{\mathbb{R}^2}e^{-itQ(\xi)}e^{-ix\cdot\xi}\,d_{\mathcal{A}}\xi
\end{equation*}
does not converge at $x=0$ for any $t>0$. As the reader may notice, an analogous renormalized integral for $Q(\xi)=q|\xi|^2$ does converge provided the approximating family $\mathcal{A}$ is replaced by one in which the sets $\mathcal{O}_\tau$ are rectangular and, in that case, the attractor that results is a product of one-dimensional heat kernels evaluated at purely imaginary time. In Section \ref{sec:Discussion} (see also Section 5.4 of \cite{RSC17}), we discuss a theory of convolution powers for a very special class of functions on $\mathbb{Z}^d$ for which local limit theorems can be established where the attractors are products of the one-dimensional attractors of \cite{RSC15}. How such local limit theorems fit within a broader theory for convolution powers -- one which would also encompass the results of the present article -- remains an open question.\\


\noindent As we will see below, an important part of the proof of Theorem \ref{thm:Attractor} makes use of the generalized polar-coordinate integration formula developed in \cite{BR21}. Specifically, Theorem 1.3 of \cite{BR21} hands us a Radon measure $\sigma=\sigma_{\abs{Q}}$ on the unital level set $S=S_{\abs{Q}}=\{\xi\in\mathbb{R}^d:\abs{Q(\xi)}=1\}$ for which
\begin{equation*}
\int_{\mathbb{R}^d}f(\xi)\,d\xi=\int_S\int_0^\infty f(r^E\eta)r^{\mu-1}\,dr\,\sigma(d\eta)
\end{equation*}
for any $f\in L^1(\mathbb{R}^d)$ and $E\in\Exp(Q)=\Exp(\abs{Q})$. We note that, for  $E\in\Exp(Q)=\Exp(\abs{Q})$, $\abs{Q(r^E\eta)}=r\abs{Q(\eta)}=r$ for all $\eta\in S$ and $r>0$ and so it follows that, for each $\tau>0$,
\begin{equation*}\chi_{\mathcal{O}_\tau}(r^E\eta)=\chi_{[0,\tau)}(r)
\end{equation*} for all $\eta\in S$ and $r>0$. We note that, by virtue of the continuity of $Q$ and Proposition \ref{prop:CharofPosHom}, each $\mathcal{O}_\tau$ is relatively compact with compact closure $\overline{\mathcal{O}_\tau}=\{\xi\in\mathbb{R}^d:\abs{Q(\xi)}\leq \tau\}$. With these observations in mind, we find that, for any $f\in L^1_{\mbox{\tiny loc}}(\mathbb{R}^d)$, $E\in\Exp(Q)=\Exp(\abs{Q})$ and $\tau>0$,
\begin{equation}\label{eq:ApproxFormulaViaPolar}
\int_{\mathcal{O}_\tau}f(\xi)\,d\xi=\int_S\int_0^\tau f(r^E\eta) r^{\mu-1}\,dr\,\sigma(d\eta).
\end{equation}
When applied to $f(\xi)=e^{-itQ(\xi)-ix\cdot\xi}$, we obtain
\begin{equation*}
\int_{\mathcal{O}_{\tau}}e^{-itQ(\xi)-ix\cdot\xi}\,d\xi=\int_S\int_0^{\tau}e^{-itrQ(\eta)-ix\cdot(r^E\eta)}r^{\mu-1}\,dr\sigma(d\eta)
\end{equation*}
for $x\in\mathbb{R}^d$ and $\tau>0$. As we will see, the advantage of this representation is that it allows us to handle the oscillatory behavior of the integrand by studying its oscillations in the ``radial" direction $r$. In this way, we reduce the question of convergence to the analysis of an oscillatory integral in one dimension. For this reason, we will make use of the famous Van der Corput lemma which we state here for completeness (for a proof, see  \cite{stein_harmonic_1993}, \cite{RSC15}, or \cite{Bui21}).

\begin{proposition}[Van der Corput Lemma]\label{prop:VDC}
Given a compact interval $I=[a,b]$, let $f$ and $g$ be real-valued elements of $C^1(I)$, i.e., $f,g\in C^1(I;\mathbb{R})$, and denote by $f'$ and $g'$ their first derivatives, respectively. If $f'$ is monotonic on $I$ and, for $\lambda>0$, $\abs{f'(\theta)}\geq \lambda$ for all $\theta\in I$, then
\begin{equation*}
\abs{\int_a^b e^{-if(\theta)}g(\theta)\,d\theta}\leq 4\frac{\|g\|_{L^\infty(I)}+\|g'\|_{L^1(I)}}{\lambda}.
\end{equation*}
\end{proposition}

\begin{proof}[Proof of Theorem \ref{thm:Attractor}]
We first prove Item \ref{item:ImagHomAttractorConvergeUniformly}. Let $\epsilon>0$, $t_0>0$ and $K\subseteq\mathbb{R}^d$ be a fixed compact set. By virtue of Proposition \ref{prop:Cauchy}, it suffices to show that there is a constant $\tau_0>0$ for which
\begin{equation*}
\abs{\int_{\mathcal{O}_{\tau_2}\setminus\mathcal{O}_{\tau_1}}\exp(-itQ(\xi)-ix\cdot\xi)\,d\xi}<\epsilon
\end{equation*}
for all $t\geq t_0$, $x\in K$ and $\tau_2\geq \tau_1\geq \tau_0$. Fix $E\in\Exp(Q)=\Exp(\abs{Q})$ and observe that, by virtue of \eqref{eq:ApproxFormulaViaPolar}, for $0\leq \tau_1\leq \tau_2$, $t>0$ and $x\in \mathbb{R}^d$,
\begin{eqnarray*}
\int_{\mathcal{O}_{\tau_2}\setminus\mathcal{O}_{\tau_1}}\exp(-itQ(\xi)-ix\cdot\xi)\,d\xi&=& \int_{S}\int_{\tau_1}^{\tau_2}\exp\left(-itQ(r^{E}\eta)-ix\cdot(r^{E}\eta)\right)r^{\mu-1}\,dr\,\sigma_P(d\eta)\\
&=& \int_S\int_{\tau_1}^{\tau_2}\exp\left(-itrQ(\eta)-ix\cdot(r^{E}\eta)\right)r^{\mu-1}\,dr\,\sigma_P(d\eta)\\
&=&\int_S I_{\tau_1,\tau_2,t,x}(\eta)\sigma_P(d\eta)
\end{eqnarray*}
where
\begin{equation*}
I_{\tau_1,\tau_2,t,x}(\eta)=\int_{\tau_1}^{\tau_2}\exp\left(-itrQ(\eta)-ix\cdot(r^{E}\eta)\right)r^{\mu-1}\,dr
\end{equation*}
for each $\eta\in S$. Upon making the change of variables $r=\theta^{1/\mu}$ and noting that $d\theta/\mu=r^{\mu-1}\,dr$, we have
\begin{equation*}
I_{\tau_1,\tau_2,t,x}(\eta)=\frac{1}{\mu}\int_{\tau_1^\mu}^{\tau_2^\mu}e^{if_{t,x,\eta}(\theta)}\,d\theta
\end{equation*}
where
\begin{equation*}
f_{t,x,\eta}(\theta)=-\left(tQ(\eta)\theta^{1/\mu}+x\cdot (\theta^F\eta)\right),
\end{equation*}
$F=E/\mu$. For simplicity, we write $f=f_{t,x,\eta}$ and observe that
\begin{equation}\label{eq:AttractorProof1}
\partial_{\theta}f(\theta)=-\left(\frac{t}{\mu}Q(\eta)\theta^{1/\mu-1}+x\cdot\left(\theta^{F-I}F\eta\right)\right)
\end{equation}
and
\begin{equation}\label{eq:AttractorProof2}
\theta^2\partial_\theta^2 f(\theta)=-\left(\frac{t}{\mu}\left(\frac{1}{\mu}-1\right)Q(\eta)\theta^{1/\mu}+x\cdot (\theta^F(F-I)F\eta\right)
\end{equation}
for all $t>0$, $x\in\mathbb{R}^d$, $\eta\in S$, and $\theta>0$. Using the estimates
\begin{equation*}
\abs{x\cdot (\theta^F(F-I)F\eta)}\leq \abs{x}\abs{(F-I)F\eta}\|\theta^F\|\hspace{.5cm}\mbox{and}\hspace{.5cm}\abs{x\cdot (\theta^{F-I}F\eta)}\leq \abs{x}\abs{F\eta}\|\theta^F\|\theta^{-1},
\end{equation*}
the compactness of $K$ and $S$, and our hypothesis that $\mu<1$, an appeal to Corollary \ref{cor:LargeTimeContractingGroup} (with $\alpha=\mu$) hands us $\tau_0>0$ for which
\begin{equation*}
\abs{x\cdot(\theta^F(F-I)F\eta)}\leq \frac{t_0}{2\mu}\left(\frac{1}{\mu}-1\right)\theta^{1/\mu}
\end{equation*}
and
\begin{equation*}
\abs{x\cdot(\theta^{F-I}F\eta)}\leq \frac{t_0}{2\mu}\theta^{1/\mu-1}
\end{equation*}
for all $x\in K$, $\eta\in S$ and $\theta\geq \tau_0^\mu$. Using these inequalities and the fact that $\abs{Q(\eta)}=1$ for all $\eta\in S$, from \eqref{eq:AttractorProof1} and \eqref{eq:AttractorProof2} it follows that, for all $t\geq t_0$, $x\in K$, $\eta\in S$ and $\theta\geq \tau_0^\mu$, 
\begin{equation*}
\abs{\theta^2\partial_\theta^2 f(\theta)}\geq \frac{t}{2\mu}\left(\frac{1}{\mu}-1\right)\theta^{1/\mu}>0
\end{equation*}
and
\begin{equation*}
\abs{\partial_\theta f(\theta)}\geq \frac{t}{2\mu}\theta^{1/\mu-1};
\end{equation*}
in particular, $\theta\mapsto \partial_\theta f(\theta)$ is monotonic on $[\tau_0^\mu,\infty)$. By further enlarging $\tau_0$ so that $\theta^{1/\mu-1}\geq 16\,\sigma(S)/(\epsilon t_0)$ for all $\theta\geq \tau_0^\mu$, if necessary, we have
\begin{equation*}
\abs{\partial_\theta f(\theta)}\geq \frac{t}{2\mu}\theta^{1/\mu-1}\geq 8\frac{\sigma(S)}{\mu\epsilon}\frac{t}{t_0}\geq 8\frac{\sigma(S)}{\mu\epsilon}
\end{equation*}
for all $t\geq t_0$, $x\in K$, $\eta\in S$ and $\theta\geq \tau_0^\mu$. By an appeal to Proposition \ref{prop:VDC} (with $f(\theta)$ as above and $g(\theta)\equiv 1$), we obtain
\begin{equation*}
\abs{I_{\tau_1,\tau_2,t,x}(\eta)}\leq\frac{1}{\mu}\frac{4(1+0)}{ 8\sigma(S)/\mu\epsilon}= \frac{1}{2}\frac{\epsilon}{\sigma(S)}
\end{equation*}
for all $t\geq t_0$, $x\in K$, $\eta\in S$ and $\tau_2\geq\tau_1\geq \tau_0$.
Consequently,
\begin{equation*}
\abs{\int_{\mathcal{O}_{\tau_2}\setminus\mathcal{O}_{\tau_1}}\exp(-itQ(\xi)-ix\cdot\xi)\,d\xi}\leq\int_S\abs{I_{\tau_1,\tau_2,t,x}(\eta)}\,\sigma(d\eta)\leq \frac{1}{2}\int_S \frac{\epsilon}{\sigma(S)}\,\sigma(d\eta)<\epsilon
\end{equation*}
for all $t\geq t_0$, $x\in K$ and $\tau_2\geq \tau_1\geq \tau_0$, as was asserted.

The convergence statement of Item \ref{item:ImagHomAttractorConverge} follows immediately from Item \ref{item:ImagHomAttractorConvergeUniformly}. Since the approximants
\begin{equation*}
(t,x)\mapsto \frac{1}{(2\pi)^d}\int_{\mathcal{O}_\tau}e^{-itQ(\xi)-ix\cdot\xi}\,d\xi
\end{equation*}
of $H_{iQ}^t(x)$ are all continuous and, by virtue of Item \ref{item:ImagHomAttractorConvergeUniformly}, converge uniformly on all compact subsets of $(0,\infty)\times \mathbb{R}^d$, it follows that $H_{iQ}^t(x)$ is continuous on its domain. This completes the proof of Item \ref{item:ImagHomAttractorConverge}.

Now, for an arbitrary $E\in\Exp(Q)=\Exp(|Q|)$, observe that, for each $t>0$, $x\in\mathbb{R}^d$ and $\tau>0$, 
\begin{eqnarray}\label{eq:ImagHomAttractorCOV}\nonumber
\int_{\mathcal{O}_\tau}e^{-itQ(\xi)-ix\cdot\xi}\,d\xi&=&\int_{\mathcal{O}_\tau}e^{-iQ(t^E\xi)-ix\cdot\xi}\,d\xi\\\nonumber
&=&\int_{t^E(\mathcal{O}_{\tau})}e^{iQ(\zeta)-ix\cdot t^{-E}\zeta}\det(t^{-E})\,d\zeta\\
&=&t^{-\mu}\int_{\mathcal{O}_{t\tau}}e^{iQ(\zeta)-i(t^{-E^*}x)\cdot\zeta}\,d\zeta
\end{eqnarray}
where we have made the change of variables $\zeta=t^E\xi$ and noted that $\det(t^{-E})=t^{-\tr E}=t^{-\mu}$ and
\begin{equation*}
t^E(\mathcal{O}_\tau)=\{t^E\xi: \abs{Q(\xi)}<\tau\}=\{t^E\xi:\abs{Q(t^E\xi)}=t\abs{Q(\xi)}< t\tau\}=\{\zeta:\abs{Q(\zeta)}< t\tau\}=\mathcal{O}_{t\tau}.
\end{equation*}
Consequently
\begin{eqnarray*}
H_{iQ}^t(x)&=&\frac{1}{(2\pi)^d}\lim_{\tau\to\infty}\int_{\mathcal{O}_\tau}e^{-itQ(\xi)-ix\cdot\xi}\,d\xi\\
&=&\frac{t^{-\mu}}{(2\pi)^d}\lim_{\tau\to\infty}\int_{\mathcal{O}_{t\tau}}e^{-iQ(\zeta)-i(t^{-E^*}x)\cdot \zeta}\,d\zeta\\
&=&t^{-\mu}H_{iQ}^1(t^{-E^*}x)
\end{eqnarray*}
for each $t>0$ and $x\in\mathbb{R}^d$; this is precisely the assertion made in Item \ref{item:ImagHomAttractorScale}.

To prove the final assertion, we first assume that $Q$ is a single-signed function and write $s=\sign(Q(\xi))$ for any non-zero $\xi$. Of course, this is always the case when $d>1$ because $\mathbb{R}^d\setminus \{0\}$ is connected and $Q(\xi)$ is continuous and non-vanishing for $\xi\neq 0$. By virtue of Item \ref{item:ImagHomAttractorScale} and \eqref{eq:ApproxFormulaViaPolar}, we have
\begin{eqnarray*}
H_{iQ}^t(0)&=&t^{-\mu}H_{iQ}^1(0)\\
&=&\frac{t^{-\mu}}{(2\pi)^d}\lim_{\tau\to\infty}\int_{\mathcal{O}_\tau}e^{-iQ(\xi)}\,d\xi\\
&=&\frac{t^{-\mu}}{(2\pi)^d}\lim_{\tau\to\infty}\int_S\int_0^\tau e^{-irQ(\eta)}r^{\mu-1}\,dr\,\sigma(d\eta)\\
&=&\frac{t^{-\mu}}{(2\pi)^d}\sigma(S)\lim_{\tau\to\infty}\int_0^\tau e^{-isr}r^{\mu-1}\,dr\\
&=&\frac{t^{-\mu}}{(2\pi)^d}\left(\frac{\sigma(S)}{\mu}\right)\lim_{\tau\to\infty}\int_0^{\tau^\mu}e^{-is\theta^{1/\mu}}\,d\theta
\end{eqnarray*}
for all $t>0$ where we have made the change of variables $r=\theta^{1/\mu}$. Given that $1/\mu>1$, a routine exercise in contour integration shows that
\begin{equation*}
\lim_{\tau\to\infty}\int_0^{\tau^\mu}e^{-is\theta^{1/\mu}}\,d\theta=\int_0^\infty e^{-is\theta^{1/\mu}}\,d\theta=e^{-is\pi/(2/\mu)}\Gamma\left(1+\frac{1}{1/\mu}\right)=e^{-is\mu\pi/2}\Gamma(1+\mu).
\end{equation*}
Upon noting that $m(B_{\abs{Q}})=\sigma(S)/\mu$ in view of Item 3 of Theorem 1.3 of \cite{BR21}, we conclude that
\begin{equation*}
H_{iQ}^t(0)=\frac{t^{-\mu}}{(2\pi)^d}m(B_{\abs{Q}})\Gamma(1+\mu) e^{-is\mu\pi/2}
\end{equation*}
from which it follows that
\begin{equation*}
\Re(H_{iQ}^t(0))=\frac{t^{-\mu}}{(2\pi)^d}m(B_{\abs{Q}})\Gamma(1+\mu)\cos\left(\frac{\mu\pi}{2}\right)
\end{equation*}
for all $t>0$. 

As pointed out in the preceding paragraph, $Q$ can only take on both positive and negative values when $d=1$. In this setting, it is straightforward to verify that
\begin{equation*}
Q(\xi)=\begin{cases}
s_+ q_+ |\xi|^{1/\mu} & \xi\geq 0\\
s_- q_- |\xi|^{1/\mu} & \xi<0
\end{cases}
\end{equation*}
for all $\xi\in\mathbb{R}$ where $q_+=\abs{Q(1)}$, $q_-=\abs{Q(-1)}$, $s_+=\sign(Q(1))$, and $s_-=\sign(Q(-1))$. Using this representation, a direct computation shows that
\begin{equation*}
H_{iQ}^t(0)=\frac{t^{-\mu}}{2\pi}\Gamma(1+\mu)\left(\frac{e^{-i(s_+)\mu\pi/2}}{q_+^{\mu}}+\frac{e^{-i(s_-)\mu\pi/2}}{q_-^{\mu}}\right)
\end{equation*}
for all $t>0$. Upon noting that
\begin{equation*}
m(B_{\abs{Q}})=m\left(\{\xi\in\mathbb{R}:\abs{Q(\xi)}<1\}\right)=\frac{1}{q_+^\mu}+\frac{1}{q_-^\mu},
\end{equation*}
we immediately obtain the desired identity in this case as well.
\end{proof}

\section{Local Limit Theorems}\label{sec:LLT}

\noindent In this section, we establish local limit theorems for a class of complex-valued functions on $\mathbb{Z}^d$. To this end, we shall fix $\phi\in\mathcal{S}_d\subseteq \ell^1(\mathbb{Z}^d)$ which is suitably normalized so that $\sup_{\xi}|\widehat{\phi}|=1$ and satisfies various hypotheses concerning the nature of the points in $\Omega(\phi)=\{\xi\in\mathbb{T}^d:|\widehat{\phi}(\xi)|=1\}$ described below. All of our results include the assumption that every point of $\Omega(\phi)$ is either of positive homogeneous type or imaginary homogeneous type for $\widehat{\phi}$.  Our main result, Theorem \ref{thm:LLTMain}, allows for $\Omega(\phi)$, under certain conditions, to contain a mixture of points positive homogeneous and imaginary homogeneous type for $\widehat{\phi}$. This theorem partially extends Theorem 1.3 of \cite{RSC15} to the $d$-dimensional setting and it extends Theorem 1.6 of \cite{RSC17} to include points of imaginary homogeneous type. As we will see, Theorem \ref{thm:LLTIntro} follows easily from Theorem \ref{thm:LLTMain}.\\

\noindent In view of \eqref{eq:FTConvolution}, our analysis will be done by studying $\widehat{\phi}(\xi)$ locally at each point $\xi_0\in\Omega(\phi)\subseteq\mathbb{T}^d$. For these local studies, it is helpful (though not essential) for the points $\xi_0\in \Omega(\phi)$ to live in the interior of $\mathbb{T}^d$, which is not always the case (e.g., simple random walk on $\mathbb{Z}^d$). As we previously discussed, our hypothesis that all points of $\Omega(\phi)$ are either of positive homogeneous or imaginary homogeneous type for $\widehat{\phi}$ ensures that $\Omega(\phi)$ is a finite set and so it follows that, for some $\xi_\phi\in\mathbb{R}^d$,
\begin{equation*}
\mathbb{T}^d_\phi:=\mathbb{T}^d+\xi_\phi
\end{equation*}
contains $\Omega(\phi)$ in its interior (as a subset of $\mathbb{R}^d)$. Using this representation of the torus, we have
\begin{equation}\label{eq:FTConvolutionPhi}
\phi^{(n)}(x)=\frac{1}{(2\pi)^d}\int_{\mathbb{T}_{\phi}^d}\widehat{\phi}(\xi)^n e^{-ix\cdot\xi}\,d\xi
\end{equation}
for all $x\in\mathbb{R}^d$ and $n\in\mathbb{N}_+$. The following lemma addresses the contribution from points $\xi_0\in\Omega(\phi)$ which are of positive homogeneous type for $\widehat{\phi}$. Though this lemma and its proof can be found, essentially, as Lemma 4.3 in \cite{RSC17}, the proof we give here is simplified and allows us to highlight the ingredients which carry over easily to the imaginary homogeneous setting and those which do not.

\begin{lemma}\label{lem:LocalLimitPosHom}
Let $\xi_0\in\Omega(\phi)$ be of positive homogeneous type for $\widehat{\phi}$ with associated $\mu=\mu_{\xi_0}$, $\alpha=\alpha_{\xi_0}$ and $P=R+iQ$ where $R=R_{\xi_0}$ and $Q=Q_{\xi_0}$. Then, for any $\epsilon>0$, there exists an open neighborhood of $\mathcal{U}_{\xi_0}\subseteq\mathbb{T}_\phi^d$ of $\xi_0$, which can be taken as small as desired, and a natural number $N$ for which
\begin{equation*}
\abs{\frac{1}{(2\pi)^d}\int_{\mathcal{U}_{\xi_0}}\widehat{\phi}(\xi)^ne^{-ix\cdot\xi}\,d\xi-\widehat{\phi}(\xi_0)^ne^{-ix\cdot\xi_0}H_P^n(x-n\alpha)}<\epsilon n^{-\mu}
\end{equation*}
for all $n\geq N$ and $x\in\mathbb{Z}^d$.
\end{lemma}

\begin{proof}
Let $\epsilon>0$ be fixed and, for each $\tau>0$, define $\mathcal{O}_\tau=\{\xi\in\mathbb{R}^d:R(\xi)<\tau\}$. Upon writing $\widetilde{R}=\widetilde{R}_{\xi_0}$, $\widetilde{Q}=\widetilde{Q}_{\xi_0}$ and $\widetilde{P}=\widetilde{R}+i\widetilde{Q}$, we have
\begin{equation*}
\Gamma(\xi):=\Gamma_{\xi_0}(\xi)=i\alpha\cdot\xi-P(\xi)-\widetilde{P}(\xi)
\end{equation*}
on some convex neighborhood $\mathcal{U}$ of $0$.  With the help of Proposition \ref{prop:Subhomequivtolittleoh}, choose $\delta>0$ for which $\mathcal{O}_\delta\subseteq\mathcal{U}$, $\mathcal{U}_{\xi_0}:=\mathcal{O}_\delta+\xi_0$ is as small as desired and (minimally) ensures that $\mathcal{U}_{\xi_0}\subseteq\mathbb{T}^d_\phi$, and
\begin{equation}\label{eq:LocalLimitPosHom1}
\abs{\widetilde{R}(\xi)}<R(\xi)/2
\end{equation}
whenever $\xi\in\mathcal{O}_{\delta}$. With this neighborhood $\mathcal{U}_{\xi_0}$ of $\xi_0$, we define
\begin{equation*}
\mathcal{E}=\mathcal{E}_{n,x}=n^{\mu}\abs{\frac{1}{(2\pi)^d}\int_{\mathcal{U}_{\xi_0}}\widehat{\phi}(\xi)^ne^{-ix\cdot\xi}\,d\xi-\widehat{\phi}(\xi_0)^ne^{-ix\cdot\xi_0}H_P^n(x-n\alpha)}
\end{equation*}
for $n\in\mathbb{N}_+$ and $x\in\mathbb{Z}^d$. It is clear that, to prove the lemma, we must find $N\in\mathbb{N}_+$ for which $\mathcal{E}<\epsilon$ for all $x\in\mathbb{Z}^d$ and $n\geq N$. To this end, we first make the change of vaiables $\xi\mapsto \xi+\xi_0$ to see that
\begin{eqnarray}\label{eq:LocalLimitPosHom2}\nonumber
\int_{\mathcal{U}_{\xi_0}}\widehat{\phi}(\xi)^ne^{-ix\cdot\xi}\,d\xi&=& \int_{\mathcal{O}_{\delta}}\widehat{\phi}(\xi+\xi_0)^ne^{-ix\cdot(\xi+\xi_0)}\,d\xi\\\nonumber
&=&\int_{\mathcal{O}_\delta}\widehat{\phi}(\xi_0)^ne^{n\Gamma(\xi)}e^{-ix\cdot\xi_0}e^{-ix\cdot\xi}\,d\xi\\
&=&\widehat{\phi}(\xi_0)^ne^{-ix\cdot\xi_0}\int_{\mathcal{O}_\delta}e^{-n(P(\xi)+\widetilde{P}(\xi))}e^{-i(x-n\alpha)\cdot\xi}\,d\xi
\end{eqnarray}
for all $n\in\mathbb{N}_+$ and $x\in\mathbb{Z}^d$. Consequently,
\begin{equation*}
\mathcal{E}=n^{\mu}\abs{\frac{1}{(2\pi)^d}\int_{\mathcal{O}_\delta}e^{-nP(\xi)-n\widetilde{P}(\xi)}e^{-i(x-n\alpha)\cdot\xi}\,d\xi-H_P^n(x-n\alpha)}
\end{equation*}
for $n\in\mathbb{N}_+$ and $x\in\mathbb{Z}^d$. Now, let $E\in \Exp(R)\cap\Exp(Q)=\Exp(P)$ be as given by Definition \ref{def:Types} and observe that
\begin{equation*}
H^n_P(x-n\alpha)=n^{-\mu}H_P^1(y)
\end{equation*}
where 
\begin{equation*}
y=y_{n,x}=n^{-E^*}(x-n\alpha)
\end{equation*}
by virtue of Proposition \ref{prop:PosHomKernel}. With this $E\in\Exp(P)$, the change of variables $\xi\mapsto n^{-E}\xi$ yields
\begin{equation*}
\int_{\mathcal{O}_\delta}e^{-nP(\xi)-n\widetilde{P}(\xi)}e^{-i(x-n\alpha)\cdot\xi}\,d\xi=n^{-\mu}\int_{\mathcal{O}_{n\delta}}e^{-P(\xi)-n\widetilde{P}(n^{-E}\xi)}e^{-iy\cdot\xi}\,d\xi
\end{equation*}
for all $n\in\mathbb{N}_+$ and $x\in\mathbb{Z}^d$ where we have made use of the fact that $n^{E}(\mathcal{O}_\delta)=\mathcal{O}_{n\delta}$. Consequently,
\begin{equation*}
\mathcal{E}=\abs{\frac{1}{(2\pi)^d}\int_{\mathcal{O}_{n\delta}}e^{-P(\xi)-n\widetilde{P}(n^{-E}\xi)}e^{-iy\cdot\xi}\,d\xi-H_P^1(y)}
\end{equation*}
for all $n\in\mathbb{N}_+$ and $x\in\mathbb{Z}^d$ where $y=n^{-E^*}(x-n\alpha)$. In view of Proposition \ref{prop:PosHomKernel} (applied to the positive homogeneous function $R$), let $\tau_0>0$ be such that
\begin{equation*}
\frac{1}{(2\pi)^d}\int_{\mathbb{R}^d\setminus\mathcal{O}_{\tau_0}}e^{-R(\xi)}\,d\xi<\epsilon/2.
\end{equation*}
Thus, for any $\tau\geq \tau_0$,
\begin{eqnarray}\label{eq:LocalLimitPosHom3}\nonumber
\abs{\frac{1}{(2\pi)^d}\int_{\mathcal{O}_{\tau}}e^{-P(\xi)}e^{-iy\cdot\xi}\,d\xi-H_P^1(y)}&=&\abs{\frac{1}{(2\pi)^d}\int_{\mathbb{R}^d\setminus\mathcal{O}_{\tau}}e^{-P(\xi)}e^{-iy\cdot\xi}\,d\xi}\\\nonumber
&\leq&\frac{1}{(2\pi)^d}\int_{\mathbb{R}^d\setminus\mathcal{O}_{\tau_0}}e^{-R(\xi)}\,d\xi\\
&<&\epsilon/2
\end{eqnarray}
for all $y\in\mathbb{R}^d$. Observe that
\begin{eqnarray}\label{eq:LocalLimitPosHom4}\nonumber
\lefteqn{\hspace{-6cm}\abs{\frac{1}{(2\pi)^d}\int_{\mathcal{O}_{n\delta}}e^{-P(\xi)-n\widetilde{P}(n^{-E}\xi)}e^{-iy\cdot\xi}\,d\xi-\frac{1}{(2\pi)^d}\int_{\mathcal{O}_{n\delta}}e^{-P(\xi)}e^{-iy\cdot\xi}\,d\xi}}\\\nonumber
\hspace{6cm}&=&\frac{1}{(2\pi)^d}\abs{\int_{\mathcal{O}_{n\delta}}e^{-P(\xi)}\left(e^{-n\widetilde{P}(n^{-E}\xi)}-1\right)e^{-iy\cdot\xi}\,d\xi}\\
&\leq&\frac{1}{(2\pi)^d}\int_{\mathbb{R}^d}e^{-R(\xi)}\abs{e^{-n\widetilde{P}(n^{-E}\xi)}-1}\chi_{\mathcal{O}_{n\delta}}(\xi)\,d\xi
\end{eqnarray}
for all $y\in\mathbb{R}^d$ and $n\in\mathbb{N}_+$; here, for a measurable set $A$ of $\mathbb{R}^d$, $\chi_A$ denotes the characteristic function of $A$. Upon noting that $n^{-E}\xi\in\mathcal{O}_\delta$ whenever $\xi\in\mathcal{O}_{n\delta}$, we see that
\begin{equation*}
\abs{e^{-n\widetilde{P}(n^{-E}\xi)}-1}\leq e^{-\widetilde{R}(n^{-E}\xi)}+1\leq e^{nR(n^{-E}\xi)/2}+1=e^{R(\xi)/2}+1\leq 2e^{R(\xi)/2}
\end{equation*}
whenever $n^{-E}\xi\in\mathcal{O}_{n\delta}$ where we have used \eqref{eq:LocalLimitPosHom1}. Consequently,
\begin{equation*}
e^{-R(\xi)}\abs{e^{-n\widetilde{P}(n^{-E}\xi)}-1}\chi_{\mathcal{O}_{n\delta}}(\xi)\leq 2e^{-R(\xi)/2}
\end{equation*}
for all $\xi\in\mathbb{R}^d$ and $n\in\mathbb{N}_+$ and we note that $\xi\mapsto 2e^{-R(\xi)/2}\in L^1(\mathbb{R}^d)$ thanks to Proposition \ref{prop:PosHomKernel}. Also, by virtue of the subhomogeneity of $\widetilde{R}$ and $\widetilde{Q}$ with respect to $E$, we have
\begin{equation*}
\lim_{n\to\infty} n\widetilde{P}(n^{-E}\xi)=\lim_{n\to\infty} n\widetilde{R}(n^{-E}\xi)+i\lim_{n\to\infty}n\widetilde{Q}(n^{-E}\xi)=0
\end{equation*}
for each $\xi\in\mathcal{O}_{n\delta}$ and so, for every $\xi\in\mathbb{R}^d$,
\begin{equation*}
\lim_{n\to\infty}e^{-R(\xi)}\abs{e^{-n\widetilde{P}(n^{-E}\xi)}-1}\chi_{\mathcal{O}_{n\delta}}(\xi)=0.
\end{equation*}
We may therefore appeal to the dominated convergence theorem to produce a natural number $N\geq \tau_0/\delta$ so that, in view of \eqref{eq:LocalLimitPosHom4}, 
\begin{equation}\label{eq:LocalLimitPosHom5}
\abs{\frac{1}{(2\pi)^d}\int_{\mathcal{O}_{n\delta}}e^{-P(\xi)-n\widetilde{P}(n^{-E}\xi)}e^{-iy\cdot\xi}\,d\xi-\frac{1}{(2\pi)^d}\int_{\mathcal{O}_{n\delta}}e^{-P(\xi)}e^{-iy\cdot\xi}\,d\xi}<\epsilon/2
\end{equation}
for all $y\in\mathbb{R}^d$ and $n\geq N$. Upon noting that $n\delta\geq\tau_0$ whenever $n\geq N$, the estimates \eqref{eq:LocalLimitPosHom3}, \eqref{eq:LocalLimitPosHom4}, and \eqref{eq:LocalLimitPosHom5} together yield
\begin{eqnarray*}
\mathcal{E}&\leq& \abs{\frac{1}{(2\pi)^d}\int_{\mathcal{O}_{n\delta}}e^{-P(\xi)-n\widetilde{P}(n^{-E}\xi)}e^{-iy\cdot\xi}\,d\xi-\frac{1}{(2\pi)^d}\int_{\mathcal{O}_{n\delta}}e^{-P(\xi)}e^{-iy\cdot\xi}\,d\xi}\\
&&\hspace{6cm}+\abs{\frac{1}{(2\pi)^d}\int_{\mathcal{O}_{n\delta}}e^{-P(\xi)}e^{-iy\cdot\xi}\,d\xi-H_P^1(y)}<\epsilon
\end{eqnarray*}
for all $x\in\mathbb{Z}^d$ and $n\geq N$, as desired.
\end{proof}

\noindent From the preceding lemma and in view of Item \ref{item:PosHomKernelOnDiagional} of Proposition \ref{prop:PosHomKernel}, we immediately obtain the following useful lemma. The lemma appears as Lemma 3.3 of \cite{BR21} and Lemma 4.4 of \cite{RSC17}.

\begin{lemma}\label{lem:SupNormEstPositiveHom}
Let $\xi_0\in\Omega(\phi)$ be of positive homogeneous type for $\widehat{\phi}$ with associated homogeneous order $\mu_{\xi_0}>0$. Then there exists an open neighborhood $\mathcal{U}_{\xi_0}\subseteq\mathbb{T}_\phi^d$ of $\xi_0$, which can be taken as small as desired, and a constant $C_{\xi_0}$ for which
\begin{equation*}
\abs{\frac{1}{(2\pi)^d}\int_{\mathcal{U}_{\xi_0}}\widehat{\phi}(\xi)^ne^{-ix\cdot\xi}\,d\xi}\leq C_{\xi_0}n^{-\mu_{\xi_0}}
\end{equation*}
for all $x\in\mathbb{Z}^d$ and $n\in\mathbb{N}_+$.
\end{lemma}

\noindent We shall now focus on the case in which $\xi_0\in\Omega(\phi)$ is of imaginary homogeneous type for $\widehat{\phi}$. In view of Definition \ref{def:Types}, there is an open neighborhood $\mathcal{U}$ of $0$ on which
\begin{equation*}
\Gamma(\xi)=i\alpha\cdot\xi-i\left(Q(\xi)+\widetilde{Q}(\xi)\right)-\left(R(\xi)+\widetilde{R}(\xi)\right)
\end{equation*}
where $\abs{Q}$ and $R$ are both positive homogeneous and there exists $E\in \Exp(Q)=\Exp(\abs{Q})$ and $\kappa>1$ for which $R$ is homogeneous with respect to $E/\kappa$, $\widetilde{Q}$ is strongly subhomogeneous with respect to $E$ of order $2$ and $\widetilde{R}$ is strongly subhomogeneous with respect to $E/\kappa$ of order $1$. Before we are able to prove a result analogous to Lemma \ref{lem:LocalLimitPosHom}, we shall first treat three preliminary lemmas, all of which are aimed at handling the oscillatory integrals which arise in this imaginary-homogeneous setting. For the reader's convenience, we remark that Lemmas \ref{lem:PhaseEstimates} and \ref{lem:AmplitudeEstimates} are established using basic properties of homogeneous and subhomogeneous functions. Lemma \ref{lem:OscillatoryEstimate} is our central oscillatory integral estimate and its proof makes use of the generalized polar-coordinate integration formula of \cite{BR21} (Theorem 1.3 of \cite{BR21}) to decompose an oscillatory integral over $\mathbb{R}^d$ into a ``radial" integral and a surface integral. The radial integral, which captures the primary oscillatory behavior, is then estimated by an application of the Van der Corput lemma (Proposition \ref{prop:VDC}) using Lemma \ref{lem:PhaseEstimates} to handle the phase and Lemma \ref{lem:AmplitudeEstimates} to handle the amplitude. We then present Lemma \ref{lem:LocalLimitImagHom}, our analogue of Lemma \ref{lem:LocalLimitPosHom} for the imaginary-homogeneous setting; its proof makes use of Theorem \ref{thm:Attractor} and Lemma \ref{lem:OscillatoryEstimate} (and does not rely directly on Lemmas \ref{lem:PhaseEstimates} or \ref{lem:AmplitudeEstimates}).

\begin{lemma}\label{lem:PhaseEstimates}
Let $Q:\mathbb{R}^d\to\mathbb{R}$ be a continuous function for which $\abs{Q}$ is positive homogeneous with unital level set $S$ and homogeneous order $\mu$. Also, given an open neighborhood $\mathcal{U}$ of $0$ in $\mathbb{R}^d$, let $\widetilde{Q}$ be a real-valued function which is  twice continuously differentiable on $\mathcal{U}$, i.e., $Q\in C^2(\mathcal{U};\mathbb{R})$. Given $E\in\Exp(Q)=\Exp(\abs{Q})$, define $F=E/\mu$ and
\begin{equation*}
f_{n,y,\eta}(\theta)=Q(\theta^F\eta)+n\widetilde{Q}(n^{-E}\theta^F\eta)+y\cdot(\theta^F\eta)
\end{equation*}
for $n\in\mathbb{N}_+$, $y\in\mathbb{R}^d$, $\eta\in S$ and $\theta>0$ for which $n^{-E}\theta^F\eta\in\mathcal{U}$. If $\mu<1$ and $\widetilde{Q}$ is strongly subhomogeneous with respect to $E$ of order $2$, then, for each compact set $K\subseteq\mathbb{R}^d$, there exist $\delta>0$ and $\theta_0\geq 1$ such that, for any natural number $n$ such that $\theta_0\leq (n\delta)^\mu$,  $\partial_{\theta}f_{n,y,\eta}(\theta)$ is monotonic on $[\theta_0,(n\delta)^\mu]$ and
\begin{equation*}
\abs{\partial_{\theta}f_{n,y,\eta}(\theta)}\geq \frac{1}{2\mu}\theta^{1/\mu-1}
\end{equation*}
for all $\theta_0\leq\theta\leq(n\delta)^\mu$, $y\in K$ and $\eta\in S$.
\end{lemma}
\begin{proof}
Let $K$ be a compact subset of $\mathbb{R}^d$. Given that $\widetilde{Q}$ is strongly subhomogeneous with respect to $E$ of order $2$, let $\delta>0$ be such that
\begin{equation}\label{eq:PhaseEstimate1}
\abs{\partial_{r}\widetilde{Q}(r^{E}\eta)}\leq \frac{1}{4}
\end{equation}
and
\begin{equation}\label{eq:PhaseEstimate2}
\abs{r\partial^2_r \widetilde{Q}(r^{E}\eta)}\leq \frac{1-\mu}{4}= \frac{\mu}{4}\left(\frac{1}{\mu}-1\right)
\end{equation}
for all $0<r<\delta$ and $\eta\in S$. Using the estimates
\begin{equation*}
\abs{y\cdot(\theta^{F-I}F\eta)}\leq \abs{y}\abs{F\eta}\|\theta^F\|\theta^{-1}\hspace{.5cm}\mbox{and}\hspace{.5cm}\abs{y\cdot(\theta^F(F-I)F\eta)}\leq\abs{y}\abs{(F-I)F\eta}\|\theta^F\|,
\end{equation*}
the compactness of $K$ and $S$, and the hypothesis that $\mu<1$, an application of Corollary \ref{cor:LargeTimeContractingGroup} (with $\alpha=\mu$) hands us $\theta_0\geq 1$ for which
\begin{equation}\label{eq:PhaseEstimate3}
\abs{y\cdot(\theta^{F-I}F\eta)}\leq \frac{1}{4\mu}\theta^{1/\mu-1}
\end{equation}
and
\begin{equation}\label{eq:PhaseEstimate4}
\abs{y\cdot (\theta^F (F-I)F\eta)}\leq \frac{1}{4\mu}\left(\frac{1}{\mu}-1\right)\theta^{1/\mu}
\end{equation}
for all $\theta\geq \theta_0$, $y\in K$ and $\eta\in S$. With these estimates in hand, let us now write $f=f_{n,y,\eta}$ and observe that, when $r=r_n(\theta)=\theta^{1/\mu}/n$, $r^{E}=n^{-E}\theta^F$ and so
\begin{equation*}
f(\theta)=\theta^{1/\mu}Q(\eta)+n\widetilde{Q}(r^{E}\eta)+y\cdot(\theta^F\eta).
\end{equation*}
We have
\begin{eqnarray}\label{eq:PhaseEstimate5}\nonumber
\partial_\theta f(\theta)&=&\frac{1}{\mu}\theta^{1/\mu-1}Q(\eta)+n\partial_r\widetilde{Q}(r^{E}\eta)\frac{\partial r}{\partial \theta}+\partial_\theta\left(y\cdot(\theta^F\eta)\right)\\\nonumber 
&=&\frac{1}{\mu}\theta^{1/\mu-1}Q(\eta)+n\partial_r\widetilde{Q}(r^{E}\eta)\frac{1}{\mu}\frac{\theta^{1/\mu-1}}{n}+y\cdot(\theta^{F-I}F\eta)\\
&=&\frac{1}{\mu}\theta^{1/\mu-1}\left(Q(\eta)-\partial_r\widetilde{Q}(r^{E}\eta)\right)+y\cdot(\theta^{F-I}F\eta)
\end{eqnarray}
for all $n\in\mathbb{N}_+$, $y\in K$, $\eta\in S$, and $\theta>0$ for which $r^{E}\eta=n^{-E}\theta^F\eta\in \mathcal{U}$. By virtue of \eqref{eq:PhaseEstimate1} and the fact that $\abs{Q(\eta)}=1$ for all $\eta\in S$, we see that $r^{E}\eta\in \mathcal{U}$ and
\begin{equation}\label{eq:PhaseEstimate6}
\abs{Q(\eta)-\partial_r\widetilde{Q}(r^{E}\eta)}\geq \frac{3}{4}
\end{equation}
whenever $\eta\in S$ and $0<\theta<(n\delta)^\mu$. Combining the estimates \eqref{eq:PhaseEstimate3}, \eqref{eq:PhaseEstimate5} and \eqref{eq:PhaseEstimate6}, guarantees that
\begin{equation*}
\abs{\partial_\theta f(\theta)}\geq \frac{3}{4\mu}\theta^{1/\mu-1}-\frac{1}{4\mu}\theta^{1/\mu-1}=\frac{1}{2\mu}\theta^{1/\mu-1}
\end{equation*}
for all $n\in\mathbb{N}_+$, $y\in K$, $\eta\in S$ and $\theta_0\leq \theta\leq (n\delta)^\mu$.

It remains to show that $\partial_\theta f(\theta)$ is monotonic. Making use of \eqref{eq:PhaseEstimate2} and \eqref{eq:PhaseEstimate4}, analogous reasoning shows that
\begin{equation*}
\abs{\theta^2\partial_\theta^2 f(\theta)}\geq \frac{1}{4\mu}\left(\frac{1}{\mu}-1\right)\theta^{1/\mu}>0
\end{equation*}
for all $n\in\mathbb{N}_+$, $y\in K$, $\eta\in S$ and $\theta_0\leq\theta\leq (n\delta)^\mu$. In particular, $\partial_\theta^2f(\theta)$ is non-vanishing on $[\theta_0,(n\delta)^\mu]$ and so $\partial_\theta f(\theta)$ is monotonic.
\end{proof}

\begin{lemma}\label{lem:AmplitudeEstimates}
Given a compact set $S\subseteq\mathbb{R}^d$ which does not contain $0$ and a positive homogeneous function $R$, let $m$ and $M$ be positive constants for which
\begin{equation*}
m\leq R(\eta)\leq M
\end{equation*}
for all $\eta\in S$. Also, let $\mathcal{U}\subseteq\mathbb{R}^d$ be an open neighborhood of $0$ and let $\widetilde{R}$ be a real-valued function which is once continuously differentiable on $\mathcal{U}$. Given $E\in\End(\mathbb{R}^d)$ for which $\{r^E\}$ is a contracting group, define $F:=E/\mu$ where $\mu=\tr E$ and
\begin{equation*}
g_{n,\eta}(\theta)=\exp\left(-nR(n^{-E}\theta^F\eta)-n\widetilde{R}(n^{-E}\theta^F\eta)\right)
\end{equation*}
for $n\in\mathbb{N}_+$, $\eta\in S$ and $\theta>0$ for which $n^{-E}\theta^F\eta\in\mathcal{U}$.  If, for some $\kappa>0$, $R$ is homogeneous with respect to $E/\kappa$ and $\widetilde{R}$ is strongly subhomogeneous with respect to $E/\kappa$ of order $1$, then, for any $\beta>1$, there exists a $\delta>0$ for which
\begin{equation*}
\|g_{n,\eta}\|_{L^{\infty}[\theta_1,\theta_2]}+\|\partial_\theta g_{n,\eta}\|_{L^1[\theta_1,\theta_2]}\leq 1+\frac{\beta M}{m}
\end{equation*}
for all $n\in\mathbb{N}_+$, $\eta\in S$ and $0<\theta_1\leq\theta_2\leq (n\delta)^\mu.$
\end{lemma}
\begin{proof}
Define
\begin{equation*}
h_{n,\eta}(r)=\exp\left(-nR(r^{E}\eta)-n\widetilde{R}(r^{E}\eta)\right)
\end{equation*}
for $n\in\mathbb{N}_+$, $\eta\in S$ and $r>0$ for which $r^{E}\eta\in \mathcal{U}$. Observe that, for $r=r_n(\theta)=\theta^{1/\mu}/n$, we have $r^{E}=n^{-E}\theta^F$ and
\begin{equation*}
g_{n,\eta}(\theta)=h_{n,\eta}(r_n(\theta))
\end{equation*}
for all $n\in\mathbb{N}_+$, $\eta\in S$ and $\theta>0$ for which $r^{E}\eta=n^{-E}\theta^F\eta\in\mathcal{U}$. Let us fix $\beta>1$ and let $0<\epsilon<1$ be such that
\begin{equation*}
\beta=\frac{1+\epsilon}{1-\epsilon}.
\end{equation*}
In view of our supposition that $\widetilde{R}$ is strongly subhomogeneous with respect to $E/\kappa$, there exists $\delta_1>0$ for which
\begin{equation*}
\abs{\widetilde{R}(r^{E}\eta)}\leq \epsilon m r^\kappa
\end{equation*}
for all $0<r<\delta_1$ and $\eta\in S$. Further, by virtue of the strong subhomogeneity of $\widetilde{R}$ with respect to $E/\kappa$, we can choose $\delta_2>0$ for which
\begin{equation*}
\abs{\partial_r\widetilde{R}(r^{E}\eta)}=\abs{\frac{\partial\widetilde{R}((r^\kappa)^{E/\kappa}\eta)}{\partial (r^\kappa)}}\abs{\frac{\partial (r^{\kappa})}{\partial r}}<\epsilon M \kappa r^{\kappa-1}
\end{equation*}
for all $0<r<\delta_2$. For $\delta:=\min\{\delta_1,\delta_2\}$, the preceding estimates guarantee that
\begin{equation}\label{eq:AmplitudeEstimates1}
R(r^{E}\eta)+\widetilde{R}(r^{E}\eta)\geq r^\kappa R(\eta)-\epsilon m r^{\kappa}\geq m(1-\epsilon)r^{\kappa}
\end{equation}
and
\begin{equation}\label{eq:AmplitudeEstimates2}
\abs{\partial_r\left(R(r^{E}\eta)+\widetilde{R}(r^{E}\eta)\right)}\leq \kappa r^{\kappa-1}R(\eta)+\epsilon M \kappa r^{\kappa-1}\leq M(1+\epsilon)\kappa r^{\kappa-1}
\end{equation}
for all $0<r<\delta$ and $\eta\in S$. Thus, for all $n\in\mathbb{N}_+$, $\eta\in S$, and $0<\theta_1\leq\theta_2\leq (\delta n)^\mu$ (equivalently, $0<\rho_1\leq \rho_2\leq \delta$ where $\rho_j:=r_n(\theta_j)=\theta^{1/\mu}_j/n$ for $j=1,2$),
\begin{equation*}
\|g_{n,\eta}\|_{L^{\infty}[\theta_1,\theta_2]}=\sup_{\rho_1\leq r\leq \rho_2}\left(e^{-n(R(r^{E}\eta)+\widetilde{R}(r^{E}\eta))}\right)\leq \sup_{0<r\leq \delta}\left(e^{-n(1-\epsilon)mr^{\kappa}}\right)\leq 1
\end{equation*}
by virtue of \eqref{eq:AmplitudeEstimates1} and the fact that $\epsilon<1$. Appealing to both \eqref{eq:AmplitudeEstimates1} and \eqref{eq:AmplitudeEstimates2}, we have
\begin{eqnarray*}
\|\partial_\theta g_{n,\eta}\|_{L^1[\theta_1,\theta_2]}&=&\int_{\theta_1}^{\theta_2} \abs{\partial_\theta g_{n,\eta}(\theta)}\,d\theta\\
&=&\int_{\theta_1}^{\theta_2}\abs{\partial_{\theta}h_{n,\eta}(r_n(\theta))}\,d\theta\\
&=&\int_{\rho_1}^{\rho_2}\abs{\partial_r h_{n,\eta}(r)}\,dr\\
&=&\int_{\rho_1}^{\rho_2} n\abs{\partial_r\left( R(r^{E}\eta)+\widetilde{R}(r^{E}\eta)\right)}e^{-n\left(R(r^{E}\eta)+\widetilde{R}(r^{E}\eta)\right)}\,dr\\
&\leq& \int_{\rho_1}^{\rho_2} nM(1+\epsilon)\kappa r^{\kappa-1}e^{-n(1-\epsilon)m r^{\kappa}}\,dr\\
&\leq &\frac{(1+\epsilon)M}{(1-\epsilon)m}\int_0^\infty e^{-u}\,du =\frac{\beta M}{m}
\end{eqnarray*}
for all $n\in\mathbb{N}_+$, $\eta\in S$ and $0<\theta_1\leq \theta_2\leq (\delta n)^\mu$. With this, our desired estimate follows without trouble.
\end{proof}

%

\begin{lemma}\label{lem:OscillatoryEstimate}
Let $Q:\mathbb{R}^d\to\mathbb{R}$ be a continuous function for which $\abs{Q}$ is positive homogeneous with unital level set $S$ and homogeneous order $\mu$. Let $R:\mathbb{R}^d\to\mathbb{R}$ be a positive homogeneous function and, given an open neighborhood $\mathcal{U}$ of $0$ in $\mathbb{R}^d$, let $\widetilde{Q}\in C^2(\mathcal{U};\mathbb{R})$ and $\widetilde{R}\in C^1(\mathcal{U};\mathbb{R})$. Suppose that, for some $E\in\Exp(Q)=\Exp(\abs{Q})$, $\widetilde{Q}$ is strongly subhomogeneous with respect to $E$ of order $1$ and, for some $\kappa>0$, $R$ is homogeneous with respect to $E/\kappa$ and $\widetilde{R}$ is strongly subhomogeneous with respect to $E/\kappa$ of order $2$. Define
\begin{equation}\label{eq:DefofPsiPhase}
\Psi_{n,y}(\xi)=nQ(n^{-E}\xi)+n\widetilde{Q}(n^{-E}\xi)+y\cdot\xi=Q(\xi)+n\widetilde{Q}(n^{-E}\xi)+y\cdot\xi
\end{equation}
and
\begin{equation}\label{eq:DefofAAmplitude}
A_n(\xi)=\exp\left(-n\left(R(n^{-E}\xi)+\widetilde{R}(n^{-E}\xi)\right)\right)=\exp(-n^{1-\kappa}R(\xi)-n\widetilde{R}(n^{-E}\xi))
\end{equation}
for $n\in\mathbb{N}_+$, $y\in\mathbb{R}^d$ and $\xi\in\mathbb{R}^d$ for which $n^{-E}\xi\in\mathcal{U}$. If $\mu<1$, then, for each $\epsilon>0$ and compact set $K\subseteq\mathbb{R}^d$, there is a $\delta>0$ and $\tau_0\geq 1$ for which
\begin{equation*}
\abs{\int_{\mathcal{O}_{n\delta}\setminus\mathcal{O}_{\tau}}e^{-i\Psi_{n,y}(\xi)}A_n(\xi)\,d\xi}\leq \epsilon
\end{equation*}
for all $y\in K$, $n\in\mathbb{N}_+$, and $\tau>0$ for which $\tau_0\leq\tau\leq  n\delta$.
\end{lemma}

\begin{proof}
Let $\epsilon>0$ and $K\subseteq\mathbb{R}^d$ be a compact set. Though we will further restrict the size of $\delta$, for the moment, let $\delta>0$ be small enough that $n^{-E}\xi\in\mathcal{U}$ for all $n\in\mathbb{N}_+$ and $\xi\in\mathcal{O}_\delta$. Just as we did leading up to the proof of Theorem \ref{thm:Attractor}, we shall denote by $S$ the unital level set of $\abs{Q}$ and, by virtue of Theorem 1.3 of \cite{BR21}, let $\sigma$ be the Radon measure on $S$ for which \eqref{eq:ApproxFormulaViaPolar} holds. If $\tau\leq n\delta$,
\begin{eqnarray*}
\int_{\mathcal{O}_{n\delta}\setminus\mathcal{O}_{\tau}}e^{-i\Psi_{n,y}(\xi)}A_n(\xi)\,d\xi&=& \int_S\int_{\tau}^{n\delta} e^{-i\Psi_{n,y}(r^{E}\eta)}A_n(r^{E}\eta)r^{\mu-1}\,dr\sigma(d\eta)\\
&=&\int_S I_{\tau,n,y}(\eta)\,\sigma(d\eta)
\end{eqnarray*}
where
\begin{equation*}
I_{\tau,n,y}(\eta)=\int_\tau^{n\delta}e^{-i\Psi_{n,y}(r^{E}\eta)}A_n(r^{E}\eta)r^{\mu-1}\,dr.
\end{equation*}
As in the proof of Theorem \ref{thm:Attractor},  we make the change of variables $r\mapsto \theta^{1/\mu}$ to observe that
\begin{equation*}
I_{\tau,n,y}(\eta)=\int_{\tau^\mu}^{(n\delta)^\mu}e^{-i\Psi_{n,y}(\theta^F\eta)}A_n(\theta^F)\,d\theta\\
=\frac{1}{\mu}\int_{\tau^\mu}^{(n\delta)^\mu}e^{-if_{n,y,\eta}(\theta)}g_{n,\eta}(\theta)\,d\theta
\end{equation*}
where, for $F=E/\mu$,
\begin{equation*}
f_{n,y,\eta}(\theta)=\Psi_{n,y}(\theta^F\eta)=Q(\theta^F\eta)+n\widetilde{Q}(n^{-E}\theta^F\eta)+y\cdot(\theta^F\eta),
\end{equation*}
and
\begin{equation*}
g_{n,\eta}(\theta)=A_n(\theta^F\eta)=\exp\left(-n\left(R(n^{-E}\theta^F\eta)+\widetilde{R}(n^{-E}\theta^F\eta)\right)\right)
\end{equation*}
in the notation of the preceding two lemmas. Our immediate goal is to give a uniform estimate for $\abs{I_{\tau,n,y}(\eta)}$ and we shall do this using the Van der Corput lemma. To this end, we first make an appeal to Lemma \ref{lem:PhaseEstimates} to obtain $\delta_1>0$ and $\theta_0\geq 1$ be such that $\partial_\theta f_{n,y,\eta}(\theta)$ is monotonic on $[\theta_0,(n\delta)^\mu]$ and
\begin{equation*}
\abs{\partial_{\theta}f_{n,y,\eta}(\theta)}\geq \frac{1}{2\mu}\theta^{1/\mu-1}
\end{equation*}
for all $n\in\mathbb{N}_+$, $y\in K$, $\eta\in S$ and  $\theta>0$ such that $\theta_0\leq \theta\leq (n\delta_1)^\mu$. Also, by an appeal to Lemma \ref{lem:AmplitudeEstimates}, choose $\delta_2>0$ for which
\begin{equation*}
\|g_{n,\eta}\|_{L^\infty[\theta_1,\theta_2]}+\|\partial_\theta g_{n,\eta}\|_{L^1[\theta_1,\theta_2]}\leq 1+\frac{2M}{m}
\end{equation*}
for all $n\in\mathbb{N}_+$, $\eta\in S$ and $0<\theta_1\leq\theta_2\leq (n\delta_2)^\mu$ where $m=\inf_{\eta\in S}R(\eta)$ and $M=\sup_{\eta\in S}R(\eta)$, both of which are necessarily finite and positive because $R$ is positive homogeneous. Upon setting $\delta=\min\{\delta_1,\delta_2\}$ and selecting $\tau_0$ such that $\tau_0\geq \theta_0^{1/\mu}$ and
\begin{equation*}
\tau_0^{\mu-1}\leq  \frac{\epsilon}{16(1+2M/m)\sigma(S)}, 
\end{equation*}
we appeal to the Van der Corput lemma, Proposition \ref{prop:VDC}, to see that
\begin{eqnarray*}
\abs{I_{\tau,n,y}(\eta)}&=&\frac{1}{\mu}\abs{\int_{\tau^{\mu}}^{(n\delta)^\mu}e^{-if_{n,y,\eta}(\theta)}g_{n,\eta}(\theta)\,d\theta}\\
&\leq &\frac{4}{\mu}\frac{\|g_{n,y,\eta}\|_{L^\infty[\tau^\mu,(n\delta)^\mu]}+\|\partial_\theta g_{n,y,\eta}\|_{L^\infty[\tau^\mu,(n\delta)^\mu]}}{\inf_{\theta\in[\tau^\mu,(n\delta)^\mu]}\theta^{1/\mu-1}/2\mu}\\
&\leq&4\left(1+\frac{2M}{m}\right)\frac{2}{\tau^{1-\mu}}\\
&\leq &8\left(1+\frac{2M}{m}\right)\tau_0^{\mu-1}\\
&\leq& \frac{\epsilon}{2\sigma(S)}
\end{eqnarray*}
for all $n\in\mathbb{N}_+$, $y\in K$, $\eta\in S$, and $\tau>0$ for which $\tau_0\leq \tau\leq n\delta$. Thus, 
\begin{equation*}
\abs{\int_{\mathcal{O}_{n\delta}\setminus\mathcal{O}_\tau}e^{-i\Psi_{n,y}(\xi)}A_n(\xi)\,d\xi}\leq \int_S\abs{I_{\tau,n,y}(\eta)}\sigma(d\eta)\leq \int_S \frac{\epsilon}{2\sigma(S)}\sigma(d\eta)<\epsilon
\end{equation*}
for all $n\in\mathbb{N}_+$, $y\in K$,  and $\tau>0$ for which $\tau_0\leq \tau\leq n\delta$.
\end{proof}

\noindent With the preceding lemma in hand, we are now in a position to prove a limit statement analogous to Lemma \ref{lem:LocalLimitPosHom} in the case that $\xi_0\in\Omega(\phi)$ is a point of imaginary homogeneous type for $\widehat{\phi}$. 

\begin{lemma}\label{lem:LocalLimitImagHom}
Let $\xi_0\in\Omega(\phi)$ be of imaginary homogeneous type for $\widehat{\phi}$ with associated $\alpha=\alpha_{\xi_0}$, $Q=Q_{\xi_0}$, and $\mu=\mu_{\xi_0}$. If $\mu<1$, then, for any compact set $K\subseteq \mathbb{R}^d$ and $\epsilon>0$, there exists an open neighborhood $\mathcal{U}_{\xi_0}\subseteq\mathbb{T}_{\phi}$ of $\xi_0$, which can be taken as small as desired, and a natural number $N$ for which
\begin{equation*}
\abs{\frac{1}{(2\pi)^d}\int_{\mathcal{U}_{\xi_0}}\widehat{\phi}(\xi)^n e^{-ix\cdot\xi}\,d\xi-\widehat{\phi}(\xi_0)^n e^{-i x\cdot\xi_0} H_{iQ}^n(x-n\alpha)}<\epsilon n^{-\mu}
\end{equation*}
whenever $n\geq N$ and $x\in\mathbb{Z}^d$ is such that $(x-n\alpha)\in n^{E^*}(K)$ for $E\in\Exp(Q)=\Exp(\abs{Q})$.
\end{lemma}

\begin{proof}
Let $\epsilon>0$ and $K\subseteq\mathbb{R}^d$ be a compact set. Set $\mathcal{U}_{\xi_0}=\xi_0+\mathcal{O}_\delta$ where $\delta>0$ is yet to be specified but small enough to ensure that $\mathcal{O}_\delta\subseteq\mathcal{U}$ and $\mathcal{U}_{\xi_0}\subseteq\mathbb{T}_\phi^d$ is as small as desired. For $n\in\mathbb{N}_+$ and $x\in\mathbb{Z}^d$, define
\begin{equation*}
\mathcal{E}=n^{\mu}\abs{\frac{1}{(2\pi)^d}\int_{\mathcal{U}_{\xi_0}}\widehat{\phi}(\xi)^n e^{-ix\cdot\xi}\,d\xi-\widehat{\phi}(\xi_0)^n e^{-i x\cdot \xi_0} H_{iQ}^n(x-n\alpha)}.
\end{equation*}
Also, let $E\in\Exp(\abs{Q})=\Exp(Q)$ be that which appears in Definition \ref{def:Types} and, for $n\in\mathbb{N}_+$ and $x\in\mathbb{Z}^d$, set
\begin{equation*}
y=y_{n,x}=n^{-E^*}(x-n\alpha).
\end{equation*}
We observe that
\begin{eqnarray*}
\mathcal{E}&=&\abs{\frac{n^{\mu}}{(2\pi)^d}\int_{\mathcal{O}_\delta}\widehat{\phi}^n(\xi+\xi_0)e^{-ix\cdot (\xi+\xi_0)}\,d\xi-\widehat{\phi}^n(\xi_0)e^{-i x\cdot\xi_0}H_{iQ}^n(n^{E^*}y)} \\
&=&\abs{\frac{n^{\mu}}{(2\pi)^d}\int_{\mathcal{O}_\delta}\widehat{\phi}^n(\xi_0)e^{-ix\cdot\xi_0}e^{n\Gamma(\xi)}e^{-ix\cdot\xi}\,d\xi-\widehat{\phi}^n(\xi_0)e^{-ix\cdot\xi_0}H_{iQ}^n(n^{E^*}y)}\\
&=&\abs{\frac{n^{\mu}}{(2\pi)^d}\int_{\mathcal{O}_\delta}e^{n\Gamma(\xi)}e^{-ix\cdot\xi}\,d\xi-n^{\mu}H_{iQ}^n(n^{E^*}y)}\\
&=&\abs{\frac{n^\mu}{(2\pi)^d}\int_{\mathcal{O}_\delta}e^{n(\Gamma(\xi)-i\alpha\cdot\xi)}e^{-i(n^{E^*}y)\cdot\xi}\,d\xi-n^{\mu}H_{iQ}^n(n^{E^*}y)}\\
\end{eqnarray*}
for all $n\in\mathbb{N}_+$ and $x\in\mathbb{Z}^d$ where we have made an analogous computation to \eqref{eq:LocalLimitPosHom2} and written $\Gamma=\Gamma_{\xi_0}$. In view of \eqref{eq:GammaExpansion}, making the change of variables $\xi\mapsto n^{-E}\xi$ yields
\begin{eqnarray*}
n^{\mu}\int_{\mathcal{O}_\delta}e^{n(\Gamma(\xi)-i\alpha\cdot\xi)}e^{-i(n^{E^*}y)\cdot\xi}\,d\xi&=&n^{\mu}\int_{\mathcal{O}_\delta}\exp\left(-in\left(Q(\xi)+\widetilde{Q}(\xi)\right)-n\left(R(\xi)+\widetilde{R}(\xi)\right)\right)e^{-iy\cdot (n^{E}\xi)}\,d\xi\\
&=&\int_{\mathcal{O}_{n\delta}}e^{-i\Psi_{n,y}(\xi)}A_{n}(\xi)\,d\xi
\end{eqnarray*}
where $\Psi_{n,y}(\xi)$ and $A_n(\xi)$ are those defined in \eqref{eq:DefofPsiPhase} and \eqref{eq:DefofAAmplitude}, respectively, and we have noted that $n^{E}(\mathcal{O}_\delta)=\mathcal{O}_{n\delta}$. Also, by virtue of \eqref{eq:ImagHomAttractorScale}, we have $n^{\mu}H_{iQ}^n(n^{E^*}y)=H_{iQ}^1(y)$. Consequently,
\begin{equation*}
\mathcal{E}=\abs{\frac{1}{(2\pi)^d}\int_{\mathcal{O}_{n\delta}}e^{-i\Psi_{n,y}(\xi)}A_n(\xi)\,d\xi-H_{iQ}^1(y)}
\end{equation*}
for all $n\in\mathbb{N}_+$ and $x\in\mathbb{Z}^d$. 

Given that $\Sym(\abs{Q})$ is compact by virtue of Proposition \ref{prop:SymPCompact}, let $K'$ be a compact subset of $\mathbb{R}^d$ for which $O^*(K)\subseteq K'$ for all $O\in\Sym(\abs{Q})$. One can take, for example, $K'$ to be the closed ball of radius 
\begin{equation*}
M=\left(\sup_{x\in K}\abs{x}\right)\left(\sup_{O\in\Sym(\abs{Q})}\|O^*\|\right)=\left(\sup_{x\in K}\abs{x}\right)\left(\sup_{O\in \Sym(\abs{Q})}\|O\|\right)<\infty.
\end{equation*}
Suppose that, for some $E'\in\Exp(Q)=\Exp(\abs{Q})$ (which is possibly different from $E$), $n\in\mathbb{N}_+$ and $x\in\mathbb{Z}^d$, $(x-n\alpha)\in n^{E'^*}(K)$, then
\begin{equation*}
y=n^{-E^*}(x-n\alpha)\subseteq n^{-E^*}n^{E'^*}(K)=(n^{E'}n^{-E})^*(K)\subseteq K'
\end{equation*}
by virtue of Proposition \ref{prop:BasicGroupProp} and the simple fact that $n^{E'}n^{-E}\in\Sym(\abs{Q})$. With this observation in mind, to prove the lemma, it suffices to find a $\delta>0$ and an $N\in\mathbb{N}_+$ for which $\mathcal{E}<\epsilon$ for all $n\geq N$ and $y\in K'$.

By virtue of Lemma \ref{lem:OscillatoryEstimate}, let us now (and finally) fix $\delta>0$ and $\tau_0>$ for which
\begin{equation*}
\abs{\frac{1}{(2\pi)^d}\int_{\mathcal{O}_{n\delta}\setminus\mathcal{O}_{\tau}}e^{-i\Psi_{n,y}(\xi)}A_n(\xi)\,d\xi}<\epsilon/3
\end{equation*}
for all $y\in K'$, $n\in\mathbb{N}_+$, and $\tau>0$ for which $\tau_0\leq\tau\leq n\delta$. By an appeal to Theorem \ref{thm:Attractor}, let $\tau_1\geq \tau_0$ be such that
\begin{equation*}
\abs{\frac{1}{(2\pi)^d}\int_{\mathcal{O}_{\tau_1}}e^{-iQ(\xi)-iy\cdot\xi}\,d\xi-H_{iQ}^1(y)}<\epsilon/3
\end{equation*}
for all $y\in K'$. Let us now observe that
\begin{eqnarray*}
\int_{\mathcal{O}_{\tau_1}}\abs{e^{-i\Psi_{n,y}(\xi)}A_n(\xi)-e^{-iQ(\xi)-iy\cdot\xi}}\,d\xi&=&\int_{\mathcal{O}_{\tau_1}}\abs{e^{-n\left(i\widetilde{Q}(n^{-E}\xi)+R(n^{-E}\xi)+\widetilde{R}(n^{-E}\xi)\right)}-1}\,d\xi\\
&=&\int_{\mathcal{O}_{\tau_1}}\abs{e^{-n\widetilde{P}(n^{-E}\xi)}-1}\,d\xi
\end{eqnarray*}
where
\begin{equation*}
\widetilde{P}(\zeta):=i\widetilde{Q}(\zeta)+R(\zeta)+\widetilde{R}(\zeta)
\end{equation*}
for $\zeta\in\mathbb{R}^d$. By our supposition that $\widetilde{Q}$ is subhomogeneous with respect to $E$, $R$ is homogeneous with respect to $E/\kappa$ with $\kappa>1$ and $\widetilde{R}$ is subhomogeneous with respect to $E/\kappa$, from Proposition \ref{prop:Subhomequivtolittleoh} it follows that $\widetilde{P}(\zeta)=o(\abs{Q(\zeta)})$ as $\zeta\to 0$. Thus, given that $\{n^{E}\}$ is a contracting group and $\mathcal{O}_{\tau_1}$ is relatively compact,
\begin{equation*}
n\widetilde{P}(n^{-E}\xi)=\frac{\widetilde{P}(n^{-E}\xi)}{\abs{Q(n^{-E}\xi)}}\abs{Q(\xi)}\to 0\hspace{1cm}\mbox{as}\hspace{1cm}n\to\infty
\end{equation*}
uniformly for $\xi\in \mathcal{O}_{\tau_1}$. Thus, there is a natural number $N\geq \tau_1/\delta$ for which
\begin{equation*}
\frac{1}{(2\pi)^d}\int_{\mathcal{O}_{\tau_1}}\abs{e^{-i\Psi_{n,y}(\xi)}A_n(\xi)-e^{-iQ(\xi)-iy\cdot\xi}}\,d\xi<\frac{\epsilon}{3}
\end{equation*}
for all $n\geq N$ and $y\in K'$. With these estimates, we observe that, for $y\in K'$ and $n\geq N$, $\tau_1<N\delta\leq n\delta$ and
\begin{eqnarray*}
\mathcal{E}&= &\left|\frac{1}{(2\pi)^d}\int_{\mathcal{O}_{n\delta}\setminus \mathcal{O}_{\tau_1}}e^{-i\Psi_{n,y}(\xi)}A_n(\xi)\,d\xi\right.\\
&&\hspace{1.5cm}+\frac{1}{(2\pi)^d}\int_{\mathcal{O}_{\tau_1}}e^{-i\Psi_{n,y}(\xi)}A_n(\xi)\,d\xi-\frac{1}{(2\pi)^d}\int_{\mathcal{O}_{\tau_1}}e^{-iQ(\xi)-iy\cdot\xi}\,d\xi\\
&&\hspace{8.5cm}+\left.\frac{1}{(2\pi)^d}\int_{\mathcal{O}_{\tau_1}}e^{-iQ(\xi)-iy\cdot\xi}\,d\xi-H_{iQ}^1(y)\right|\\
&\leq&\abs{\frac{1}{(2\pi)^d}\int_{\mathcal{O}_{n\delta}\setminus \mathcal{O}_{\tau_1}}e^{-i\Psi_{n,y}(\xi)}A_n(\xi)\,d\xi}\\
&&\hspace{1.5cm}+\frac{1}{(2\pi)^d}\int_{\mathcal{O}_{\tau_1}}\abs{e^{-i\Psi_{n,y}(\xi)}A_n(\xi)-e^{-iQ(\xi)-iy\cdot\xi}}\,d\xi\\
&&\hspace{5cm}+\abs{\frac{1}{(2\pi)^d}\int_{\mathcal{O}_{\tau_1}}e^{-iQ(\xi)-iy\cdot\xi}\,d\xi-H_{iQ}^1(y)}\\
&<&\epsilon,
\end{eqnarray*}
as desired.
\end{proof}

\begin{lemma}\label{lem:SupNormEstImagHom}
Let $\xi_0\in\Omega(\phi)$ be of imaginary homogeneous type for $\widehat{\phi}$ with associated drift $\alpha_{\xi_0}\in\mathbb{R}^d$, homogeneous order $\mu_{\xi_0}>0$, and polynomial $Q_{\xi_0}$. If $\mu_{\xi_0}<1$, then, for any compact set $K\subseteq\mathbb{R}^d$, there exists an open neighborhood $\mathcal{U}_{\xi_0}\subseteq\mathbb{T}_\phi^d$ of $\xi_0$, which can be taken as small as desired, and a constant $C_{\xi_0}$ for which
\begin{equation*}
\abs{\frac{1}{(2\pi)^d}\int_{\mathcal{U}_{\xi_0}}\widehat{\phi}(\xi)^ne^{-ix\cdot\xi}\,d\xi}\leq C_{\xi_0}n^{-\mu_{\xi_0}}
\end{equation*}
whenever $x\in\mathbb{Z}^d$ and $n\in\mathbb{N}_+$ is such that $(x-n\alpha_{\xi_0})\in n^{E^*}(K)$ for $E\in\Exp(Q_{\xi_0})=\Exp(|Q_{\xi_0}|)$.
\end{lemma}

\begin{proof}
By virtue of the continuity of $y\mapsto H_{iQ_{\xi_0}}^1(y)$ ensured by Theorem \ref{thm:Attractor}, we note that
\begin{equation*}
M:=\sup_{y\in K}\abs{H_{iQ_{\xi_0}}(y)}<\infty
\end{equation*}
and so, thanks to Item \ref{item:ImagHomAttractorScale} of Theorem \ref{thm:Attractor}, it follows that
\begin{equation*}
\abs{H_{iQ_{\xi_0}}^n(x-n\alpha_{\xi_0})}=n^{-\mu_{\xi_0}}\abs{H_{iQ_{\xi_0}}^1\left(n^{-E^*}(x-n\alpha_{\xi_0})\right)}\leq n^{-\mu_{\xi_0}}M
\end{equation*}
whenever $x\in n\alpha_{\xi_0}+n^{E^*}(K)$ for $E\in\Exp(Q_{\xi_0})$. Consequently, an application of Lemma \ref{lem:LocalLimitImagHom} and the triangle inequality yield $C_{\xi_0}>0$ for which the desired estimate holds for $n$ sufficiently large. By modifying $C_{\xi_0}$, if necessary, we obtain the desired estimate for all $n\in\mathbb{N}_+$.
\end{proof}

\begin{theorem}\label{thm:LLTMain}
Let $\phi\in\mathcal{S}_d$ be such that $\sup_{\xi}|\widehat{\phi}(\xi)|=1$ and suppose that $\Omega(\phi)$ consists only of points of positive homogeneous or imaginary homogeneous type for $\widehat{\phi}$ and let $\mu_\phi>0$ be the homogeneous order of $\phi$ defined by \eqref{eq:HomOrderPhi}. Denote by $\Omega_p(\phi)$ and $\Omega_i(\phi)$ those points of $\Omega(\phi)$ of positive homogeneous type and imaginary homogeneous type for $\widehat{\phi}$, respectively. If $\Omega_i(\phi)$ is non-empty, assume additionally that
\begin{enumerate}[(i)]
\item\label{hyp:LLTMainLessThanOne} $\mu_{\xi}<1$ for every $\xi\in\Omega_i(\phi)$, 
\item\label{hyp:LLTMainDomImag} there exists $\xi\in\Omega_i(\phi)$ for which $\mu_{\xi}=\mu_\phi$,
\end{enumerate}
and
\begin{enumerate}[(i),resume]
\item\label{hyp:LLTMainSameDrift} there exists $\alpha\in\mathbb{R}^d$ for which $\alpha_\xi=\alpha$ for all $\xi\in\Omega_i(\phi)$.
\end{enumerate}
Then, there holds the following.
\begin{enumerate}[\text{Case} 1, align = left] 
\item\label{item:LLTMain1} In the case that $\Omega(\phi)=\Omega_p(\phi)$, i.e., every point of $\Omega(\phi)$ is of positive homogeneous type for $\widehat{\phi}$, let $\{\xi_1,\xi_2,\dots,\xi_A\}$ be the set of points in $\Omega(\phi)$ for which $\mu_{\xi_k}=\mu_\phi$ and, for each $k=1,2,\dots,A$, set $\alpha_k=\alpha_{\xi_k}$ and $P_k=P_{\xi_k}=R_{\xi_k}+iQ_{\xi_k}$ in view of Definition \ref{def:Types} and let $H_{P_k}$ be the corresponding attractor defined in Proposition \ref{prop:PosHomKernel}. Then,
\begin{equation*}
\phi^{(n)}(x)=\sum_{k=1}^A\widehat{\phi}(\xi_k)^ne^{-ix\cdot\xi_k}H_{P_k}^n(x-n\alpha_k)+o(n^{-\mu_\phi})
\end{equation*}
uniformly for $x\in\mathbb{Z}^d$.
\item\label{item:LLTMain2} Alternatively, in the case that $\Omega_i(\phi)$ is nonempty, we have the following:
\begin{enumerate}
\item\label{item:LLTMain2a} If $\Omega(\phi)=\Omega_i(\phi)$ or $\mu_{\xi}>\mu_{\phi}$ for all $\xi\in\Omega_p(\phi)$, let $\{\xi_1,\xi_2,\dots,\xi_A\}$ be the set of points in $\Omega(\phi)$ for which $\mu_{\xi_k}=\mu_\phi$. Necessarily, $\{\xi_1,\xi_2,\dots,\xi_A\}\subseteq\Omega_i(\phi)$ and so, for each $k=1,2,\dots,A$, let $Q_k=Q_{\xi_k}$ in view of Definition \ref{def:Types} and let $H_{iQ_k}$ be the corresponding attractor defined in Theorem \ref{thm:Attractor}. Then, for each compact set $K\subseteq \mathbb{R}^d$, there exists a nested increasing sequence of compact sets $\{K_n\}$ all containing $K$ and whose union is $\mathbb{R}^d$ for which
\begin{equation*}
\phi^{(n)}(x)=\sum_{k=1}^A\widehat{\phi}(\xi_k)^ne^{-ix\cdot\xi_k}H_{iQ_k}^n(x-n\alpha)+o(n^{-\mu_\phi})
\end{equation*}
uniformly for $x\in \left(n\alpha+K_n\right)\bigcap\mathbb{Z}^d$.

\item\label{item:LLTMain2b} If, otherwise, $\Omega_i(\phi)$ and $\Omega_p(\phi)$ both contain points of order $\mu_\phi$, we shall denote by $\{\xi_1,\xi_2,\dots,\xi_{A_i}\}$ those points of $\Omega_i(\xi)$ for which $\mu_{\xi_k}=\mu_\phi$ for $k=1,2,\dots, A_i$. Also, denote by $\{\zeta_1,\zeta_2,\dots,\zeta_{A_p}\}$ those points of $\Omega_p(\xi)$ for which $\mu_{\zeta_j}=\mu_\phi$ for $j=1,2,\dots A_p$. For each $k=1,2\dots, A_i$, let $Q_k=Q_{\xi_k}$ be that given in Definition \ref{def:Types} and let $H_{iQ_k}$ be the associated attractor given in Theorem \ref{thm:Attractor}. For each $j=1,2\dots, A_p$, let $\alpha_{j}=\alpha_{\zeta_j}\in\mathbb{R}^d$ and $P_j=P_{\zeta_j}=R_{\zeta_j}+iQ_{\zeta_j}$ be those given by Definition \ref{def:Types} and let $H_{P_j}$ be the associated attractor given in Propsition \ref{prop:PosHomKernel}. Then, for each compact set $K\subseteq\mathbb{R}^d$, there exists a nested increasing sequence of compact sets $\{K_n\}$ all containing $K$ and whose union is $\mathbb{R}^d$ for which
\begin{equation*}
\phi^{(n)}(x)=\sum_{k=1}^{A_i}\widehat{\phi}(\xi_k)^ne^{-ix\cdot\xi_k}H_{iQ_k}^n(x-n\alpha)+\sum_{j=1}^{A_p}\widehat{\phi}(\zeta_j)^ne^{-ix\cdot\zeta_j}H^n_{P_j}(x-n\alpha_j)+ o(n^{-\mu_\phi})
\end{equation*}
uniformly for $x\in (n\alpha+K_n)\bigcap\mathbb{Z}^d$.
\end{enumerate}
\end{enumerate}
\end{theorem}

\begin{proof} Throughout the proof, we will write $\mu=\mu_\phi$.

\begin{subproof}[\ref{item:LLTMain1}.]
Let $\epsilon>0$. In view of our supposition that every point of $\Omega(\phi)$ is of positive homogeneous type for $\widehat{\phi}$, we write
\begin{equation*}
\Omega(\phi)=\{\xi_1,\xi_2,\dots,\xi_A,\dots,\xi_{A'}\}
\end{equation*}
where $\mu_{\xi_k}=\mu$ for $k=1,2,\dots A$ and, if $A<k\leq A'$, $\mu_{\xi_k}>\mu_\phi$. For $k=1,2,\dots, A$, let $\alpha_k=\alpha_{\xi_k}\in\mathbb{R}^d$ and $P_k=P_{\xi_k}=R_{\xi_k}+iQ_{\xi_k}$ be those associated to $\xi_k$ in view of Definition \ref{def:Types}. By making simultaneous appeals to Lemma \ref{lem:LocalLimitPosHom} and Lemma \ref{lem:SupNormEstPositiveHom}, we may choose a disjoint collection of open sets $\mathcal{U}_{1},\mathcal{U}_2,\dots,\mathcal{U}_{A'}\subseteq \mathbb{T}_\phi^d$, each containing $\xi_k\in\Omega(\phi)$ for the appropriate $k=1,2\dots, A'$, and a natural number $N_1$ for which, if $k=1,2,\dots, A$,
\begin{equation}\label{eq:LLTMainCase1_1}
\abs{\frac{1}{(2\pi)^d}\int_{\mathcal{U}_k}\widehat{\phi}(\xi)^ne^{-ix\cdot\xi}\,d\xi-\widehat{\phi}(\xi_k)^ne^{-ix\cdot\xi_k}H_{P_k}^n(x-n\alpha_k)}<\frac{\epsilon n^{-\mu}}{2A'}
\end{equation}
for all $n\geq N_1$, $x\in\mathbb{Z}^d$ and, if $A<k\leq A'$,
\begin{equation*}
\abs{\frac{1}{(2\pi)^d}\int_{\mathcal{U}_k}\widehat{\phi}(\xi)^ne^{-ix\cdot\xi}\,d\xi}< C_k n^{-\mu_{\xi_k}}
\end{equation*}
for all $n\in\mathbb{N}_+$ and $x\in\mathbb{Z}^d$; here, $C_k>0$. We remark that, for $A<k\leq A'$, $n^{-\mu_{\xi_k}}=o(n^{-\mu})$ as $n\to\infty$ because $\mu<\mu_{\xi_k}$. Thus, there exists a natural number $N_2$ for which
\begin{equation}\label{eq:LLTMainCase1_2}
\abs{\frac{1}{(2\pi)^d}\int_{\mathcal{U}_k}\widehat{\phi}(\xi)^ne^{-ix\cdot\xi}\,d\xi}<\frac{\epsilon n^{-\mu}}{2A'}
\end{equation}
for $n\geq N_2$, $x\in\mathbb{Z}^d$ and $A<k\leq A'$. For the compact set
\begin{equation*}
G=\mathbb{T}_\phi^d\setminus\left(\bigcup_{k=1}^{A'} \mathcal{U}_k\right),
\end{equation*}
we observe that, for each $n\in\mathbb{N}_+$ and $x\in\mathbb{Z}^d$,
\begin{equation*}
\abs{\frac{1}{(2\pi)^d}\int_G \widehat{\phi}(\xi)^n e^{-ix\cdot\xi}\,d\xi}\leq \frac{1}{(2\pi)^d}\int_G |\widehat{\phi}(\xi)|^n\,d\xi\leq s^n
\end{equation*}
where
\begin{equation*}
s:=\sup_{\xi\in G}|\widehat{\phi}(\xi)|<1
\end{equation*}
because $\Omega(\phi)\subseteq \bigcup_{k=1}^{A'} \mathcal{U}_k$. Upon noting that $s^n=o(n^{-\mu})$ as $n\to\infty$, we can therefore select $N_3\in\mathbb{N}_+$ for which
\begin{equation}\label{eq:LLTMainCase1_3}
\abs{\frac{1}{(2\pi)^d}\int_G\widehat{\phi}(\xi)^ne^{-ix\cdot\xi}\,d\xi}\leq s^n<\frac{\epsilon n^{-\mu}}{2}
\end{equation}
for all $n\geq N_3$ and $x\in\mathbb{Z}^d$. By virtue of the identity \eqref{eq:FTConvolutionPhi}, observe that
\begin{eqnarray}\label{eq:LLTMainCase1_4}\nonumber
\lefteqn{\hspace{-0.5cm}\abs{ \phi^{(n)}(x)-\sum_{k=1}^A\widehat{\phi}(\xi_k)^ne^{-ix\cdot\xi_k}H_{P_k}^n(x-n\alpha_k)}}\\\nonumber
\hspace{0cm}&=&\abs{\frac{1}{(2\pi)^d}\int_{\mathbb{T}_\phi^d}\widehat{\phi}(\xi)^ne^{-ix\cdot\xi}\,d\xi-\sum_{k=1}^A\widehat{\phi}(\xi_k)^ne^{-ix\cdot\xi_k}H_{P_k}^n(x-n\alpha_k)}\\\nonumber
&=&\left|\sum_{k=1}^{A'}\frac{1}{(2\pi)^d}\int_{\mathcal{U}_k}\widehat{\phi}(\xi)^ne^{-ix\cdot\xi}\,d\xi+\frac{1}{(2\pi)^d}\int_{G}\widehat{\phi}(\xi)^ne^{-ix\cdot\xi}\,d\xi\right.\\\nonumber
&&\hspace{4cm}\left.-\sum_{k=1}^A\widehat{\phi}(\xi_k)^ne^{-ix\cdot\xi_k}H_{P_k}^n(x-n\alpha_k)\right|\\\nonumber
&=&\left|\sum_{k=1}^A\left(\frac{1}{(2\pi)^d}\int_{\mathcal{U}_k}\widehat{\phi}(\xi)^ne^{-ix\cdot\xi}\,d\xi-\widehat{\phi}(\xi_k)^ne^{-ix\cdot\xi_k}H_{P_k}^n(x-n\alpha_k)\right)\right.\\\nonumber
&&\hspace{0cm}\left.+\sum_{k=A+1}^{A'}\frac{1}{(2\pi)^d}\int_{\mathcal{U}_k}\widehat{\phi}(\xi)^ne^{-ix\cdot\xi}\,d\xi+\frac{1}{(2\pi)^d}\int_{G}\widehat{\phi}(\xi)^ne^{-ix\cdot\xi}\,d\xi\right|\\\nonumber
&\leq&\sum_{k=1}^A\abs{\frac{1}{(2\pi)^d}\int_{\mathcal{U}_k}\widehat{\phi}(\xi)^ne^{-ix\cdot\xi}\,d\xi-\widehat{\phi}(\xi_k)^ne^{-ix\cdot\xi_k}H_{P_k}^n(x-n\alpha_k)}\\
&&+\sum_{k=A+1}^{A'}\abs{\frac{1}{(2\pi)^d}\int_{\mathcal{U}_k}\widehat{\phi}(\xi)^n e^{-ix\cdot\xi}\,d\xi}+\abs{\frac{1}{(2\pi)^d}\int_{G}\widehat{\phi}(\xi)^ne^{-ix\cdot\xi}\,d\xi}
\end{eqnarray}
for all $n\in\mathbb{N}_+$ and $x\in\mathbb{Z}^d$. Making use of the estimates \eqref{eq:LLTMainCase1_1}, \eqref{eq:LLTMainCase1_2}, and \eqref{eq:LLTMainCase1_3}, the preceding inequality guarantees that
\begin{equation*}
\abs{ \phi^{(n)}(x)-\sum_{k=1}^A\widehat{\phi}(\xi_k)^ne^{-ix\cdot\xi_k}H_{P_k}^n(x-n\alpha_k)}< A'\left(\frac{\epsilon n^{-\mu}}{2A'}\right)+\frac{\epsilon n^{-\mu}}{2}=\epsilon n^{-\mu}
\end{equation*}
whenever $n\geq N=\max\{N_1,N_2,N_3\}$ and $x\in\mathbb{Z}^d$, as desired.
\end{subproof}
\begin{subproof}[\ref{item:LLTMain2a}.]
Let $\epsilon>0$ and $K\subseteq \mathbb{R}^d$ be compact. In view of our supposition, we can write
\begin{equation*}
\Omega(\phi)=\{\xi_1,\xi_2,\dots,\xi_A,\cdots,\xi_{A'}\}
\end{equation*}
where, for each $k=1,2,\dots, A$, $\xi_k$ is a point of imaginary homogeneous type for $\widehat{\phi}$ with homogeneous order $\mu_{\xi_k}=\mu<1$ and, for $A<k\leq A'$, $\xi_k$ is a point of positive homogeneous type or imaginary homogeneous type for $\widehat{\phi}$ with homogeneous order $\mu<\mu_{\xi_k}$; in the case that $\xi_k$ is of imaginary homogeneous type, we also know that $\mu<\mu_{\xi_k}<1$. 

We now construct the sequence of compact sets $\{K_n\}$. For each $k=1,2,\dots,A'$ for which $\xi_k$ is of imaginary homogeneous type for $\widehat{\phi}$ with associated $Q_k=Q_{\xi_k}$, let $E_k\in\Exp(Q_k)=\Exp(\abs{Q_k})$. Define $\psi_k:[0,\infty)\times \mathbb{R}^d\to\mathbb{R}^d$ by
\begin{equation*}
\psi_k(s,x)=\begin{cases}
s^{E_k^*}x & s>0\\
0 & s=0
\end{cases}
\end{equation*}
for $s\geq 0$ and $x\in\mathbb{R}^d$. Upon noting that $\{t^{E_k^*}\}$ is a contracting group (this follows from Proposition \ref{prop:CharofPosHom} and the fact that $(t^{E_k})^*=t^{E_k^*}$), it is straightforward to verify that $\psi_k$ is continuous. Let us fix a closed ball $\overline{B}=\overline{B_R(0)}$ containing $K$ and observe that
\begin{equation*}
J_k:=\psi_k([0,1]\times \overline{B})=\bigcup_{0<r\leq 1}r^{E_k^*}(\overline{B})
\end{equation*}
which is necessarily compact by virtue of the continuity of $\psi_k$. Thanks to the above identity,  we see that, for any $t\geq 1$, we have
\begin{equation*}
J_k\subseteq \bigcup_{0<s\leq t}s^{E_k^*}(\overline{B})=\bigcup_{0<r\leq 1} (t\cdot r)^{E_k^*}(\overline{B})=t^{E_k^*}\left(\bigcup_{0<r\leq 1}r^{E_k^*}(\overline{B})\right)=t^{E_k^*}(J_k).
\end{equation*}
Consequently, for any natural numbers $n\leq m$, we have
\begin{equation*}
n^{E_k^*}(J_k)\subseteq m^{E_k^*}(J_k).
\end{equation*}
With the above observations in mind, and in view of the fact that $\{t^{E_k^*}\}$ is contracting, it follows that
\begin{equation*}
J_k^n:=n^{E_k^*}(J_k)
\end{equation*}
is a nested increasing sequence of compact sets all containing $K$ and whose union is $\mathbb{R}^d$. Finally, for each $n\in\mathbb{N}_+$, we define
\begin{equation*}
K_n=\bigcap_{\substack{k=1,2,\dots,A' \\ \xi_k\in\Omega_i(\phi)}}J_k^n=\bigcap_{\substack{k=1,2,\dots,A' \\ \xi_k\in\Omega_i(\phi)}}n^{E_k^*}(J_k)
\end{equation*}
From our construction of $J_k^n$, it is straightforward to see that $\{K_n\}$ is a nested increasing sequence of compact sets all containing $K$ and whose union is $\mathbb{R}^d$.

We now make simultaneous appeals to Lemma \ref{lem:LocalLimitImagHom}, Lemma \ref{lem:SupNormEstPositiveHom}, and Lemma \ref{lem:SupNormEstImagHom} to produce a collection of disjoint open sets $\mathcal{U}_1,\mathcal{U}_2,\dots\mathcal{U}_{A'}\subseteq\mathbb{T}_\phi^d$, each containing $\xi_k\in\Omega(\phi)$ for the appropriate $k=1,2,\dots A'$, and a natural number $N_1$ for which the following estimates hold. For $k=1,2,\dots, A$,
\begin{equation}\label{eq:LLTMainCase1_5}
\abs{\frac{1}{(2\pi)^d}\int_{\mathcal{U}_k}\widehat{\phi}(\xi)^n e^{-ix\cdot\xi}\,d\xi-\widehat{\phi}(\xi_k)^{n}e^{-ix\cdot\xi_k}H_{iQ_k}^n(x-n\alpha)}<\frac{\epsilon n^{-\mu}}{2A'}
\end{equation}
whenever $n\geq N_1$, $x\in\mathbb{Z}^d$, and $(x-n\alpha)\in K_n\subseteq n^{E_k^*}(J_k)$. If $A<k\leq A'$ and $\xi_k\in\Omega_p(\phi)$, 
\begin{equation*}
\abs{\frac{1}{(2\pi)^d}\int_{\mathcal{U}_k}\widehat{\phi}(\xi)^ne^{-ix\cdot\xi}\,d\xi}<C_k n^{-\mu_{\xi_k}}
\end{equation*}
for all $n\in\mathbb{N}_+$ and $x\in\mathbb{Z}^d$. If $A<k\leq A'$ and $\xi_k\in\Omega_i(\phi)$, 
\begin{equation*}
\abs{\frac{1}{(2\pi)^d}\int_{\mathcal{U}_k}\widehat{\phi}(\xi)^ne^{-ix\cdot\xi}\,d\xi}<C_k n^{-\mu_{\xi_k}}
\end{equation*}
whenever $n\in\mathbb{N}_+$, $x\in\mathbb{Z}^d$ and $(x-n\alpha)\in K_n\subseteq n^{E_k^*}(J_k)$. Upon noting that $n^{-\mu_{\xi_k}}=o(n^{-\mu})$ as $n\to\infty$ whenever $A<k\leq A'$, we may further choose a natural number $N_2$ for which
\begin{equation}\label{eq:LLTMainCase1_6}
\abs{\frac{1}{(2\pi)^d}\int_{\mathcal{U}_k}\widehat{\phi}(\xi)^ne^{-ix\cdot\xi}\,d\xi}<\frac{\epsilon n^{-\mu}}{2A'}
\end{equation}
whenever $n\geq N_1$, $x\in\mathbb{Z}^d$ and $(x-n\alpha)\in K_n$. Finally, as in \ref{item:LLTMain1}, we select a natural number $N_3$ for which
\begin{equation}\label{eq:LLTMainCase1_7}
\abs{\frac{1}{(2\pi)^d}\int_{G}\widehat{\phi}(\xi)^ne^{-ix\cdot\xi}\,d\xi}<\frac{\epsilon n^{-\mu}}{2}
\end{equation}
for all $n\geq N_3$ and $x\in\mathbb{Z}^d$ where $G=\mathbb{T}_\phi^d\setminus\bigcup_{k=1}^{A'}\mathcal{U}_k$. By an analogous argument to that given in \ref{item:LLTMain1}, we combine the estimates \eqref{eq:FTConvolutionPhi}, \eqref{eq:LLTMainCase1_5}, \eqref{eq:LLTMainCase1_6}, and \eqref{eq:LLTMainCase1_7} to find that
\begin{equation*}
\abs{\phi^{(n)}(x)-\sum_{k=1}^A\widehat{\phi}(\xi_k)^ne^{-ix\cdot\xi_k}H_{iQ_k}^n(x-n\alpha)}<\epsilon n^{-\mu}
\end{equation*}
for all $n\geq N=\max\{N_1,N_2,N_3\}$ and $x\in\mathbb{Z}^d$ for which $(x-n\alpha)\in K_n$.
\end{subproof}
\begin{subproof}[\ref{item:LLTMain2b}.]
The proof of this final case is almost identical to the proof of \ref{item:LLTMain2a}. The only difference is that, in addition to appealing to Lemma \ref{lem:LocalLimitImagHom} to handle those points $\xi_k\in\Omega_i(\phi)$ for which $\mu_{\xi_k}=\mu$, one also appeals to Lemma \ref{lem:LocalLimitPosHom} to obtain local limits for those points $\eta_j\in\Omega_p(\phi)$ for which $\mu_{\eta_j}=\mu$. We leave this argument to the interested and committed reader.
\end{subproof}
\end{proof}

\begin{proof}[Proof of Theorem \ref{thm:LLTIntro}]
We assume that $\Omega(\phi)=\{\xi_0\}$ where $\xi_0$ is of positive homogeneous type for $\widehat{\phi}$ or $\xi_0$ is of imaginary homogeneous type for $\widehat{\phi}$ with $\mu_\phi=\mu_{\xi_0}<1$. In the positive-homogeneous case, the stated local limit theorem follows directly from \ref{item:LLTMain1} of Theorem \ref{thm:LLTMain} with $A=1$. In the imaginary-homogeneous case, the hypotheses \ref{hyp:LLTMainLessThanOne}-\ref{hyp:LLTMainSameDrift} of Theorem \ref{thm:LLTMain} are met automatically and we find ourselves in \ref{item:LLTMain2a} with $A=1$ because $\Omega(\phi)=\Omega_i(\phi)=\{\xi_0\}$. Let us select a compact set $K\subseteq\mathbb{R}^d$ and, using the argument given in the proof of Lemma \ref{lem:LocalLimitImagHom}, choose a compact set $K'\subseteq\mathbb{R}^d$ for which $O^*(K)\subseteq K'$ for all $O\in \Sym(\abs{Q_{\xi_0}})$. With $K'$ in hand, an appeal to Theorem \ref{thm:LLTMain} guarantees that
\begin{equation*}
\phi^{(n)}(x)=\widehat{\phi}(\xi_0)^ne^{-ix\cdot\xi_0}H_{iQ_{\xi_0}}^n(x-n\alpha_{\xi_0})+o(n^{-\mu_{\phi}})
\end{equation*}
uniformly for $x\in(n\alpha_{\xi_0}+K_n)\cap\mathbb{Z}^d$ where $K_n$ is a sequence of compact sets containing $K'$ and whose union is $\mathbb{R}^d$. Upon careful study of the construction of the sets $K_n$ in the proof of Theorem \ref{thm:LLTMain}, we see that, for a chosen $E_0\in \Exp(\abs{Q_{\xi_0}})$,
\begin{equation*}
n^{E_0^*}(K')\subseteq n^{E_0^*}(\overline{B})\subseteq n^{E_0^*}(J_0)=K_n
\end{equation*}
for each $n\in\mathbb{N}_+$ where $\overline{B}=\overline{B_R(0)}$ is a closed ball containing $K'$ and $J_0=\bigcup_{0<r\leq 1}r^{E_0^*}(\overline{B})$. Consequently, for any $E\in\Exp(Q_{\xi_0})=\Exp(\abs{Q_{\xi_0}})$ and $n\in\mathbb{N}_+$, we have
\begin{equation*}
n^{E^*}(K)=n^{E_0^*}(n^{E}n^{-E_0})^*(K)\subseteq n^{E_0^*}(K')\subseteq K_n
\end{equation*}
because $n^{E}n^{-E_0}\in\Sym(\abs{Q_{\xi_0}})$. Thus, for any $E\in\Exp(Q_{\xi_0})=\Exp(\abs{Q_{\xi_0}})$, we have
\begin{equation*}
\phi^{(n)}(x)=\widehat{\phi}(\xi_0)^ne^{-ix\cdot\xi_0}H_{iQ_{\xi_0}}^n(x-n\alpha_{\xi_0})+o(n^{-\mu_{\phi}})
\end{equation*}
uniformly for $x\in (n\alpha_{\xi_0}+n^{E^*}(K))\cap\mathbb{Z}^d$.
\end{proof}

\section{Sup-norm Type Estimates}\label{sec:SupNormEst}

\noindent In this short section, we will discuss sup-norm type estimates for convolution powers of functions $\phi\in\mathcal{S}_d$ satisfying the hypotheses of Theorem \ref{thm:LLTMain}. In the case that $\Omega(\phi)=\Omega_p(\phi)$, i.e., every point of $\Omega(\phi)$ is of positive homogeneous type for $\widehat{\phi}$, Theorem 1.4 of \cite{RSC17} guarantees positive constants $C,C'>0$ for which
\begin{equation}\label{eq:SupNormPos}
Cn^{-\mu_\phi}\leq \|\phi^{(n)}\|_\infty\leq C'n^{-\mu_\phi}
\end{equation}
for all $n\in\mathbb{N}_+$ where
\begin{equation*}
\|\phi^{(n)}\|_{\infty}:=\sup_{x\in\mathbb{Z}^d}\abs{\phi^{(n)}(x)}.
\end{equation*} 
The upper estimate in \eqref{eq:SupNormPos} can be established by piecing together lemmas of the type Lemma \ref{lem:SupNormEstPositiveHom} or by making an appeal to the local limit theorem of \ref{item:LLTMain1} of Theorem \ref{thm:LLTMain}, both of which hold uniformly for $x\in\mathbb{Z}^d$. Using the local limit theorem, the lower estimate in \eqref{eq:SupNormPos} is established using a generalization of E. Artin's result on the linear independence of characters to ensure that the sum
\begin{equation*}
\sum_{k=1}^A e^{-ix\cdot\xi_k}\widehat{\phi}(\xi_k)^n H_{P_{\xi_k}}^n(x-n\alpha_k)
\end{equation*}
cannot uniformly collapse. This result on the independence of characters appears as Lemma 4.5 in \cite{RSC17} and is proved by induction; an analytic proof that makes use of the Stone-Weierstrass theorem can be found in Section 5.5 of \cite{KhK21}. In the recent article \cite{BR21}, it is established that, for $\phi\in\mathcal{S}_d$ which satisfies the hypotheses of Theorem 1 and, in addition, has $\alpha_\xi=\alpha=0$ for all $\xi\in\Omega_i(\phi)$, that, to each compact set $K$ there is a constant $C>0$ for which
\begin{equation}\label{eq:SupNormImagHuanAndMe}
\abs{\phi^{(n)}(x)}\leq Cn^{-\mu_\phi}
\end{equation}
for all $n\in\mathbb{N}_+$ and $x\in K\cap\mathbb{Z}^d$; this is Theorem 3.6 of \cite{BR21}. Our final result in this section gives a matching lower bound for \eqref{eq:SupNormImagHuanAndMe} and, in addition, asserts that \eqref{eq:SupNormImagHuanAndMe} holds on sets of the form $K_n$ as appearing in Theorem \ref{thm:LLTMain}. 

\begin{theorem}\label{thm:SupNormEst}
Assume that $\phi\in\mathcal{S}_d$ satisfies the hypotheses of Theorem \ref{thm:LLTMain} and that $\Omega_i(\phi)$ is nonempty. Then, for each compact set $K\subseteq\mathbb{R}^d$, there exists a constant $C'>0$ and a sequence of compact sets $\{K_n\}$ all containing $K$ and whose union is $\mathbb{R}^d$ such that
\begin{equation*}
\abs{\phi^{(n)}(x)}\leq C' n^{-\mu_\phi}
\end{equation*}
for all $x\in (n\alpha+K_n)\cap\mathbb{Z}^d$ and $n\in\mathbb{N}_+$. Furthermore, there is a constant $C>0$ for which
\begin{equation*}
C n^{-\mu_\phi}\leq \|\phi^{(n)}\|_\infty
\end{equation*}
for all $n\in\mathbb{N}_+$.
\end{theorem}

\begin{proof}
As in \ref{item:LLTMain2a}, we assume that $\Omega(\phi)=\Omega_i(\phi)$ or $\mu_\xi>\mu_\phi$ for all $\xi\in\Omega_p(\phi)$; the proof for the situation in  \ref{item:LLTMain2b} is similar and omitted. We first establish the upper bound. Given $K\subseteq\mathbb{R}^d$ compact, construct $\{K_n\}$ in precisely the same way as in the proof of Theorem \ref{thm:LLTMain}. This collection of sets $\{K_n\}$ has the stated properties and also has the property that, for $k=1,2,\dots,A$ and $n\in\mathbb{N}_+$,
\begin{equation*}
K_n\subseteq n^{E_k^*}(J_k)
\end{equation*}
where $J_k$ is a compact set and $E_k\in\Exp(Q_k)$. For each $k=1,2,\dots,A$, we set
\begin{equation*}
C_k=\sup_{y\in J_k}\abs{H_{iQ_k}^1(y)}
\end{equation*}
(which is finite by virtue of the continuity of $H_{iQ_k}^1$) and observe that
\begin{equation}\label{eq:AttractorSupEstimate}
\abs{H_{iQ_k}^n(x-n\alpha)}=n^{-\mu_\phi}\abs{H_{iQ_k}^1(n^{-E_k^*}(x-n\alpha))}\leq C_k n^{-\mu_\phi}
\end{equation}
whenever $x\in\mathbb{Z}^d$ is such that $x-n\alpha\in K_n\subseteq n^{E^*_k}(J_k)$ for each $n\in\mathbb{N}_+$; here we have used Theorem \ref{thm:Attractor}.  The desired estimate now follows directly from the local limit theorem of \ref{item:LLTMain2b}, \eqref{eq:AttractorSupEstimate}, and the triangle inequality. 

We now establish the lower estimate. Thanks to the continuity of attractors and our supposition that $\alpha=\alpha_{\xi_k}$ for all $k=1,2,\dots, A$, the argument for the lower bound in the proof of Theorem 1.4 of \cite{RSC17} can be made successfully with only minor modification. This guarantees that, for some $\epsilon>0$ and $N\in\mathbb{N}_+$,
\begin{equation*}
n^{\mu_\phi}\|\phi^{(n)}\|_{\infty}\geq \epsilon \abs{H_{iQ_1}(0)}
\end{equation*}
for all $n\geq N$. By virtue of Item \ref{item:ImagHomAttractorAtOrigin} of Theorem \ref{thm:Attractor}, $\abs{H_{iQ_1}(0)}>0$ and so we find that
\begin{equation*}
\|\phi^{(n)}\|_\infty\geq Cn^{-\mu_\phi}
\end{equation*}
for all $n\geq N$ where $C=\epsilon \abs{H_{iQ_1}(0)}>0$. By modifying $C>0$ to account for those $n\in\mathbb{N}_+$ for which $n< N$, if necessary, we obtain our desired result at once.
\end{proof}

\section{Examples}\label{sec:Examples}

In this section, we illustrate the results of Theorem \ref{thm:LLTIntro}, Theorem \ref{thm:LLTMain}, and Theorem \ref{thm:SupNormEst} by presenting concrete examples of finitely-supported functions $\phi:\mathbb{Z}^d\to\mathbb{C}$ for which $\sup_\phi|\widehat{\phi}(\xi)|=1$ and where the points of $\Omega(\phi)$ are all of positive homogeneous or imaginary homogeneous type for $\widehat{\phi}$. As the article \cite{RSC17} focuses on the situation in which $\Omega(\phi)$ consists strictly of points of positive homogeneous type for $\widehat{\phi}$, we remark that Section 7 therein contains several examples to which \ref{item:LLTMain1} of Theorem \ref{thm:LLTMain} (and Theorem 1.6 of \cite{RSC17}) applies and each illustrates a different feature of the relevant local limit (e.g., multiple attractors with distinct drifts, an attractor $H_P$ given by a polynomial $P=R+iQ$ for which $R$ is positive homogeneous yet $P$ is not of the form \eqref{eq:SemiEllipticPolynomial}). Consistent with focus of this article on the imaginary-homogeneous setting, below we consider examples to which \ref{item:LLTMain2} of Theorem \ref{thm:LLTMain} applies -- none of which can be treated by the results of \cite{RSC17}. To this end,  we first state a useful proposition which gives sufficient conditions for a point $\xi_0\in\Omega(\phi)$ to be of positive homogeneous or imaginary homogeneous type for $\hat{\phi}$ in terms of the Taylor expansion for $\Gamma_{\xi_0}$. This proposition can be found as Proposition 3.8 of \cite{BR21}; we have included a proof in Subsection \ref{subsec:HomAndSubHom} of the Appendix for completeness.

\begin{proposition}\label{prop:ExpandGamma}
Let $\phi\in\mathcal{S}_d$ with $\sup_{\xi}|\widehat{\phi}(\xi)|=1$. For $\xi_0\in \Omega(\phi)$, suppose that there exists $\mathbf{m}\in \mathbb{N}^d_+$ and $k \geq 1$ such that the Taylor expansion of $\Gamma_{\xi_0} : \mathcal{U}\to\mathbb{C}$ centered at $0$ is a series of the form
\begin{eqnarray}\label{eq:SemiEllipticImaginaryExpansion}\nonumber
    \Gamma_{\xi_0}(\xi) 
    &=& i\alpha \cdot \xi - i \left( \sum_{\abs{\beta : 2\mathbf{m}} \geq 1} A_\beta \xi^\beta\right) - \sum_{\abs{\beta : 2\mathbf{m}} \geq k} B_\beta \xi^\beta \\ \nonumber
    &=& i\alpha \cdot \xi - i \left( \sum_{\abs{\beta : 2\mathbf{m}} = 1} A_\beta \xi^\beta + \sum_{\abs{\beta : 2\mathbf{m}} > 1} A_\beta \xi^\beta\right) \nonumber\\
    &&\hspace{3cm} - \left( \sum_{\abs{\beta : 2\mathbf{m}} = k} B_\beta \xi^\beta + \sum_{\abs{\beta : 2\mathbf{m}} > k} B_\beta \xi^\beta \right) \nonumber\\
    &=&  i\alpha \cdot \xi - i\left( Q(\xi) + \widetilde{Q}(\xi)\right) - \left( R(\xi) + \widetilde{R}(\xi) \right),
\end{eqnarray}
where $\alpha \in \mathbb{R}^d$;   $Q$ and $R$ are real-valued polynomials for which $R$ is positive definite; and  $\widetilde{Q}$ and $\widetilde{R}$ are real multivariate power series which are absolutely and uniformly convergent on $\mathcal{U}$. If $k=1$, then $\xi_0$ is of positive homogeneous type for $\widehat{\phi}$ with associated polynomials $R_{\xi_0}=R$, $Q_{\xi_0}=Q$ and $P_{\xi_0}=R_{\xi_0}+iQ_{\xi_0}$. If $k>1$ and $|Q|$ is positive definite, then $\xi_0$ is of imaginary homogeneous type for $\hat{\phi}$ with associated polynomial $Q_{\xi_0}=Q$. In either case, $\xi_0$ has drift $\alpha_{\xi_0}=\alpha$ and homogeneous order
\begin{equation*}
    \mu_{\xi_0}=\abs{\mathbf{1}:2\mathbf{m}}=\sum_{j=1}^d\frac{1}{2m_j}.
\end{equation*}
\end{proposition}

\begin{example}[The details for Example \ref{ex:OneDEx1}]
Upon simplification, we find that
\begin{equation*}
\widehat{\phi}(\eta,\gamma)=\frac{1}{3} \left(-\frac{1}{16} i (\sin (\eta) \cos (\gamma)-\sin (\eta))-\sin ^4\left(\frac{\eta}{2}\right)-\frac{1}{2} i \sin ^2\left(\frac{\eta}{2}\right)-\sin
   ^8\left(\frac{\gamma}{2}\right)-\frac{1}{2} i \sin ^4\left(\frac{\gamma}{2}\right)+3\right)
   \end{equation*}
   for $\xi=(\eta,\gamma)\in\mathbb{R}^d$. With this, it is easy to verify that
   \begin{equation*}
       1=\sup_{\xi\in\mathbb{R}^2}|\widehat{\phi}(\xi)|=\widehat{\phi}(0,0)
   \end{equation*}
   and that $\xi_0=(0,0)$ is the unique point in $\mathbb{T}^2$ at which $|\widehat{\phi}(\xi)|$ is maximized, i.e., $\Omega(\phi)=\{\xi_0\}$. By a direct computation, we find
\begin{equation*}
    \Gamma_{\xi_0}(\xi)=-i\left(Q(\xi)+\widetilde{Q}(\xi)\right)-\left(R(\xi)+\widetilde{R}(\xi)\right)
\end{equation*}
where
\begin{equation*}
    Q(\xi)=\sum_{|\beta:(2,4)|=1}A_\beta \xi^\beta=\frac{\eta^2}{24}-\frac{\eta\gamma^2}{96} +\frac{ \gamma^4}{96}=\frac{1}{96}\left(4\eta^2-\eta\gamma^2+\gamma^4\right),
\end{equation*}
\begin{equation*}
    R(\xi)=\sum_{|\beta:(2,4)|=2}B_\beta \xi^\beta=\frac{23\eta^4}{1152}  + \frac{\eta^3\gamma^2}{2304}  - \frac{\eta^2\gamma^4}{2048} + \frac{\eta\gamma^6}{9216}+ \frac{23\gamma^8}{18432},
\end{equation*}
\begin{equation*}
    \widetilde{Q}(\xi)=\sum_{|\beta:(2,4)|\geq 3/2}A_\beta \xi^\beta=- \frac{\eta^4}{288}+\frac{\eta\gamma^4}{1152} +\frac{\eta^3\gamma^2}{576}  -\frac{\eta^3\gamma^4}{6912} 
    + \cdots,
\end{equation*}
and
\begin{equation*}
    \widetilde{R}(\xi)=\sum_{|\beta:(2,4)|\geq 5/2}B_\beta \xi^\beta
    =  - \frac{\eta^3\gamma^4}{27648} + \frac{\eta^4\gamma^4}{18432} +  \frac{\eta^4\gamma^4}{18432} + \cdots
\end{equation*}
for $\xi=(\eta,\gamma)\in\mathcal{U}$ where $\mathcal{U}\subseteq\mathbb{R}^2$ is a neighborhood of $0$. We remark that this is of the form \eqref{eq:SemiEllipticImaginaryExpansion} with $\alpha_{\xi_0}=(0,0)$, $2\mathbf{m}=(2,4)$ and $k=2$. It is readily verified that
\begin{equation*}
    \abs{Q(\xi)}=Q(\xi)=\frac{1}{96}(4\eta^2-\eta\gamma^2+\gamma^4)
\end{equation*}
and
\begin{equation*}
    R(\xi)=\frac{23\eta^4}{1152}  + \frac{\eta^3\gamma^2}{2304}  - \frac{\eta^2\gamma^4}{2048} + \frac{\eta\gamma^6}{9216}+ \frac{23\gamma^8}{18432}
\end{equation*}
are both positive definite and so Proposition \ref{prop:ExpandGamma} ensures that $\xi_0=(0,0)$ is of imaginary homogeneous type for $\widehat{\phi}$ with drift $\alpha_{\xi_0}=(0,0)$, associated homogeneous polynomial $Q_{\xi_0}=Q$, and homogeneous order
\begin{equation*}
    \mu_\phi=\mu_{\xi_0}=\frac{1}{2}+\frac{1}{4}=3/4<1.
\end{equation*}
Thus, by an appeal to Theorem \ref{thm:LLTIntro}, we obtain the stated local limit theorem \eqref{eq:TwoDEx1_1}. We see also that Theorems \ref{thm:LLTMain} and \ref{thm:SupNormEst} apply; the former agrees with \eqref{eq:TwoDEx1_1} and the latter yields \eqref{eq:TwoDEx1_2} and \eqref{eq:TwoDEx1_3}.
\end{example}

\begin{example}
This example illustrates a complex-valued function $\phi$ on $\mathbb{Z}^2$ whose Fourier transform is maximized in absolute value at two distinct points in $\mathbb{T}^2$ and falls within the scope of \ref{item:LLTMain2a} of Theorem \ref{thm:LLTMain}. Specifically, $\Omega(\phi)$ consists of two points, one of which is a point of imaginary homogeneous type for $\widehat{\phi}$ with homogeneous order $2/3=\mu_\phi$ and the other is a point of positive homogeneous type for $\widehat{\phi}$ of homogeneous order $1>\mu_\phi$. This $\phi$ appeared in Example 7 in \cite{BR21}. As we will see, our local limit theorem consists of a single attractor of the form $H_{iQ}$ and we note that, in this case, Theorem \ref{thm:LLTIntro} does not apply because $\Omega(\phi)$ is not a singleton.  

Consider $\phi: \mathbb{Z}^2 \to \mathbb{C}$ defined by $\phi=2^{-7}\phi_1-i2^{-11}\phi_2+2^{-21}\phi_3$ where
\begin{equation*}
    \phi_1(x,y)=\begin{cases}
    15 + 15i &(x,y) = (\pm 1,0)\\
    16 + 16i &(x,y) = (0, \pm 1)\\
     1 + 1i &(x,y) = (\pm 3,0)\\
    0 &\mbox{otherwise}
    \end{cases},
    \hspace{1cm}
    \phi_2(x,y) = 
    \begin{cases}
    682 &(x,y) = (0,0)\\
    152  &(x,y) = (\pm 2,0)\\
    -28  &(x,y) = (\pm 4,0)\\
    8 &(x,y) = (\pm 6, 0)\\
    -1 &(x,y) = (\pm 8, 0)\\
    60  &(x,y) = (0, \pm 2)\\
    -24 &(x,y) = (0,\pm 4)\\
    4 &(x,y) = (0,\pm 6)\\
    0 &\mbox{otherwise}
    \end{cases},
\end{equation*}
and
\begin{equation*}
    \phi_3(x,y) = 
    \begin{cases}
    1387004 &(x,y) = (0,0)\\
    -106722 &(x,y) = (\pm 2,0)\\
    3960 &(x,y) = (\pm 4,0)\\
    -1045 &(x,y) = (\pm 6, 0)\\
    138  &(x,y) = (\pm 8, 0)\\
    -9 &(x,y) = (\pm 10, 0)\\
    -131072 &(x,y) = (0, \pm 2)\\
    0 &\mbox{otherwise}
    \end{cases}
\end{equation*}
for $(x,y)\in\mathbb{Z}^2$. Though this example is slightly more complicated than the previous one, it is straightforward to verify that $\sup_{\xi}|\widehat{\phi}(\xi)|=1$  and $\Omega(\phi)=\{\xi_1,\zeta_1\}$ where $\xi_1=(0,0)$ and $\zeta_1=(\pi,\pi)$. For $\xi_1$, $\widehat{\phi}(\xi_1)=1$ and
\begin{equation*}
    \Gamma_{\xi_1}(\xi)=-i\left(Q_{\xi_1}(\xi)-\widetilde{Q_{\xi_1}}(\xi)\right)-\left(R_{\xi_1}(\xi)+\widetilde{R_{\xi_1}}(\xi)\right)
\end{equation*}
for $\xi=(\tau,\gamma)\in\mathcal{U}_{\xi_1}$ where $\mathcal{U}_{\xi_1}\subseteq\mathbb{R}^2$ is an open neighborhood of $(0,0)$ and
\begin{equation*}
    Q_{\xi_1}(\xi)=\sum_{|\beta:(6,2)|=1}A_\beta \xi^\beta=\frac{\tau^6}{128}+\frac{\gamma^2}{8},
\end{equation*}
\begin{equation*}
    R_{\xi_1}(\xi)=\sum_{|\beta:(6,2)|=2}B_\beta\xi^\beta=\frac{111\tau^{12}}{32768}-\frac{\tau^6 \gamma^2}{1024}+\frac{3\gamma^4}{128},
\end{equation*}
\begin{equation*}
    \widetilde{Q_{\xi_1}}(\xi)=\sum_{|\beta:(6,2)|\geq 4/3}A_\beta \xi^\beta=-\frac{65\tau^8}{512} -\frac{\gamma^4}{96} + \frac{\tau^6\gamma^4}{8192} 
    +\cdots,
\end{equation*}
and
\begin{equation*}
    \widetilde{R_{\xi_1}}(\xi)=\sum_{|\beta:(6,2)|\geq 7/3}B_\beta\xi^\beta=\frac{65\tau^8\gamma^2}{4096}  + \frac{\tau^6 \gamma^4}{12288} + \cdots
\end{equation*}
for $\xi=(\tau,\gamma)\in\mathcal{U}_{\xi_1}$. We observe that $Q_{\xi_1}=\abs{Q_{\xi_1}}$ and $R_{\xi_1}$ are positive definite and so Proposition \ref{prop:ExpandGamma} guarantees that $\xi_1=(0,0)$ is of imaginary homogeneous type for $\widehat{\phi}$ with $\mathbf{m}=(3,1)$, $k=2$, drift $\alpha_{\xi_1}=(0,0)$ and homogeneous order
\begin{equation*}
    \mu_{\xi_1}=|\mathbf{1}:2\mathbf{m}|=\frac{1}{6}+\frac{1}{2}=\frac{2}{3}.
\end{equation*}
For $\zeta_1= (\pi,\pi)$,
\begin{equation*}
    \Gamma_{\zeta_1}(\xi)=-i\left(Q_{\zeta_1}(\xi)+\widetilde{Q_{\zeta_1}}(\xi)\right)-\left(R_{\zeta_1}(\xi)+\widetilde{R_{\zeta_1}}(\xi)\right)
\end{equation*}
for $\xi=(\tau,\gamma)\in\mathcal{U}_{\zeta_1}$ where $\mathcal{U}_{\zeta_1}\subseteq\mathbb{R}^2$ is an open neighborhood of $(0,0)$ and 
\begin{equation*}
    Q_{\zeta_1}(\xi)=\sum_{|\beta:(2,2)|=1}A_\beta \xi^\beta=-\left(\frac{3\tau^2}{8}+\frac{\gamma^2}{4}\right),
\end{equation*}
\begin{equation*}
    R_{\zeta_1}(\xi)=\sum_{|\beta:(2,2)|=1}B_\beta\xi^\beta=\frac{\tau^2}{8}+\frac{3\gamma^2}{8},
\end{equation*}
\begin{equation*}
    \widetilde{Q_{\zeta_1}}(\xi)=\sum_{|\beta:(2,2)|\geq 2}A_\beta\xi^\beta=\frac{\tau^4}{64}-\frac{9\tau^2\gamma^2}{64}+\frac{\gamma^4}{48}+\cdots,
\end{equation*}
and
\begin{equation*}
    \widetilde{R_{\zeta_1}}(\xi)=\sum_{|\beta:(2,2)|\geq 2}B_\beta\xi^\beta=-\frac{\tau^4}{8}-\frac{3\tau^2\gamma^2}{64}-\frac{13\gamma^4}{384}+\cdots
\end{equation*}
for $\xi=(\tau,\gamma)\in\mathcal{U}_{\zeta_1}$. Thus, the expansion is of the form \eqref{eq:SemiEllipticImaginaryExpansion} with $\mathbf{m}=(1,1)$ and $k=1$. Since $R_{\zeta_1}$ is clearly positive definite, Proposition \ref{prop:ExpandGamma} guarantees that $\zeta_1=(\pi,\pi)$ is of positive homogeneous type for $\widehat{\phi}$ with drift $\alpha_{\zeta_1}=(0,0)$ and homogeneous order
\begin{equation*}
    \mu_{\zeta_1}=|\mathbf{1}:2\mathbf{m}|=\frac{1}{2}+\frac{1}{2}=1.
\end{equation*}
Thus $\mu_\phi=\min\{\mu_{\xi_1},\mu_{\zeta_1}\}=\mu_{\xi_1}=2/3<1$, $\alpha=\alpha_{\xi_1}=(0,0)$ and so we have met the hypotheses of Theorem \ref{thm:LLTMain}. For the situation at hand, \ref{item:LLTMain2a} is applicable so we are guaranteed that, to each compact set $K\subseteq\mathbb{R}^2$, there is a nested collection of compact set $\{K_n\}$ all containing $K$ and whose union is $\mathbb{R}^2$ for which
\begin{equation*}
\phi^{(n)}(x)=H_{iQ_{\xi_1}}^n(x)+o(n^{-2/3})
\end{equation*}
uniformly for $x\in K_n\cap \mathbb{Z}^2$ where
\begin{equation*}
H_{iQ_{\xi_1}}^n(x)=n^{-2/3}H_{iQ_{\xi_1}}^1(n^{-1/6}x_1,n^{-1/2}x_2)
\end{equation*}
for $x=(x_1,x_2)\in\mathbb{R}^2$ since $E^*=E\in\Exp(Q_{\xi_1})$ has standard matrix representation $\diag(1/6,1/2)$. This local limit theorem is illustrated in Figure \ref{fig:TwoDEx3} wherein $\Re(\phi^{(n)})$ and $\Re(H_{iQ_{\xi_1}}^n)$ are shown for $n=300$ and $n=600$ on the grid $K=[-50,50]^2$.

\begin{figure}[!htb]
\begin{center}
\resizebox{\textwidth}{!}{	
	    \begin{subfigure}{0.5\textwidth}
		\includegraphics[width=\textwidth]{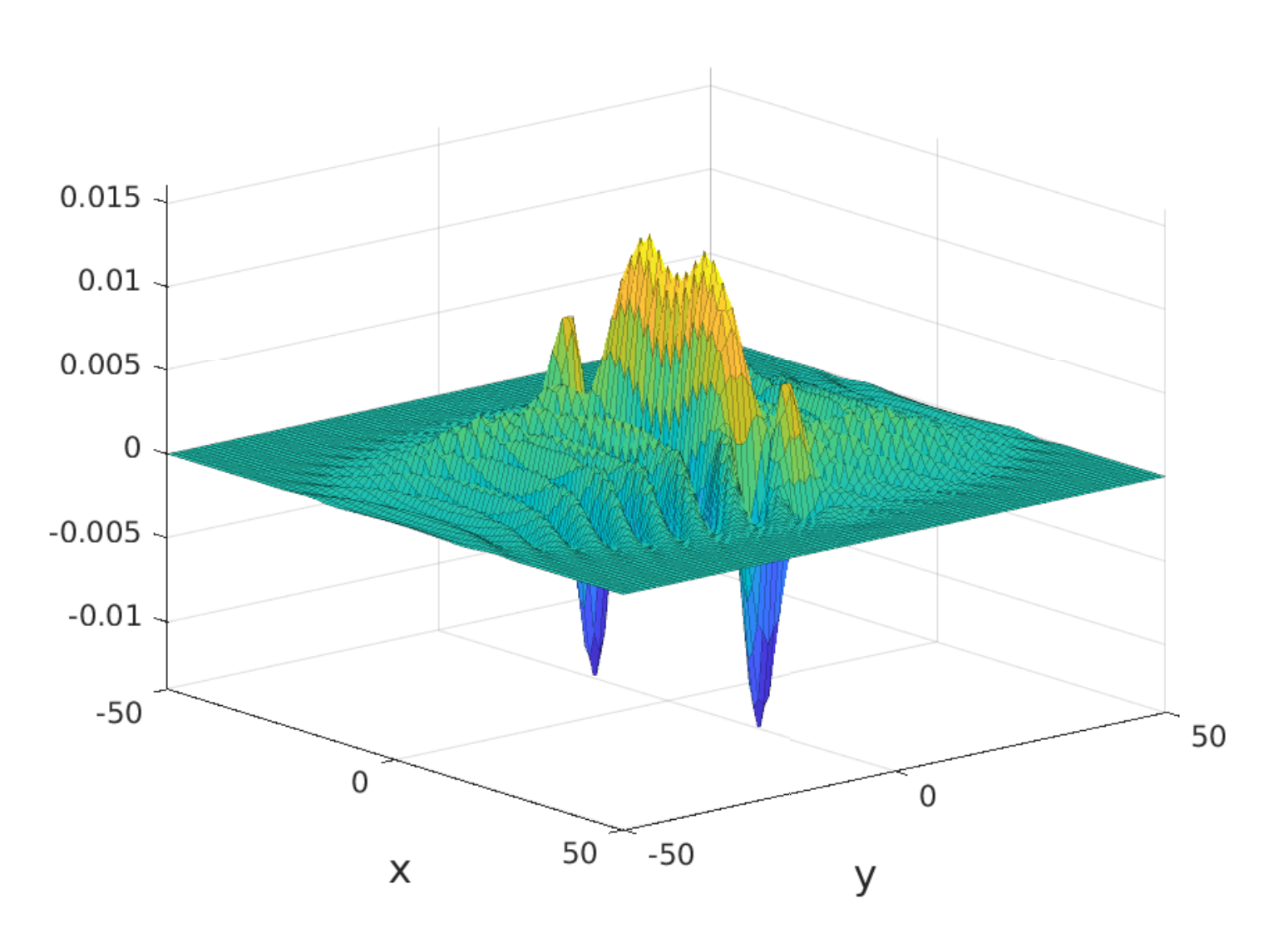}
		\caption{$\Re(\phi^{(n)})$ for $n=200$}
		\label{fig:TwoDEx3_Phi_200}
	    \end{subfigure}
	    \begin{subfigure}{0.5\textwidth}
		\includegraphics[width=\textwidth]{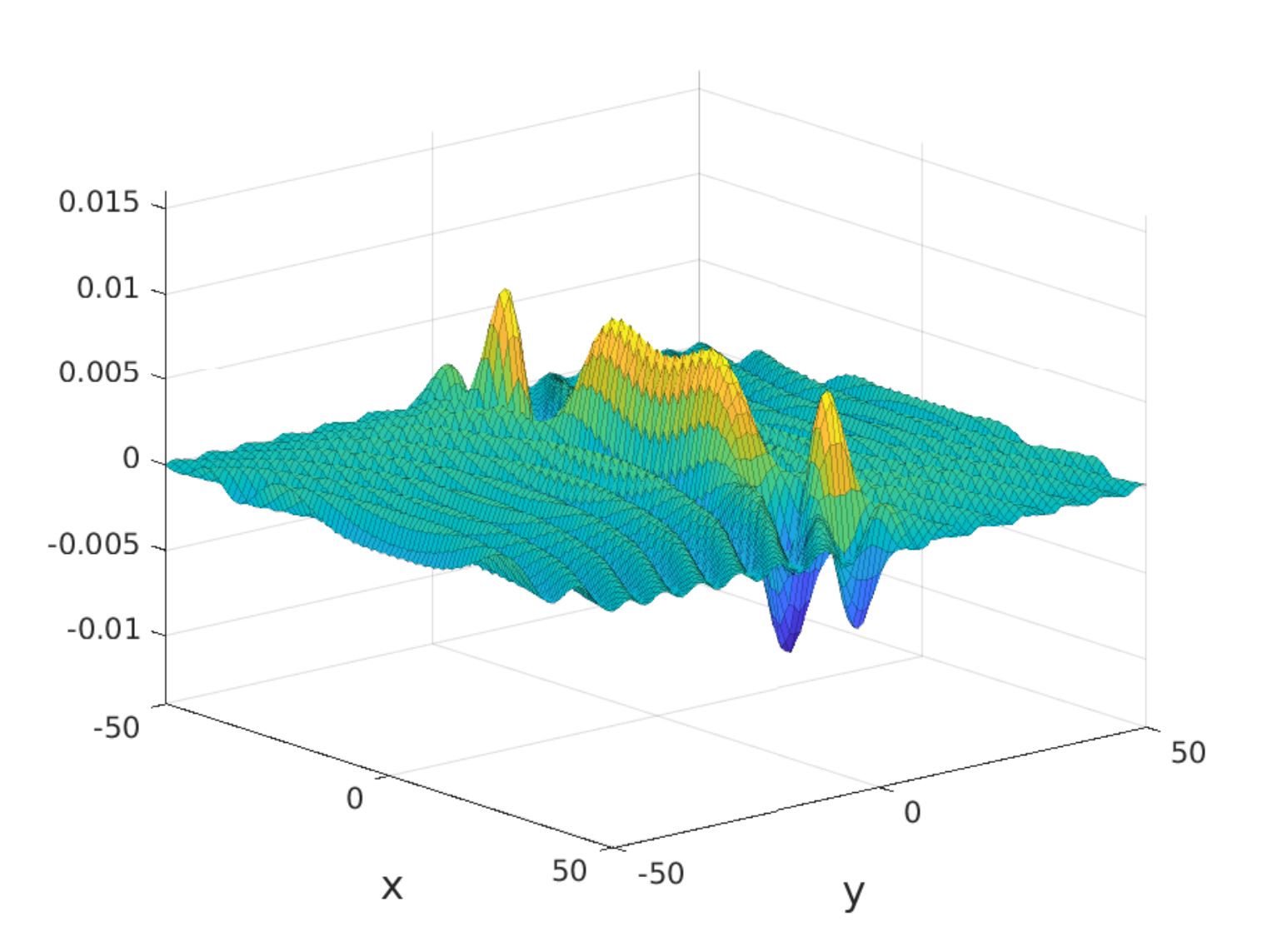}
		\caption{$\Re(\phi^{(n)})$ for $n=400$}
		\label{fig:TwoDEx3_Phi_400}
	    \end{subfigure}}
\resizebox{\textwidth}{!}{	
	    \begin{subfigure}{0.5\textwidth}
		\includegraphics[width=\textwidth]{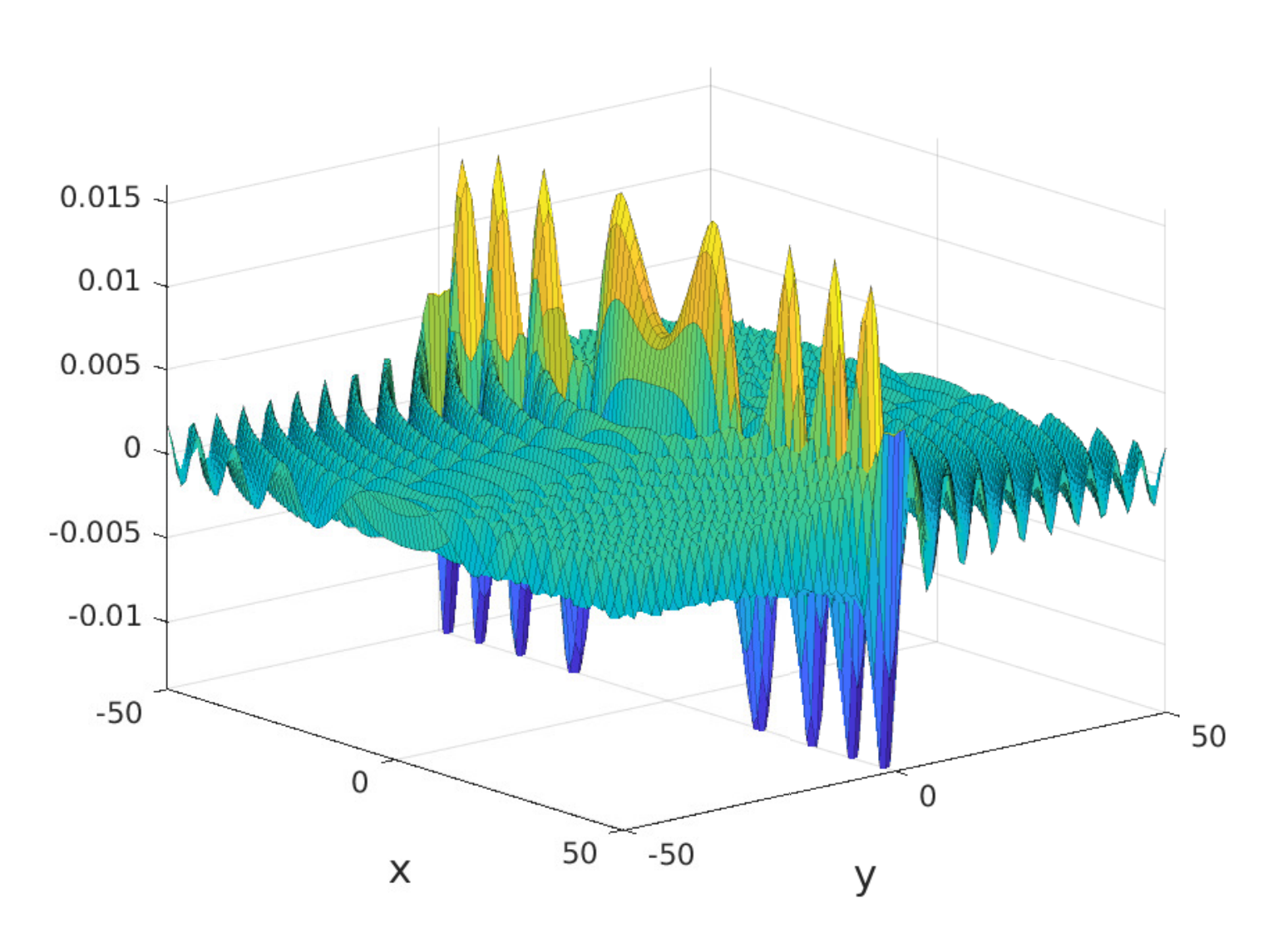}
		\caption{$\Re(H_{iQ_{\xi_1}}^n)$ for $n=200$}
		\label{fig:TwoDEx3_H_200}
	    \end{subfigure}
	    \begin{subfigure}{0.5\textwidth}
		\includegraphics[width=\textwidth]{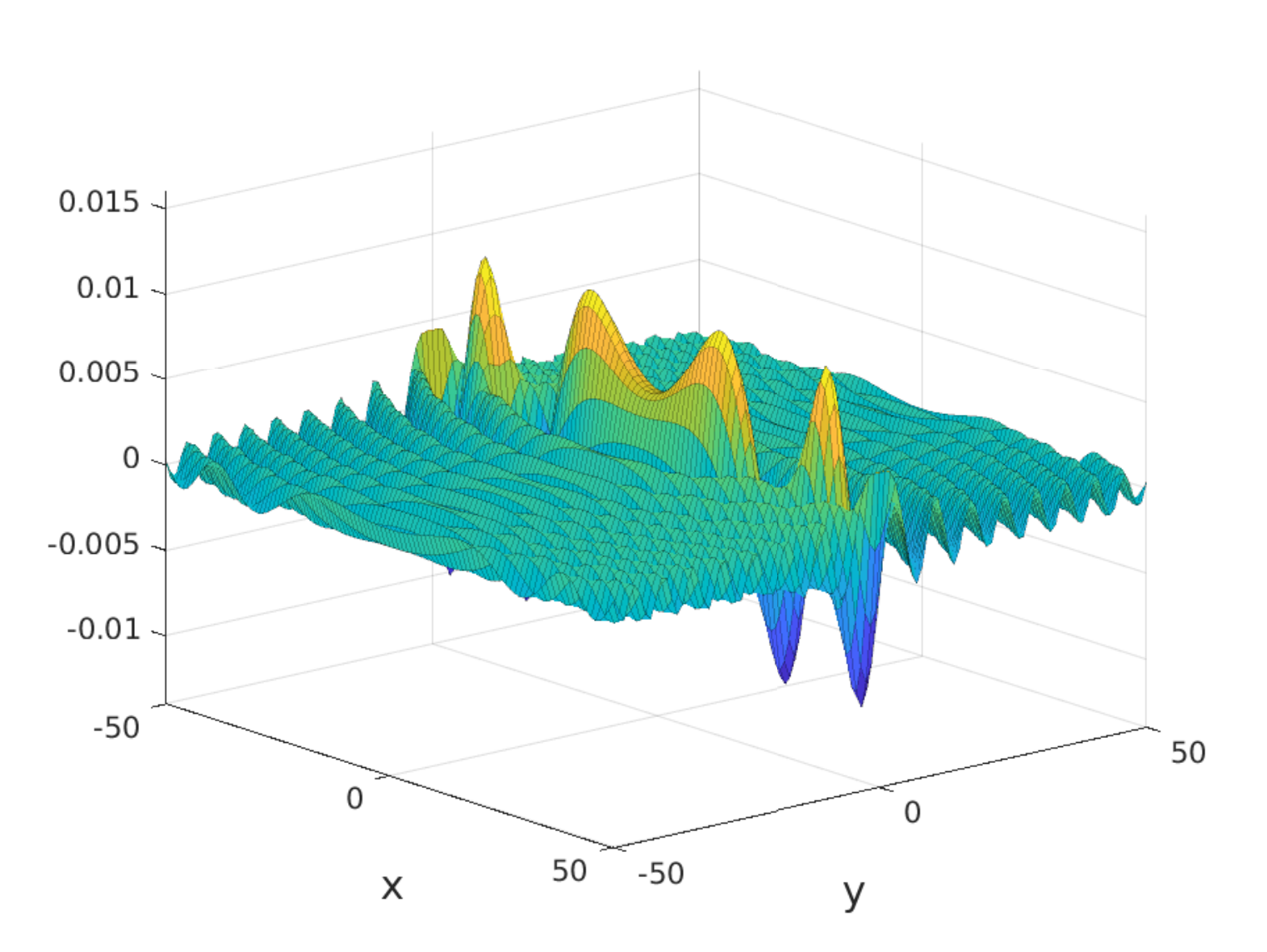}
		\caption{$\Re(H_{iQ_{\zeta_1}}^n)$ for $n=400$}
		\label{fig:TwoDEx3_H_400}
	    \end{subfigure}}
\caption{The graphs of $\Re(\phi^{(n)})$ and $\Re(H_{iQ_{\xi_1}}^n)$ for $n=200$ and $n=400$.}
\label{fig:TwoDEx3}
\end{center}
\end{figure}

Finally, Theorem \ref{thm:SupNormEst} gives positive constants $C,C'>0$ for which
\begin{equation*}
\abs{\phi^{(n)}(x)}\leq\frac{C'}{n^{2/3}}
\end{equation*}
whenever $x\in K_n\cap\mathbb{Z}^2$ and
\begin{equation*}
\frac{C}{n^{2/3}}\leq \|\phi^{(n)}\|_\infty
\end{equation*}
for all $n\in\mathbb{N}_+$.
\end{example}

\begin{example}
This example illustrates a complex-valued function on $\mathbb{Z}^2$ whose Fourier transform is maximized in absolute value at two points in $\mathbb{T}^2$ and falls within the scope of \ref{item:LLTMain2b} of Theorem \ref{thm:LLTMain}. Specifically, $\Omega(\phi)=\{\xi_1,\zeta_1\}$ where $\xi_1$ is of imaginary homogenous type for $\widehat{\phi}$, $\zeta_1$ is of positive homogeneous type for $\widehat{\phi}$ and $\mu_{\xi_1}=\mu_{\zeta_1}=\mu_\phi=1/2$. In this case, the local limit theorem consists of a sum of attractors, one of which is of the form $H_{iQ}$ and the other of the form $H_{P}$. 

Consider $\psi:\mathbb{Z}\to\mathbb{C}$ defined by
\begin{equation*}
    \psi(z)=
    \frac{1}{2^{12}}\begin{cases}
    1292(2+i) & z=0\\
    \left(552-540i\right) & z=\pm 1\\
    -\left(177-\frac{499}{2}i\right) & z=\pm 2\\
    -\left(28-10i\right) & z=\pm 3\\
    \left(42-59i\right) & z=\pm 4\\
    -\left(12- 18i\right) & z=\pm 5\\
    \left(1-\frac{3}{2}i\right) & z=\pm 6\\
    0 & \mbox{else}
    \end{cases}
\end{equation*}
for $z\in\mathbb{Z}$. With this, we define $\phi:\mathbb{Z}^2\to\mathbb{C}$ by
\begin{equation*}
    \phi(x,y)=\begin{cases}
    \psi(x) & \mbox{if }y=0\\
    \psi(y) & \mbox{if }x=0\\
    0 & \mbox{ otherwise}
    \end{cases}
\end{equation*}
for $(x,y)\in\mathbb{Z}^2$.
It is straightforward to verify that $\sup_{\xi}|\widehat{\phi}(\xi)|=1$ and $\Omega(\phi)=\{(0,0),(\pi,\pi)\}$. At $\xi_1=(0,0)$, we have $\widehat{\phi}(\xi_1)=1$ and 
\begin{equation*}
    \Gamma_{\xi_1}(\xi)=-i\left(Q_{\xi_1}(\xi)+\sum_{|\beta:(4,4)|>3/2
    }A_\beta \xi^\beta\right)-\left(R_{\xi_1}(\xi)+\sum_{|\beta:(4,4)|>5/2}B_\beta\xi^\beta\right)
\end{equation*}
for $\xi\in\mathcal{U}$ where
\begin{equation*}
    Q_{\xi_1}(\xi)=\frac{1}{2^5}\left(\eta^4+\gamma^4\right)=\sum_{|\beta:(4,4)|=1}A_\beta\xi^\beta
\end{equation*}
and
\begin{equation*}
    R_{\xi_1}(\xi)=\frac{1}{2^{11}}\left(11\eta^8-2\eta^4\gamma^4+11\gamma^8\right)=\sum_{|\beta:(4,4)|=2}B_\beta \xi^\beta
\end{equation*}
for $\xi=(\eta,\gamma)$ which is of the form \eqref{eq:SemiEllipticImaginaryExpansion} with $\mathbf{m}=(2,2)$ and $k=2>1$. It is clear that $Q_{\xi_1}=\abs{Q_{\xi_1}}$ and $R_{\xi_1}$ are positive definite and so Proposition \ref{prop:ExpandGamma} guarantees that $\xi_1$ is a point of imaginary homogeneous type for $\widehat{\phi}$ with drift $\alpha_{\xi_1}=(0,0)$, polynomial $Q_{\xi_1}$, and homogeneous order $\mu_{\xi_1}=1/4+1/4=1/2<1$. At $\zeta_1=(\pi,\pi)$, we have $\widehat{\phi}(\zeta_1)=i$ and
\begin{equation*}
    \Gamma_{\zeta_1}(\xi)=-i\left(Q_{\zeta_1}(\xi)+\sum_{|\beta:(4,4)|\geq 3/2}A_\beta\xi^\beta\right)-\left(R_{\zeta_1}(\xi)+\sum_{|\beta:(4,4)|\geq 3/2}B_\beta\xi^\beta\right)
\end{equation*}
where
\begin{equation*}
    Q_{\zeta_1}(\xi)=\frac{3}{8}(\eta^4+\gamma^4)=\sum_{|\beta:(4,4)|=1}A_\beta\xi^\beta
\end{equation*}
\begin{equation*}
    R_{\zeta_1}(\xi)=\frac{1}{2}(\eta^4+\gamma^4)=\sum_{|\beta:(4,4)=1}B_\beta\xi^\beta
\end{equation*}
for $\xi\in(\eta,\gamma)\in\mathcal{U}$ which is of the form \eqref{eq:SemiEllipticImaginaryExpansion} with $\mathbf{m}=(2,2)$ and $k=1$. By an appeal to Proposition \ref{prop:ExpandGamma}, we conclude that $\zeta_1$ is of positive homogeneous type for $\widehat{\phi}$ with drift $\alpha_{\zeta_1}=(0,0)$, polynomial
\begin{equation*}
P_{\zeta_1}(\xi)=iQ_{\zeta_1}(\xi)+R_{\zeta_1}(\xi)=\left(\frac{1}{2}+\frac{3}{8}i\right)\left(\eta^4+\gamma^4\right),
\end{equation*}
and homogeneous order $\mu_{\zeta_1}=1/4+1/4=1/2$. With this, we see that the hypotheses of Theorem \ref{thm:LLTMain} are met with $\mu_\phi=1/2$ and, given that $\xi_1\in\Omega_i(\phi)$ and $\zeta_1\in\Omega_p(\phi)$ both have the same homogeneous order of $1/2$, \ref{item:LLTMain2a} is applicable and we find that, for any compact set $K\subseteq\mathbb{R}^d$, there exists a sequence of compact sets $\{K_n\}$ all containing $K$ and whose union is $\mathbb{R}^2$ for which
\begin{eqnarray*}
    \phi^{(n)}(x,y)&=&\widehat{\phi}(\xi_1)^n e^{-ix\cdot\xi_1}H_{iQ_{\xi_1}}^n(x,y)+\widehat{\phi}(\zeta_1)^ne^{-ix\cdot\xi}H_{P_{\zeta_1}}^n(x,y)+o(n^{-1/2})\\
    &=&H_{iQ_{\xi_1}}^n(x,y)+i^n e^{-i\pi(x+y)}H_{P_{\zeta_1}}^n(x,y)+o(n^{-1/2})\\
    &=&H_{iQ_{\xi_1}}^n(x,y)+i^{n+2(x+y)}H_{P_{\zeta_1}}^n(x,y)+o(n^{-1/2})
\end{eqnarray*}
uniformly for $(x,y)\in K_n\cap \mathbb{Z}^2$. Upon setting
\begin{equation*}
    A_n(x,y)=H_{iQ_{\xi_1}}(x,y)+i^{n+2(x+y)}H_{P_{\zeta_1}}^n(x,y)
\end{equation*}
for $(x,y)\in\mathbb{R}^2$ so that
\begin{equation*}
    \phi^{(n)}(x,y)=A_n(x,y)+o(n^{-1/2})
\end{equation*}
uniformly for $(x,y)\in K_n\cap\mathbb{Z}^2$, this local limit theorem is illustrated in Figure \ref{fig:TwoDEx2_LLT} wherein $\Re(\phi^{(n)})$ and $\Re(A_n)$ are shown for $n=1000$ on the grid $K=[-50,50]^2\cap\mathbb{Z}^2$. To see the contribution from each attractor making up $A_n(x,y)$, Figure \ref{fig:TwoDEx2_AttractorContributions} illustrates $\Re(H_{iQ_{\xi_1}})^n$ alongside $\Re(H_{P_{\zeta_1}}^n)$ for $n=1000$.

\begin{figure}[!htb]
\begin{center}
\resizebox{\textwidth}{!}{	
	    \begin{subfigure}{0.5\textwidth}
		\includegraphics[width=\textwidth]{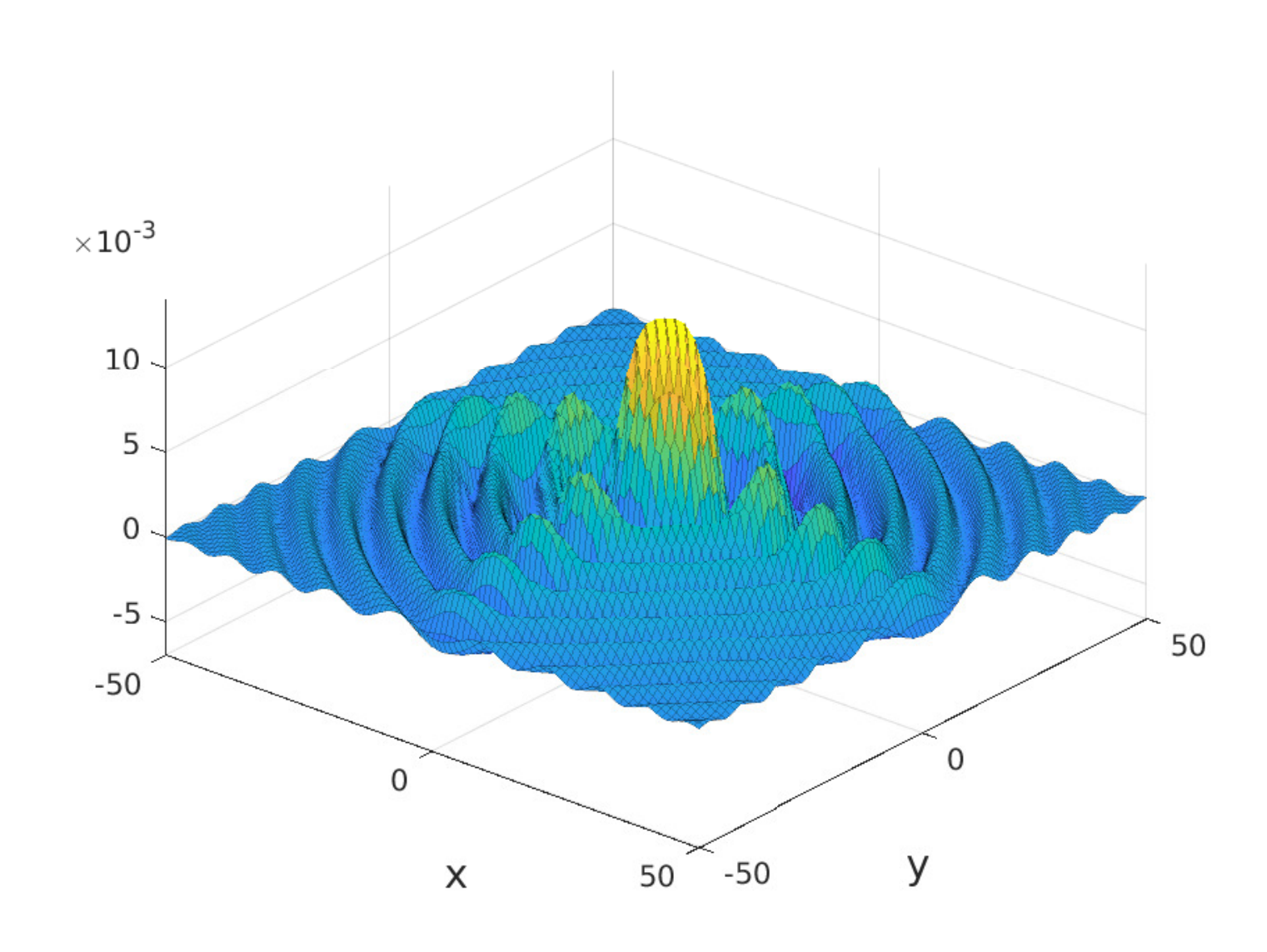}
		\caption{$\Re(\phi^{(n)})$ for $n=1000$}
		\label{fig:TwoDEx2_Phi_1000}
	    \end{subfigure}
	    \begin{subfigure}{0.5\textwidth}
		\includegraphics[width=\textwidth]{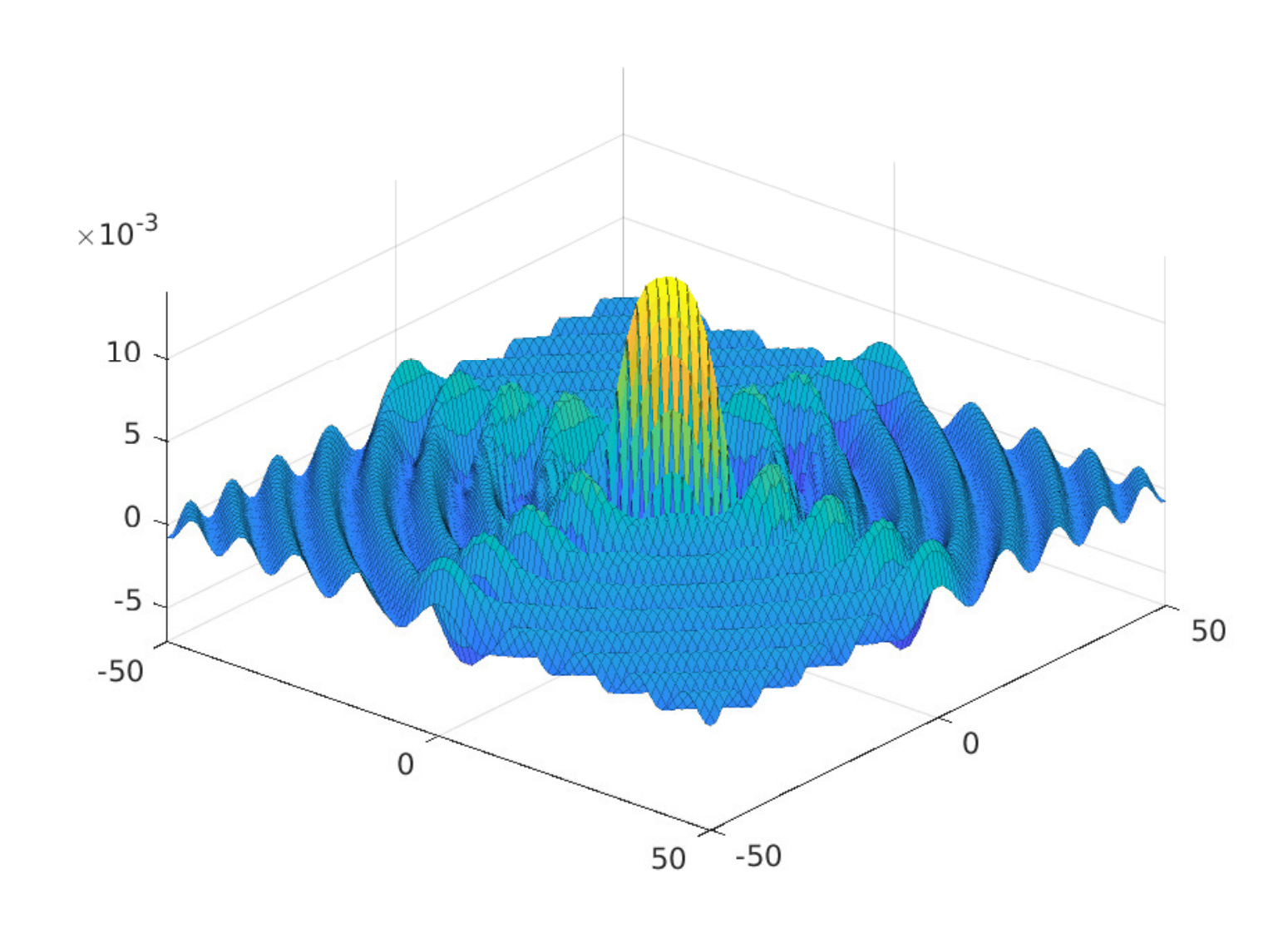}
		\caption{$\Re(A_{n})$ for $n=1000$}
		\label{fig:TwoDEx2_A_1000}
	    \end{subfigure}}
\caption{The graphs of $\Re(\phi^{(n)})$ and $\Re(A_n)$ for $n=1000$.}
\label{fig:TwoDEx2_LLT}
\end{center}
\end{figure}
\begin{figure}[!htb]
\begin{center}
\resizebox{\textwidth}{!}{	
	    \begin{subfigure}{0.5\textwidth}
		\includegraphics[width=\textwidth]{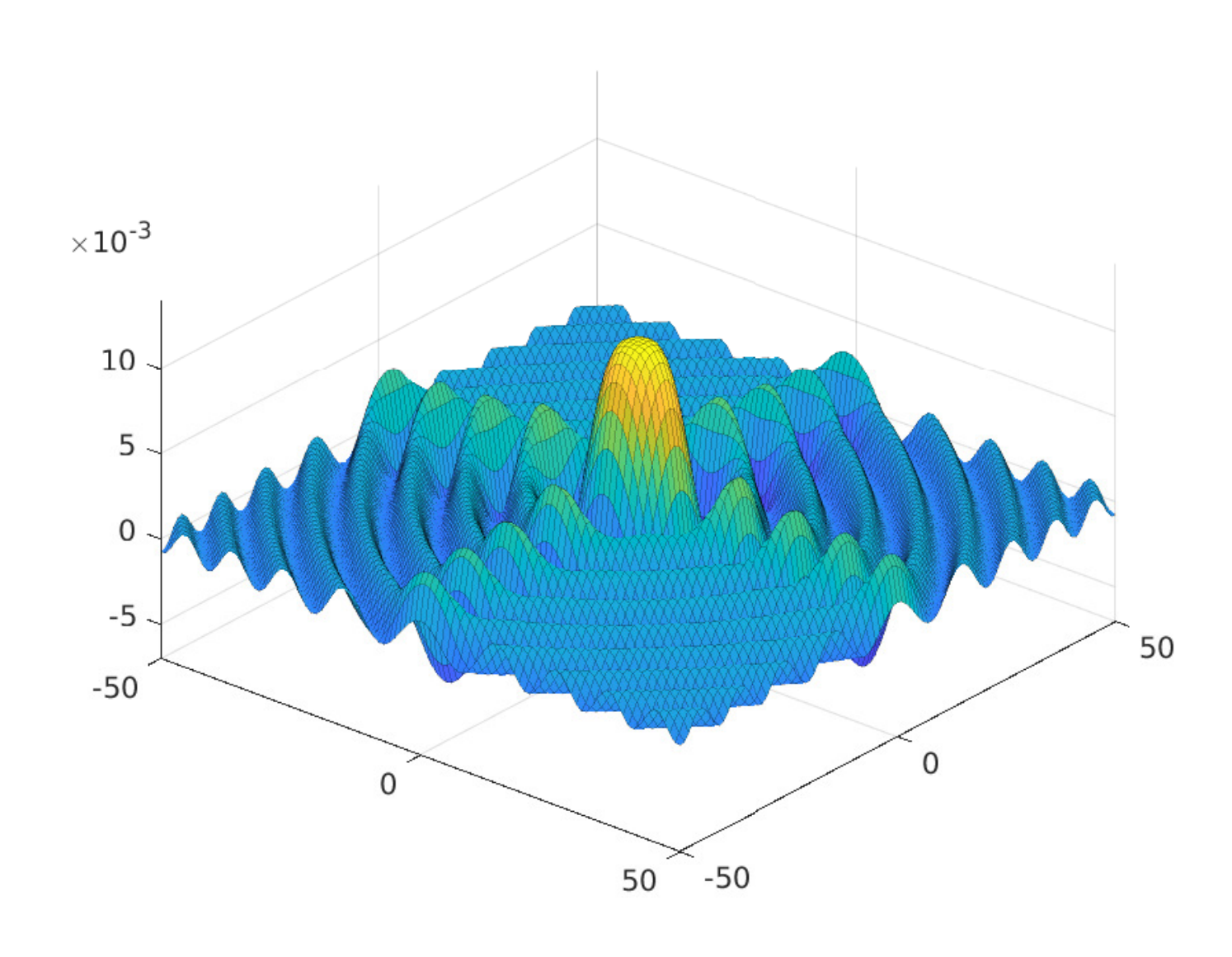}
		\caption{$\Re(H_{iQ_{\xi_1}}^n)$ for $n=1000$}
		\label{fig:TwoDEx2_Imag_Attractor}
	    \end{subfigure}
	    \begin{subfigure}{0.5\textwidth}
		\includegraphics[width=\textwidth]{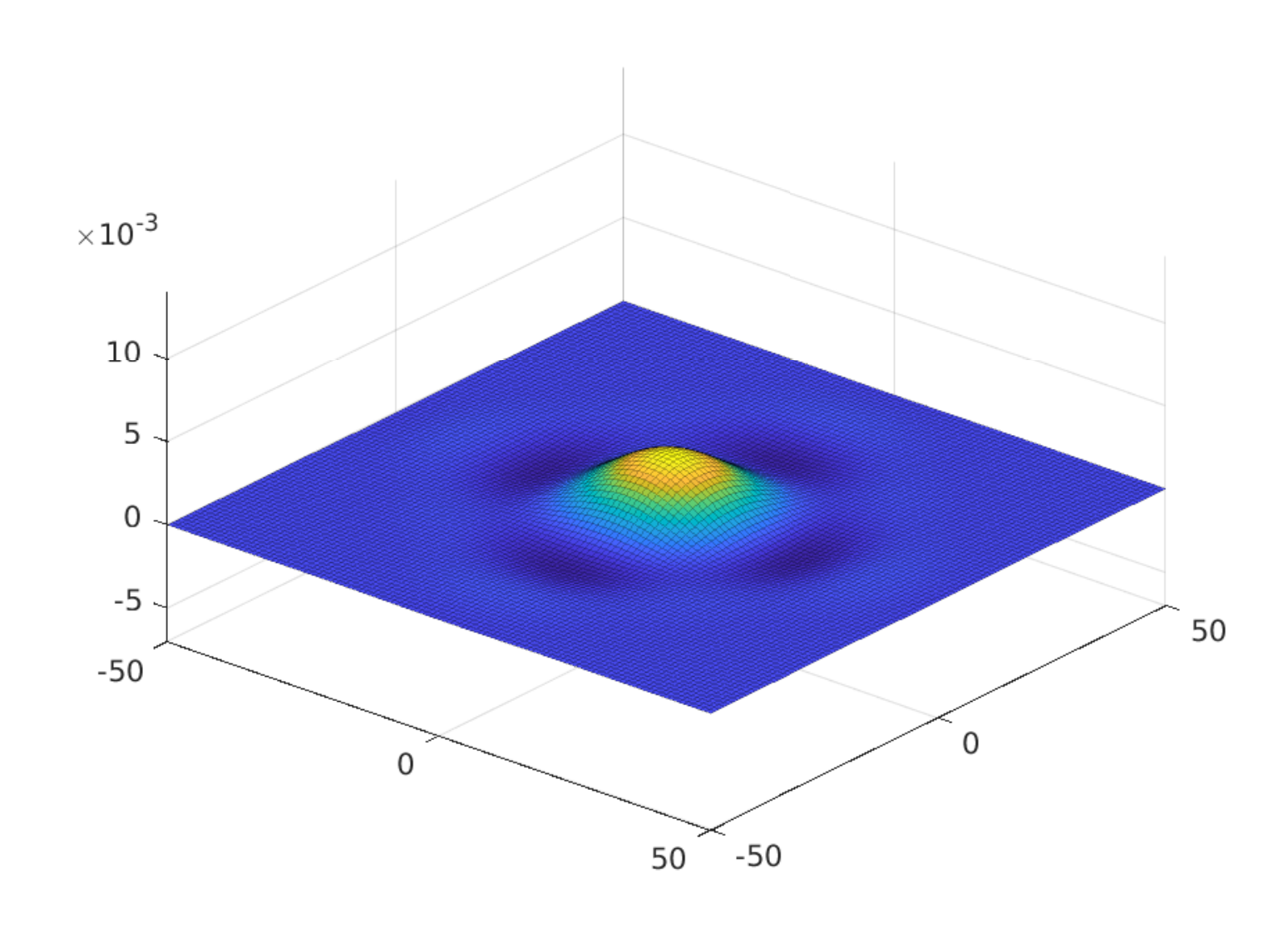}
		\caption{$\Re(i^{n+2(x+y)}H_{P_{\zeta_1}}^n)$ for $n=1000$}
		\label{fig:TwoDEx2_Pos_Attractor}
	    \end{subfigure}}
\caption{The graphs of $\Re(H_{iQ_{\xi_1}}^n)$ and $\Re(H_{P_{\zeta_1}}^n)$ for $n=1000$.}
\label{fig:TwoDEx2_AttractorContributions}
\end{center}
\end{figure}

Finally, Theorem \ref{thm:SupNormEst} gives positive constants $C,C'>0$ for which
\begin{equation*}
\abs{\phi^{(n)}(x)}\leq\frac{C'}{n^{1/2}}
\end{equation*}
whenever $x\in K_n\cap\mathbb{Z}^2$ and
\begin{equation*}
\frac{C}{n^{1/2}}\leq \|\phi^{(n)}\|_\infty
\end{equation*}
for all $n\in\mathbb{N}_+$.
\end{example}

\section{Discussion and Future Directions}\label{sec:Discussion}

This article has focused on describing the asymptotic behavior of convolution powers for a class of complex-valued functions on $\mathbb{Z}^d$ in the form of local limit theorems and identifying the attractors appearing therein. In particular, we have focused on those $\phi\in \mathcal{S}_d\subseteq \ell^1(\mathbb{Z}^d)$ whose Fourier transform $\hat{\phi}$ is maximized in absolute value at points of positive homogeneous type and imaginary homogeneous type. As the local limit theorems of \cite{RSC17} only treat the case in which $\widehat{\phi}$ is maximized at points of positive homogeneous type, the results of the present article broaden the class of functions on $\mathbb{Z}^d$ for which it is possible to prove local limit theorems. In light of the complete theory (for finitely supported functions) that is available in one dimension (established in \cite{RSC15}), there is much to be learned and many open questions remain concerning the convolution powers of complex-valued function on $\mathbb{Z}^d$ for $d>1$. In this section, we discuss a number of directions of further study. To preface this discussion, we first outline another class of functions on $\mathbb{Z}^d$ for which it is possible to prove local limit theorems. Though this class is very special/limited, it allows us to easily produce examples that go beyond the scope of the results in the present article and thus test conjectures concerning convolution powers. \\

\begin{example}\label{ex:Tensor}
For $j=1,2,\dots, d$, let $\phi_j:\mathbb{Z}\to\mathbb{C}$ have finite support consisting of more than one point and satisfy $\sup_{\xi_j}|\widehat{\phi_j}(\xi_j)|=1$. Let us assume, for simplicity that, $\Omega(\phi_j)=\{0\}$ and $\widehat{\phi_j}(0)=1$ for $j=1,2,\dots d$. By the results of \cite{RSC15}, for each $j$, there exists $\alpha_j\in\mathbb{R}$, $\beta_j\in\mathbb{C}$ with $\Re(\beta_j)\geq 0$ and an integer $m_j\geq 2$ such that
\begin{equation*}
\phi_j^{(n)}(x_j)=H_{m_j,\beta_j}^n(x_j-n\alpha_j)+o(n^{-1/m_j})
\end{equation*}
which holds uniformly for $x_j\in\mathbb{Z}$ or uniformly for $x_j\in (n\alpha_j+n^{1/m_j}K_j)\cap \mathbb{Z}$ for any compact set $K_j\subseteq\mathbb{R}$; this latter limitation applies precisely when  $m_j=2$ and $\beta_k$ is purely imaginary -- this is exactly when $H_{m_j,\beta_j}^n$ is a scaled version of the heat kernel evaluated at purely imaginary time. Consider now the function $\phi:\mathbb{Z}^d\to\mathbb{C}$ defined as the tensor product $\phi=\phi_1\otimes\phi_2\otimes\cdots\otimes\phi_d$, i.e., 
\begin{equation*}
\phi(x)=\phi_1(x_1)\phi_2(x_2)\cdots \phi_{d}(x_d)
\end{equation*}
for $x=(x_1,x_2,\dots,x_d)\in\mathbb{Z}^d$. It is easily seen that
\begin{equation*}
\phi^{(n)}(x)=\phi_1^{(n)}(x_1)\phi_2^{(n)}(x_2)\cdots\phi_{d}^{(n)}(x_d)
\end{equation*}
for each $x=(x_1,x_2,\dots,x_d)\in\mathbb{Z}^d$ and $n\in\mathbb{N}_+$. Using the one-dimensional local limit theorems above in conjunction with uniform estimates for $\phi_j^{(n)}$ and $H_{m_j,\beta_j}^n$ for each $j$ (established in \cite{RSC15}), we find that
\begin{equation*}
\phi^{(n)}(x)=H^n(x-n\alpha)+o(n^{-\mu})
\end{equation*}
where $\alpha=(\alpha_1,\alpha_2,\dots,\alpha_d)$, $\mu=1/m_1+1/m_2+\cdots+1/m_d$, and
\begin{equation*}
H^n(y)=H_{m_1,\beta_1}^n(y_1)H_{m_2,\beta_2}^n(y_2)\cdots H_{m_d,\beta_d}^n(y_d)
\end{equation*}
for $y=(y_1,y_2,\dots,y_d)\in\mathbb{R}^d$; whether or not the above local limit theorem holds uniformly on $\mathbb{Z}^d$ or just on growing and drifting semi-infinite rectangles depends on the nature the constants $m_j$ and $\beta_j$ as described above. Now, functions $\phi$ of this form do not, in general, satisfy the hypotheses of the theorems in this article. There is, however, some overlap: Suppose, additionally, that $m_j$ is even for each $j=1,2,\dots, d$. Observe that $\widehat{\phi}(\xi)=\widehat{\phi_1}(\xi_1)\widehat{\phi_2}(\xi_2)\cdots\widehat{\phi_d}(\xi_d)$ for $\xi=(\xi_1,\xi_2,\dots,\xi_d)\in\mathbb{R}^d$ and therefore $\sup_\xi |\widehat{\phi}(\xi)|=1$ and $\Omega(\phi)=\{0\}\subseteq \mathbb{T}^d$ since we have assumed that $\Omega(\phi_j)=\{0\}\in\mathbb{T}$ for each $j=1,2,\dots, d$. If $\Re(\beta)_j>0$ for all $j=1,2,\dots, d$, it is easy to verify that $0$ is a point of positive homogeneous type for $\widehat{\phi}$ with $\alpha=(\alpha_1,\alpha_2,\dots,\alpha_d)$,
\begin{equation*}
P_0(\xi)=\sum_{j=1}^d\beta_j\xi_j^{m_j}
\end{equation*}
for $\xi=(\xi_1,\xi_2,\dots,\xi_d)\in\mathbb{R}^d$, and homogeneous order $\mu=1/m_1+1/m_2+\cdots+1/m_d$. In this case, it is easy to verify that the hypotheses of Theorem \ref{thm:LLTIntro} are satisfied and the attractor $H_{P_0}$ coincides with $H$ given above. If instead, for each $j=1,2\dots, d$, $\beta_j=iq_j$ where $q_j>0$, then it is not difficult to see that $0$ is of imaginary type for $\widehat{\phi}$ with $\alpha=(\alpha_1,\alpha_2,\dots,\alpha_d)$,
\begin{equation*}
Q_0(\xi)=\sum_{j=1}^d q_j\xi_j^{m_j}
\end{equation*}
for $\xi=(\xi_1,\xi_2,\dots,x_d)$, and homogeneous order $\mu=1/m_1+1/m_2+\cdots+1/m_d$. If additionally $\mu<1$, then Theorem \ref{thm:LLTIntro} is valid and we find that, for any compact set $K\subseteq\mathbb{R}^d$,
\begin{equation*}
\phi^{(n)}(x)=H_{iQ_0}^n(x-n\alpha)+o(n^{-\mu})
\end{equation*}
for $x\in (n\alpha+n^{E}(K))\cap\mathbb{Z}^d$ where $E=E_{1/\mathbf{m}}$ has standard matrix representation $\diag(m_1^{-1},m_2^{-1},\dots,m_{d}^{-1})$. 
\end{example}

\noindent Armed with the above class of examples, let us now push forward with a general discussion. Suppose that $\phi\in\mathcal{S}_d$ has $\sup_\xi|\widehat{\phi}(\xi)|=1$. In the case that $\Omega(\phi)$ contains points of imaginary homogeneous type, the local limit theorems of the present article and those of \cite{RSC15} are limited to describing the behavior of convolution powers on growing (and drifting) compact sets, i.e., sets of the form $(n\alpha+K_n)\cap\mathbb{Z}^d$. This stands in contrast to the case in which $\Omega(\phi)$ consists only of points of positive homogeneous type wherein local limit theorems hold on all of $\mathbb{Z}^d$ (and also the one-dimensional case described by Theorem 1.2 of \cite{RSC15}). As discussed in Remark \ref{rmk:Support} and seen in Example \ref{ex:OneDEx2}, this compact set limitation is necessary in some cases. In this vein, there are two open questions that arise. First, when is this limitation of local limit theorems to sets of the form $(n\alpha+K_n)\cap \mathbb{Z}^d$ necessary? In one-dimension, Theorem 1.2 of \cite{RSC17} shows that local limit theorems hold uniformly on $\mathbb{Z}$ so as long as all points $\xi$ of Type 2 (imaginary homogeneous type) give rise to Taylor expansions for $\Gamma_\xi$ containing no quadratic term, i.e., $m>2$. Would such a statement be true in $\mathbb{Z}^d$? In terms of homogeneous orders, requiring that $\mu_\xi<1/2$ for all points $\xi\in \Omega(\phi)$ of imaginary homogeneous type would rule out the appearance of quadratic terms in the expansion for $\Gamma_\xi$ (for any $d$) and so it might be a sufficient condition for local limit theorems to hold uniformly on $\mathbb{Z}^d$. It is not, however, necessary and this can be seen by considering $\phi=\psi\otimes \psi$ where $\psi:\mathbb{Z}\to\mathbb{C}$ satisfies the hypotheses concerning the multiplicands in Example \ref{ex:Tensor} where $\Omega(\psi)=\{0\}$ and $0$ is a point of Type 2 (imaginary homogeneous type) for $\widehat{\psi}$ with $m=4$ and $\beta=iq$ for some $q>0$. By considering different possibilities for the multiplicands in Example \ref{ex:Tensor}, it becomes clear that no characterization (for a given local limit theorem to hold uniformly on $\mathbb{Z}^d$) can be made on the basis of homogeneous order alone.  The second question involves the behavior of $\phi^{(n)}$ away from the drifting compact sets of the form $(\alpha n+K_n)\cap\mathbb{Z}^d$. In the case that $\phi$ is finitely supported on $\mathbb{Z}^d$, for any $R>0$ for which $B_R=\{x\in\mathbb{R}^d:\abs{x}<R\}$ contains the support of $\phi$, $\supp(\phi)$, it is not difficult to show that $\supp(\phi^{(n)})\subseteq (nB_R)\cap \mathbb{Z}^d$ for each $n\in\mathbb{N}_+$. Thus, it remains an open question (even in one dimension) to describe the behavior of convolution of powers $\phi^{(n)}$ outside of the growing and drifting sets $(\alpha n+K_n)$ but within $nB_R$ for $n\in\mathbb{N}_+$. In one dimension, this region of interest is the union of intervals $[-nR,n\alpha-n^{1/m}r]\cup [n\alpha+n^{1/m}r, nR]$ for a fixed $r>0$. The author suspects that the one-dimensional problem is tractable in light of the methods of \cite{RSC15} and \cite{Co22}; however, a more robust analysis of the oscillatory integrals considered in this article would likely be needed to handle the case in which $d>1$. \\

\noindent Another direction of further research involves the condition $\mu<1$ appearing in Theorem \ref{thm:LLTIntro}, Theorem \ref{thm:Attractor}, and Theorem \ref{thm:LLTMain}. For a local limit theorem to be valid for $\phi$ where $\Omega(\phi)$ contains a point of imaginary homogeneous type for $\widehat{\phi}$, this condition is not necessary. To see this, let $\nu:\mathbb{Z}\to\mathbb{C}$ coincide with the function given in Example \ref{ex:OneDEx2} and consider $\phi=\nu\otimes\nu$ on $\mathbb{Z}^2$. In view of the results of Example \ref{ex:Tensor}, 
\begin{eqnarray*}
\phi^{(n)}(x_1,x_2)&=&H_{2,i/8}^n(x_1)H_{2,i/8}^n(x_2)+o(n^{-1})\\
&=&\frac{1}{ 4n i \pi/8}\exp\left(-\frac{x_1^2+x_2^2}{4ni/8}\right)+o(n^{-1})
\end{eqnarray*}
uniformly for $x=(x_1,x_2)\in n^E(K)\cap \mathbb{Z}^2$ where $E$ has standard matrix representation $\diag(1/2,1/2)$ and $K$ is a compact rectangle in $\mathbb{R}^2$. For this example, it is easy to verify that $\Omega(\phi)=\{0\}$ is a point of imaginary homogeneous type for $\widehat{\phi}$ with $Q_0(\xi)=|\xi|^2/8=(\xi_1^2+\xi_2^2)/8$ for $\xi=(\xi_1,\xi_2)\in\mathbb{R}^2$ and homogeneous order $\mu=1$. Looking back to the discussion immediately following Theorem \ref{thm:Attractor}, we see that the renormalized integral
\begin{equation*}
\fint_{\mathbb{R}^2}e^{-itQ_0(\xi)}e^{-ix\cdot\xi}\,d_\mathcal{A}\xi
\end{equation*}
does not converge for the approximating family $\mathcal{A}$ given in terms of the sets $\mathcal{O}_\tau=\{\xi:Q(\xi)<\tau\}=\{\xi:|\xi|^2<8\tau\}$. Hence, the above local limit theorem cannot be deduced directly using the results of the present article. Still, as we noted following Theorem \ref{thm:Attractor}, if the approximating family $\mathcal{A}$ were replaced by an approximating family of rectangles $\mathcal{A}'$, e.g., $\mathcal{A}'=\{[-\tau,\tau]^2:\tau>0\}$, then, thanks to the factorability of the integrand $e^{-itQ_0(\xi)}e^{-ix\cdot\xi}$ and Fubini's theorem, it is easy to see that the renormalized integral in
\begin{equation*}
H_{iQ_0}^t(x):=\frac{1}{(2\pi)^2}\fint_{\mathbb{R}^2}e^{-itQ_0(\xi)}e^{-ix\cdot\xi}\,d_{\mathcal{A}'}\xi
\end{equation*}
does converge and, further, it coincides with the attractor of the local limit theorem above, i.e., the $2$-dimensional heat kernel evaluated at purely imaginary time $it/8$. With this observation, it is likely that the renormalized integral and the machinery surrounding it can be used to identify attractors and prove local limit theorems in greater generality than is done in this article (and, in particular, relax the assumption that $\mu<1$). This would require a method of choosing approximating families to use in the renormalized integrals defining attractors, one that goes beyond the simple choice we have made in this article.\\

\noindent Concerning the attractors $H_{iQ}$ defined in Theorem \ref{thm:Attractor} and their properties, much remains to be explored. In one dimension, these attractors are known to be smooth and each satisfies a related differential equation called de Forest's Equation \cite{Gre66,DSC14}. For $d>1$, the regularity of the attractors $H_{iQ}$ is unknown. Given its structure, we suspect that $H_{iQ}$ satisfies, in some weak sense, the heat-type partial differential equation $\partial_t+iQ(D)=0$; however, we remark that this equation is not hypoelliptic. Perhaps most directly applicable to the study of convolution powers, establishing $L^\infty$-estimates for $H_{iQ}^t(\cdot)$ would go a long way toward understanding sup-norm estimates for $\phi^{(n)}$. Based on what is known in one dimension and in the positive homogeneous setting for $d>1$, it is reasonable to conjecture that $x\mapsto H_{iQ}^1(x)$ is bounded on $\mathbb{R}^d$ and from Theorem \ref{thm:Attractor} it would then follow that
\begin{equation*}
\abs{H_{iQ}^t(x)}\leq Ct^{-\mu}
\end{equation*}
for all $t>0$ and $x\in\mathbb{R}^d$ where $C=\|H_{iQ}^1\|_{L^\infty(\mathbb{R}^d)}$ and $\mu=\mu_Q$. Such an estimate and the methods used to establish it would likely be useful in extending the local limit theorems and sup-norm estimates of the present article.\\

\noindent To conclude this section, we briefly mention two recent works that extend the theory of local limit theorems of complex-valued functions on $\mathbb{Z}^d$ in distinct directions from that of the present article. The first is presented in the article \cite{Co22} of L. Coeuret and is stated in the context of one dimension. Consider a suitably normalized and finitely supported function $\phi$ on $\mathbb{Z}$ for which $\Omega(\phi)=\{\xi_1,\xi_2,\dots,\xi_l\}$ consists only of points of Type 1 (positive homogeneous type) for $\widehat{\phi}$.  As noted in the introduction, the results of \cite{DSC14}, \cite{RSC15}, and \cite{RSC17} show that the local limit theorem \eqref{eq:LLTType1OneAttractorMultiple} is valid. Under several additional assumptions concerning the drifts $\alpha_1,\alpha_2,\dots,\alpha_l$ and the values $\widehat{\phi}(\xi_k)$ for $k=1,2,\dots l$, Theorem 1 and Corollary 1 of \cite{Co22} show that the error in \eqref{eq:LLTType1OneAttractorMultiple} can be significantly improved by establishing the following local limit theorem: There exist positive constants $C$ and $M$ for which
\begin{equation}\label{eq:CoLLT}
\abs{\phi^{(n)}(x)-\sum_{k=1}^l\widehat{\phi}(\xi_k)^ne^{-ix\xi_k}H_{m_k,\beta_k}^n(x-n\alpha_k)}\leq \sum_{k=1}^l \frac{C}{n^{2/m_k}}\exp\left(-M\left|\frac{x-n\alpha_k}{n^{1/m_k}}\right|^{m_k/(m_k-1)}\right)
\end{equation} 
uniformly for $x\in\mathbb{Z}$ (see Theorem 1 and Corollary 1 of \cite{Co22}). Given the presence of the prefactors $n^{-2/m_k}$ in the upper bound, this result shows, in particular, that the uniform error of $o(n^{-1/m})$ in \eqref{eq:LLTType1OneAttractorMultiple} can be replaced by $O(n^{-2/m})$ when $\phi$ satisfies the hypotheses of Theorem 1 or Corollary 1 of \cite{Co22}.  Coeuret's proofs follow the method of ``spatial dynamics", first applied in this context in \cite{coulombel2022generalized}, and the method is quite different from the purely Fourier-analytic approach used in \cite{DSC14,RSC15,RSC17} (and the present article). It should be noted that the bounds in \eqref{eq:CoLLT} are paralleled by the following Gaussian-type off-diagonal estimates for the attractors: For each even integer $m\geq 2$ and $\beta\in\mathbb{C}$ with $\Re(\beta)>0$, there are positive constants $C$ and $M$ for which
\begin{equation*}
\abs{H_{m,\beta}^t(x)}\leq \frac{C}{t^{1/m}}\exp\left(-M\left|\frac{x}{t^{1/m}}\right|^{m/(m-1)}\right)
\end{equation*}
for $x\in\mathbb{R}$ and $t>0$. In fact, it is known that this estimate is sharp (see, e.g., \cite{BaD96}). In the context of $\mathbb{Z}^d$ for $d\geq 1$, the attractors $H_P$ corresponding to points of positive homogeneous type satisfy the following sharp off-diagonal estimate: There are positive constants $C$ and $M$ for which
\begin{equation*}
\abs{H_P^t(x)}\leq \frac{C}{t^{\mu}}\exp(-tMR^{\#}(x/t))
\end{equation*}
for $x\in\mathbb{R}^d$ and $t>0$; here $R^{\#}$ is the Legendre-Fenchel transform of the positive homogeneous polynomial $R=\Re P$ and $\mu=\mu_R$. This result appears as Proposition 2.6 of \cite{RSC17} (see also Theorem 8.2 of \cite{RSC20}). To see the connection between the two preceding estimates, one can easily verify that $R^{\#}(x)=|x|^{m/(m-1)}$ for $x\in\mathbb{R}$ when $P(\xi)=\beta \xi^m$ for $\xi\in\mathbb{R}$ when $m\geq 2$ is an even integer and $\Re(\beta)>0$. In view of Coeuret's one-dimensional result \eqref{eq:CoLLT}, this off-diagonal estimate, and Theorem 1.6 of \cite{RSC17} (or \ref{item:LLTMain1} of Theorem \ref{thm:LLTMain} of the present article), we ask: Given $\phi:\mathbb{Z}^d\to\mathbb{C}$ which is finitely supported, suitably normalized, and has $\Omega(\phi)=\{\xi_1,\xi_2,\dots,\xi_l\}$ consisting only of points of positive homogeneous type for $\widehat{\phi}$,  under what additional conditions on $\phi$ can we find positive constants $C$ and $M$ for which
\begin{equation*}
\abs{\phi^{(n)}(x)-\sum_{k=1}^l\widehat{\phi}(\xi_k)^ne^{-ix\cdot\xi_k}H_{P_k}^n(x-n\alpha_k)}\leq\sum_{k=1}^l\frac{C}{n^{2\mu_k}}\exp(-nMR_k^{\#}((x-n\alpha_k)/n))
\end{equation*}
uniformly for $x\in\mathbb{Z}^d$? Even a weaker result in which $2\mu_k$ is replaced by $c \mu_k$ for some $c>1$ would be interesting and useful. \\

\noindent The second new development appears in the recent article \cite{RSC22} of the present author and L. Saloff-Coste and is relevant only when $d>1$. Given a finitely supported function $\phi:\mathbb{Z}^d\to\mathbb{C}$ which is suitably normalized, suppose that $\Omega(\phi)$ consists of a single point $\xi_0\in\mathbb{T}^d$ for which
\begin{equation*}
\Gamma_{\xi_0}(\xi)=i\alpha\cdot\xi-R(\xi)+\widetilde{R}(\xi)
\end{equation*}
as $\xi\to 0$ where $\alpha\in\mathbb{R}^d$, $R(\xi)\geq 0$ is a positive-definite function and $\widetilde{R}(\xi)=o(R(\xi))$ as $\xi\to 0$. In contrast to \cite{RSC17} and the present article, the function $R$ is not homogeneous in any reasonable sense nor does it contain any lower-order part which is positive homogeneous. Given this lack of homogeneity, the class of examples (though special and limited) considered in \cite{RSC22} falls outside of the theory established in the present article (and all other known results for convolution powers on $\mathbb{Z}^d$). For such a $\phi$, Theorem 5.1 establishes a local limit theorem of the form
\begin{equation*}
\phi^{(n)}(x)=\widehat{\phi}(\xi_0)^ne^{-ix\cdot\xi_0}H_R^n(x-n\alpha)+o(n^{-\mu_R})
\end{equation*} 
uniformly for $x\in\mathbb{Z}^d$ where the constant $\mu_R>0$ arises not as the homogeneous order of $R$ (as no such order exists) but through an argument developed to deduce the on-diagonal asymptotic, 
\begin{equation*}
\sup_x\abs{H_R^t(x)}=H_R^t(0)\sim Ct^{-\mu_R}
\end{equation*}
as  $t\to \infty$ (see Theorem 3.1 of \cite{RSC22}). At present, little is known (away from $x=0$) of the attractor $H_R$.

\vspace{1cm}
\noindent\textbf{\large Acknowledgment:} I would like to extend my sincere thanks to Huan Bui who wrote the majority of the software used to illustrate the examples throughout, created several of the examples (including those of our previous joint work \cite{BR21}), and provided many invaluable discussions over the course of this work. I am indebted to Professors Leonard Gross, Leo Livshits, and Laurent Saloff-Coste who provided encouragement and helpful feedback (even at formative stages) for this work. I would also like to thank the anonymous referee for making a number of useful comments and asking several interesting questions that led to the inclusion of Section \ref{sec:Discussion}. As I wrote the majority of this article at Cornell University during my sabbatical leave from Colby College, I would like to thank both Colby College and Cornell University for supporting me and Cornell University for hosting me during the academic year $2021$-$2022$. Finally, I am endlessly grateful to Darby Beaulieu for her support and encouragement.

\appendix

\section{Appendix}

\subsection{Contracting Groups}\label{subsec:OneParameterGroups}

\noindent As discussed in the introduction, we shall denote by $\End(\mathbb{R}^d)$ the ring of (linear) endomorphisms of $\mathbb{R}^d$ which we take to be equipped with the operator norm $\|\cdot\|$ (inherited from the usual Euclidean norm $\abs{\cdot}$ on $\mathbb{R}^d$). For a given $A\in\End(\mathbb{R}^d)$, we shall denote by $\det(A)$, $\tr A$ and $A^*$, its determinant, trace, and adjoint/transpose, respectively. The associated general linear group will be denoted by $\Gl(\mathbb{R}^d)$ and its identity element by $I$. Given $E\in\End(\mathbb{R}^d)$, we define
\begin{equation*}
T_r=r^E=\exp(\log(r)E)=\sum_{k=0}^\infty \frac{(\log(r))^k}{k!}E^k
\end{equation*}
for $r>0$. The following amasses some basic facts about $T_r=r^E$; proofs can be found in the references \cite{Haz01,Hall03,engel_one-parameter_2000}.

\begin{proposition}\label{prop:BasicGroupProp}
For $E,F\in\End(\mathbb{R}^d)$ and $A\in\Gl(\mathbb{R}^d)$, the following properties hold: 
\begin{multicols}{2}
\begin{enumerate}
\item For every $r>0$, $r^E\in\Gl(\mathbb{R}^d)$.
\item For every $r>0$, $(r^E)^*=r^{E^*}$.
\item For every $r>0$, $\det(r^E)=r^{\tr E}$.\\
\item For every $r>0$, $A^{-1}r^EA=r^{A^{-1}EA}$.
\item\label{item:BasicGroupPropNormBound} For every $r\geq 1$, $\|r^E\|\leq r^{\|E\|}$.
\item If $EF=FE$, then $r^Er^F=r^{E+F}$\\ for every $r>0$.
\end{enumerate}
\end{multicols}
\noindent Finally, the map $(0,\infty)\ni r\mapsto r^E\in \Gl(\mathbb{R}^d)$ is a Lie group homomorphism. In particular, it is continuous and satisfies:
\begin{enumerate}
\item  $1^E=I$
\item For each $r>0$, $r^{-E}=(1/r)^E=(r^E)^{-1}$.
\item For each $r,s>0$, $r^Es^E=(ts)^E$.
\end{enumerate}
\end{proposition}

\noindent As guaranteed by the preceding proposition, for each $E\in\End(\mathbb{R})$, $\{r^E\}$ is a subgroup of $\Gl(\mathbb{R}^d)$ which we commonly refer to as a one-parameter group. As stated in the introduction, the one-parameter group $\{r^E\}$ is said to be \textbf{contracting} if
\begin{equation*}
\lim_{r\to 0}\|r^E\|=0.
\end{equation*}
The contracting property for $\{r^E\}$ is easily seen to be equivalent to Lyapunov stability of the (additive) one-parameter group $\mathbb{R}\ni t\mapsto e^{tE}$ (see \cite{engel_one-parameter_2000}).  We refer the reader to Appendix A of \cite{BR21} which contains many results on one-parameter contracting groups, many of which are used in this article.  In particular, Proposition A.2 of \cite{BR21} guarantees that $\{r^E\}$ is contracting if and only if, for each $\xi\in\mathbb{R}^d$,
\begin{equation*}
\lim_{r\to 0}|r^E\xi|=0.
\end{equation*}
For the remainder of this subsection, we focus on two results concerning the large-$r$ behavior of one-parameter contracting groups, neither of which can be found in \cite{BR21}.  By definition, contracting groups $\{r^E\}$ enjoy the property that their norms are well controlled as $r\to 0$. On the other hand, Item \ref{item:BasicGroupPropNormBound} of Proposition \ref{prop:BasicGroupProp} informs the large-$r$ behavior of $\|r^E\|$ based on the norm $\|E\|$. The following lemma gives us an estimate for $\|r^E\|$ using $E$'s trace.

\begin{lemma}\label{lem:LargeTimeContractingGroup}
Let $E\in\End(\mathbb{R}^d)$ be such that $\{r^E\}$ is a contracting group and suppose that $\tr E<1$. Then, for any $\epsilon>0$ there is $r_0\geq 1$ for which
\begin{equation*}
\|r^E\|\leq \epsilon r
\end{equation*}
for all $r\geq r_0$. In other words, $\|r^E\|=o(r)$ as $r\to\infty$.
\end{lemma}
\begin{proof}
We write $\mbox{Spec}(E)=\{\lambda_1,\lambda_2,\dots,\lambda_d\}$. 
Our assumption that $\{r^E\}$ is contracting guarantees that $\Re(\lambda_k)>0$ for all $k=1,2,\dots, d$ by virtue of the classical Lyapunov theorem (see, e.g., Theorem 2.10 of \cite{engel_one-parameter_2000}). Set $\rho=\max_{k=1,2,\dots,d}\Re(\lambda_k)$ and observe that
\begin{equation*}
0<\rho\leq \sum_{k=1}^d\Re(\lambda_k)=\Re\left(\sum_{k=1}^d\lambda_k\right)=\Re(\tr E)=\tr E<1.
\end{equation*}
Now, for $A\in\End(\mathbb{C}^d)$, we denote by $\|A\|'$ its operator norm. It is easy to see that the inclusion map $\iota:(\End(\mathbb{R}^d), \|\cdot\|)\to (\End(\mathbb{C}^d),\|\cdot\|')$ is a contraction, i.e, for all $A\in\End(\mathbb{R}^d)$, $\|A\|\leq \|A\|'$. Viewing $E$ as an element of $\End(\mathbb{C}^d)$ and making use of the Jordan-Chevelley decomposition, we write $E=D+N$ where $D\in \End(\mathbb{C}^d)$ is diagonalizable with $\Spec(D)=\Spec(E)$, $N\in\End(\mathbb{C}^d)$ is nilpotent and $DN=ND$. Because $D$ is diagonalizable, there is a constant $M\geq 1$ for which
\begin{equation*}
\|r^D\|'=\|\exp(\log(r)D)\|'\leq M \max_{\lambda\in \Spec(D)}\abs{e^{\log(r)\lambda}}\leq M\max_{\lambda\in\Spec(D)}e^{\log(r)\Re(\lambda)}=Mr^{\rho}
\end{equation*}
for $r\geq 1$ where we have used the fact that $\Spec(E)=\Spec(D)$. Thus, by virtue of the fact that $N$ and $D$ commute and $N$ is nilpotent, we have
\begin{equation*}
\|r^E\|\leq \|r^E\|'=\|r^Nr^D\|'\leq M\|r^N\|' r^\rho\leq P(\log(r)\|N\|')r^\rho
\end{equation*}
for $r\geq 1$ where $P$ is a polynomial. In view of the logarithm's slow growth and the fact that $\rho\leq \tr E<1$, the preceding inequality guarantees that, for any $\rho<\omega<1$, there is an $M'\geq 1$ for which
\begin{equation*}
\|r^E\|\leq M' r^\omega=\left(M'r^{\omega-1}\right) r
\end{equation*}
for all $r\geq 1$. With this, the desired estimate follows immediately.
\end{proof}

\noindent The following corollary follows immediately from the lemma above.

\begin{corollary}\label{cor:LargeTimeContractingGroup}
Let $E\in\End(\mathbb{R}^d)$ and, for $\alpha>0$, define $F=E/\alpha$. If $\{r^E\}$ is a contracting group, then $\{\theta^F\}_{\theta>0}$ is a contracting group. Further, if $\tr E<1$, then, for any $\epsilon>0$, there is a $\theta_0\geq 1$ for which
\begin{equation*}
\|\theta^F\|\leq \epsilon \theta^{1/\alpha}
\end{equation*}
for all $\theta\geq\theta_0$.
\end{corollary}

\subsection{Homogeneous and Subhomogeneous functions}\label{subsec:HomAndSubHom}

\noindent As mentioned in the introduction, the article \cite{BR21} develops the theory of positive homogeneous functions and related subhomogeneous functions. This development is more complete than that treated (or needed) in the present article and, for this reason, we have often appealed directly to the results of \cite{BR21}. In this short appendix, we present results concerning positive homogeneous and subhomogeneous functions which are of particular interest for our study of convolution powers.

\begin{lemma}\label{lem:AbsPosHomExponentSets}
Let $Q:\mathbb{R}^d\to\mathbb{R}$ be continuous and suppose that $\abs{Q}$ is positive homogeneous. Then $\Exp(Q)=\Exp(\abs{Q})$. 
\end{lemma}
\begin{proof} 
If $d=1$, it is easy to see that $Q$ is necessarily of the form
\begin{equation*}
Q(\xi)=\begin{cases}
Q(1)\abs{\xi}^{\alpha} & \xi\geq 0\\
Q(-1)\abs{\xi}^{\alpha} & \xi<0
\end{cases}
\end{equation*}
for $\xi\in\mathbb{R}$ where $\alpha>0$. In this case, it is easy to see that $\Exp(Q)=\Exp(\abs{Q})=\{E_{1/\alpha}\}$ where $E_{1/\alpha}\in \End(\mathbb{R})$ is the transformation taking $\xi$ to $\xi/\alpha$. For $d>1$, given that $Q$ is non-vanishing on the connected set $\mathbb{R}^d\setminus\{0\}$, $Q$ must be single signed, i.e., given any fixed non-zero $\xi_0\in\mathbb{R}^d$, $Q(\xi)=\sign(Q(\xi_0))\abs{Q(\xi)}$ for all $\xi\in\mathbb{R}^d$ and from this it follows immediately that $Q$ and $\abs{Q}$ share the same exponent set.
\end{proof}

\begin{proof}[Proof of Proposition \ref{prop:Subhomequivtolittleoh}]
In what follows, $P$ is a positive homogeneous function and $\widetilde{P}$ is a continuous complex-valued function defined on an open set $\mathcal{U}$ of $\mathbb{R}^d$.
\begin{subproof}[\ref{item:Subhomequivtolittleoh1}$\implies$\ref{item:Subhomequivtolittleoh2}]
Let $E\in\Exp(P)$ and fix $\epsilon>0$ and a compact set $K\subseteq\mathbb{R}^d$. By our supposition, let $\mathcal{O}\subseteq\mathcal{U}$ be a neighborhood of $0$ for which
\begin{equation*}
\abs{\widetilde{P}(\zeta)}\leq \frac{\epsilon}{M+1} P(\zeta)
\end{equation*}
for all $\zeta\in\mathcal{O}$ where
\begin{equation*}
M=\sup_{\xi\in K}P(\xi).
\end{equation*} Since $\{t^E\}$ is contracting thanks to Proposition \ref{prop:CharofPosHom}, there exists $\tau>0$ for which $t^E\xi\in \mathcal{O}$ for all $0<t<\tau$ and $\xi\in K$  (Proposition A.6 of \cite{BR21}). Consequently, for any $\xi\in K$ and $0<t<\tau$,
\begin{equation*}
\abs{\widetilde{P}(t^E\xi)}\leq\frac{\epsilon}{M+1}P(t^E\xi)=\epsilon t\frac{P(\xi)}{M+1}\leq \epsilon t\frac{M}{M+1}<\epsilon t.
\end{equation*}
\end{subproof}
\begin{subproof}[\ref{item:Subhomequivtolittleoh2}$\implies$\ref{item:Subhomequivtolittleoh3}]
This is immediate.
\end{subproof}
\begin{subproof}[\ref{item:Subhomequivtolittleoh3}$\implies$\ref{item:Subhomequivtolittleoh1}]
Let $\epsilon>0$. Because the unital level set $S$ of $P$ is compact (Proposition \ref{prop:CharofPosHom}), our supposition guarantees $\tau>0$ for which
\begin{equation*}
\abs{\widetilde{P}(t^E\eta)}\leq \epsilon t
\end{equation*}
for all $0<t<\tau$ and $\eta\in S$. For the open set $\mathcal{O}_\tau=\{\zeta\in\mathbb{R}^d:P(\zeta)<\tau\}$, we claim that
\begin{equation*}
\mathcal{O}_\tau\setminus\{0\}=\{t^E\eta:0<t<\tau,\eta\in S\}.
\end{equation*}
To see this, first suppose that $\zeta\in \mathcal{O}_\tau\setminus\{0\}$ or, equivalently, $0<P(\zeta)<\tau$. Then, for $t=P(\zeta) \in (0,\tau)$, observe that $\eta:=t^{-E}\zeta\in S$ because 
$P(\eta)=P(t^{-E}\zeta)=P(\zeta)/t=1$. Consequently, $\zeta=t^{E}\eta$ for $\eta\in S$ and $0<t<\tau$ and so $\mathcal{O}_\tau\setminus\{0\}\subseteq\{t^E\eta:0<t<\tau,\eta\in S\}$. Of course, for $0<t<\tau$ and $\eta\in S$, $P(t^E\eta)=tP(\eta)=t\in (0,\tau)$ and so we have justified our claim. 

With this identification, we observe that
\begin{equation*}
\abs{\widetilde{P}(\zeta)}\leq \epsilon t=\epsilon P(\zeta).
\end{equation*}
for each $\zeta=t^E\eta\in \mathcal{O}\setminus \{0\}$. By the continuity of $\widetilde{P}$, it immediately follows that $\widetilde{P}(0)=0\leq \epsilon P(0)$. Thus, we have found an open neighborhood $\mathcal{O}=\mathcal{O}_\tau$ of $0$ for which
\begin{equation*}
\abs{\widetilde{P}(\zeta)}\leq \epsilon P(\zeta)
\end{equation*}
for all $\zeta\in\mathcal{O}$ which is precisely the statement that $\widetilde{P}(\xi)=o(P(\xi))$ as $\xi\to 0$.
\end{subproof}
\end{proof}

\noindent The remainder of the section is dedicated to the proof of Proposition \ref{prop:ExpandGamma}. The proof makes use of the following lemma.

\begin{lemma}
Given an open neighborhood $\mathcal{U}$ of $0$ in $\mathbb{R}^d$, suppose that $Q:\mathcal{U}\to\mathbb{C}$ is \\
analytic on $\mathcal{U}$ with absolutely and uniformly convergent series expansion
\begin{equation*}
    Q(\xi)=\sum_{|\beta:\mathbf{n}|>1}A_\beta\xi^\beta
\end{equation*}
for some $\mathbf{n}\in\mathbb{N}_+^d$. Consider $E\in\End(\mathbb{R}^d)$ with the standard representation $\diag(1/n_1,1/n_2,\dots,1/n_d)$. Then, for each $l\in\mathbb{N}_+$, $Q$ is strongly subhomogeneous with respect to $E$ of order $l$. 
\end{lemma}
\begin{proof}
We will prove that, for each, $j\in\mathbb{N}$, $\epsilon>0$ and compact set $K\subseteq\mathbb{R}^d$, there is a $\delta>0$ for which
\begin{equation*}
    \abs{r^j\partial_r^jQ(r^E\eta)}\leq \epsilon r
\end{equation*}
for all $0<r<\delta$ and $\eta\in K$. To this end, we fix $j$, $\epsilon$, and $K$ as above and write $Q=Q_1+Q_2$ where
\begin{equation*}
Q_1(\xi)=\sum_{1+\rho\leq |\beta:\mathbf{n}|\leq 2j+4}A_\beta\xi^\beta
\hspace{0.5cm}\mbox{and}\hspace{0.5cm}
    Q_2(\xi)=\sum_{|\beta:\mathbf{n}|> 2j+4}A_{\beta}\xi^\beta
\end{equation*}
where $\rho:=\min\{|\beta:\mathbf{n}|:A_\beta\neq 0\}-1>0$; of course, $Q_1$ is identically zero provided $1+\rho>2j+4$. For each $q\geq 1$ and $l\in \mathbb{N}$, define
\begin{equation*}
    \mathcal{P}(q,l)=q(q-1)(q-2)\cdots (q-(l-1))
\end{equation*}
where we assume the convention that $\mathcal{P}(q,0)=1$. In this notation, we observe that 
\begin{equation*}
    \partial_r^j(r^E\xi)^\beta=\partial_r^j\left(r^{|\beta:\mathbf{n}|}\xi^\beta\right)=\mathcal{P}(|\beta:\mathbf{n}|,j)r^{|\beta:\mathbf{n}|-j}\xi^\beta
\end{equation*}
for $\xi\in\mathbb{R}^d$, $r>0$ and $\beta\in\mathbb{N}^d$.
Because $Q_1$ is a polynomial and $K$ is compact, we have
\begin{equation*}
    M_1:=\sup_{\eta\in K}\left(\sum_{1+\rho\leq |\beta:\mathbf{n}|\leq 2j+4}\abs{A_\beta \mathcal{P}(|\beta:\mathbf{n}|,j)\eta^\beta}\right)<\infty.
\end{equation*}
In view of our hypotheses, let $\mathcal{O}$ be an open  neighborhood of $0$ for which $\mathcal{O}\subseteq \overline{\mathcal{O}}\subseteq\mathcal{U}$ and
\begin{equation*}
    M_2:=\sup_{\xi\in \mathcal{O}}\left( \sum_{|\beta:\mathbf{n}|>2j+4}\abs{A_\beta\xi^\beta}\right)<\infty.
\end{equation*}
We now specify $\delta$. First, given that $\{r^E\}$ and $\{r^{E/4}\}$ are contracting and the set $K$ is compact, there exists $0<\delta_1$ for which $r^E\eta$ and $r^{E/4}\eta$ belong to $\mathcal{O}$ whenever $0<r<\delta_1$ and $\eta\in K$. Also, there exists $0<\delta_2\leq 1$ for which
\begin{equation}\label{eq:PermuationEst}
    \abs{\mathcal{P}(q,j)}r^{q/4}\leq 1
\end{equation}
for all $q>j$ and $0<r\leq \delta_2$; in fact, one may take $\delta_2=e^{-4j}$.  Finally, given that $\rho>0$, let $\delta_3>0$ be such that
\begin{equation*}
    M_1 r^\rho+M_2r<\epsilon
\end{equation*}
for all $0<r<\delta_3$. Set $\delta=\min\{\delta_1,\delta_2,\delta_3\}$ and observe that
for all $\eta\in K$ and $0<r<\delta$, we have
\begin{eqnarray*}
    \abs{r^j\partial_r^jQ_1(r^E\eta)}&=&r^j\abs{\sum_{1+\rho\leq|\beta:\mathbf{n}|\leq 2j+4}A_\beta \partial_r^j\left( r^E\eta\right)^{\beta}}\\
    &\leq&r^j\sum_{1+\rho\leq|\beta:\mathbf{n}|\leq 2j+4}\abs{A_\beta\mathcal{P}(|\beta:\mathbf{n}|,j)r^{|\beta:\mathbf{n}|-j}\eta^\beta}\\
    &\leq&r^{1+\rho}\sum_{1+\rho\leq |\beta:\mathbf{n}|\leq 2j+4}\abs{A_\beta \mathcal{P}(|\beta:\mathbf{n}|,j)\eta^\beta}\\
&\leq&r  M_1r^\rho
\end{eqnarray*}
where we have used the fact that $\delta\leq 1$. By virtue of \eqref{eq:PermuationEst}, observe that, for each $q=|\beta:\mathbf{n}|>2j+4$, $|\beta:\mathbf{n}|/2-j>2$ and so
\begin{eqnarray*}
    \abs{\partial_r^j\left( A_\beta(r^E\eta)^\beta\right)}&=&\abs{A_\beta}\abs{\mathcal{P}(|\beta:\mathbf{n}|,j)}r^{|\beta:\mathbf{n}|-j}\abs{\eta^\beta}\\
    &=&r^{|\beta:\mathbf{n}|/2-j}\abs{A_\beta}\abs{\mathcal{P}(|\beta:\mathbf{n}|,j)r^{|\beta:\mathbf{n}|/4}}\abs{(r^{E/4}\eta)^\beta}\\
    &\leq& r^2\abs{A_\beta(r^{E/4}\eta)^\beta}
\end{eqnarray*}
for all $0<r<\delta\leq 1$ and $\eta\in K$. It follows that
\begin{eqnarray*}
    \abs{\partial_r^jQ_2(r^E\eta)}&=&\abs{\sum_{|\beta:\mathbf{n}|> 2j+4}\partial_r^j\left( A_\beta(r^E\eta)^\beta\right)}\\
    &\leq & \sum_{|\beta:\mathbf{n}|>2j+4}\abs{\partial_r^j\left( A_\beta(r^E\eta)^\beta\right)}\\
    &\leq&\sum_{|\beta:\mathbf{n}|>2j+4}r^2\abs{A_\beta(r^{E/4}\eta)^\beta}\\
    &\leq &r^2M_2
\end{eqnarray*}
for all $0<r<\delta$ and $\eta\in K$. Therefore, for each $0<r<\delta$ and $\eta\in K$, we have
\begin{eqnarray*}
    \abs{r^j\partial_r^jQ(r^E\eta)}&\leq&\abs{r^j\partial_r^jQ_1(r^E\eta)}+\abs{r^j\partial_r^jQ_2(r^E\eta)}\\
    &\leq& rr^\rho M_1+r^{j+2}M_2\\
    &\leq& r(M_1r^\rho+M_2r)\\
    &<&r \epsilon,
\end{eqnarray*}
as desired.
\end{proof}

\begin{proof}[Proof of Proposition \ref{prop:ExpandGamma}.]
It is easy to see that $E\in \Exp(Q)\cap \Exp(\abs{Q})$ and $E/k\in\Exp(R)$ for $E\in\End(\mathbb{R}^d)$ with standard matrix representation 
\begin{equation*}
\diag((2m_1)^{-1}, (2m_2)^{-1},\dots, (2m_d)^{-1}).
\end{equation*}
If $k=1$, $R$ is positive homogeneous with $E\in\Exp(R)\cap\Exp(Q)$. By virtue of the preceding lemma (with $\mathbf{n}=2\mathbf{m}$), $\widetilde{Q}$ and $\widetilde{R}$ are strongly subhomogeneous with respect to $E$ of order $1$ and so necessarily subhomogeneous with respect to $E$. In this case, we may conclude that $\xi_0$ is of positive homogeneous type for $\widehat{\phi}$ with drift $\alpha$ and homogeneous order
\begin{equation*}
    \mu_{\xi_0}=\tr E=|\mathbf{1}:2\mathbf{m}|=\sum_{j=1}^d\frac{1}{2m_j}.
\end{equation*}
If $k>1$, our supposition guarantees that $\abs{Q}$ is positive homogeneous with respect to $E$ and $R$ is positive homogeneous with respect to $E/k$. By virtue of the preceding lemma, $\widetilde{Q}$ is strongly subhomogeneous with respect to $E$ of order $2$ and $\widetilde{R}$ is strongly subhomogeneous with respect to $E/k$ of order $1$. Consequently, $\xi_0$ is of imaginary homogeneous type for $\widehat{\phi}$ with drift $\alpha$ and homogeneous order $\mu_{\xi_0}=\tr E=|\mathbf{1}:2\mathbf{m}|$ as in the previous case.
\end{proof}

\vspace{1cm}
\noindent Evan Randles: Department of Mathematics, Colby College, Waterville, ME 04901, USA.\\

\noindent E-mail: \href{evan.randles@colby.edu}{evan.randles@colby.edu}

\end{document}